\documentclass[english]{article}
\usepackage{lmodern}

\usepackage[T1]{fontenc}
\usepackage[latin9]{luainputenc}
\usepackage{geometry}
\usepackage{color}
\usepackage{babel}
\usepackage{refstyle}
\usepackage{mathrsfs}
\usepackage{amsmath}
\usepackage{amsthm}
\usepackage{amssymb}
\usepackage{stackrel}
\usepackage[pdfusetitle,
 bookmarks=true,bookmarksnumbered=false,bookmarksopen=false,
 breaklinks=false,pdfborder={0 0 1},backref=false,colorlinks=true]
 {hyperref}
\hypersetup{
 pdfborderstyle=,linkcolor=black,citecolor=black,urlcolor=blue,filecolor=blue,pdfpagelayout=OneColumn,pdfnewwindow=true,pdfstartview=XYZ,plainpages=false}

\makeatletter

\AtBeginDocument{\providecommand\secref[1]{\ref{sec:#1}}}
\AtBeginDocument{\providecommand\Secref[1]{\ref{Sec:#1}}}
\AtBeginDocument{\providecommand\thmref[1]{\ref{thm:#1}}}
\AtBeginDocument{\providecommand\lemref[1]{\ref{lem:#1}}}
\RS@ifundefined{subsecref}
  {\newref{subsec}{name = \RSsectxt}}
  {}
\RS@ifundefined{thmref}
  {\def\RSthmtxt{theorem~}\newref{thm}{name = \RSthmtxt}}
  {}
\RS@ifundefined{lemref}
  {\def\RSlemtxt{lemma~}\newref{lem}{name = \RSlemtxt}}
  {}

\theoremstyle{plain}
\newtheorem*{thm*}{\protect\theoremname}
\theoremstyle{plain}
\newtheorem{thm}{\protect\theoremname}[section]
\theoremstyle{definition}
\newtheorem{defn}[thm]{\protect\definitionname}
\theoremstyle{plain}
\newtheorem{lem}[thm]{\protect\lemmaname}
\theoremstyle{remark}
\newtheorem{rem}[thm]{\protect\remarkname}

\usepackage{babel}
\AtBeginDocument{\providecommand\Secref[1]{\ref{Sec:#1}}}\AtBeginDocument{\providecommand\defref[1]{\ref{def:#1}}}\newref{def}{name=definition~,Name=Definition~,names=definitions~,Names=Definitions~}

\newref{thm}{name=theorem~,Name=Theorem~,names=theorems~,Names=Theorems~}

\newref{cor}{name=corollary~,Name=Corollary~,names=Corollaries~,Names=Corollaries~}

\newref{lem}{name=lemma~,Name=Lemma~,names=lemmas~,Names=Lemmas~}

\newref{rem}{name=remark~,Name=Remark~,names=remarks~,Names=Remarks~}

\newref{conj}{name=conjecture~,Name=Conjecture~,names=conjectures~,Names=Conjectures~}

\providecommand{\definitionname}{Definition}
\providecommand{\lemmaname}{Lemma}
\providecommand{\remarkname}{Remark}
\providecommand{\theoremname}{Theorem}

\makeatother

\providecommand{\definitionname}{Definition}
\providecommand{\lemmaname}{Lemma}
\providecommand{\remarkname}{Remark}
\providecommand{\theoremname}{Theorem}

\begin{document}
\title{Sentences over Random Groups II: Sentences of Minimal Rank}
\author{Sobhi Massalha}
\maketitle
\begin{abstract}
Random groups of density $d<\frac{1}{2}$ are infinite hyperbolic,
and of density $d>\frac{1}{2}$ are finite. We prove the existence
of a uniform quantifier elimination procedure for formulas of minimal
rank (probably the superstable part of the theory). Namely, given
a minimal rank formula $V(p)$, we prove the existence of a formula
$\varphi(p)$ that belongs to the Boolean algebra of two quantifiers,
so that the two formulas $V(p)$ and $\varphi(p)$ define the same
set over the free group $F_{k}$ and over a random group of density
$d<\frac{1}{2}$. We conclude that any given sentence of minimal rank
is a truth sentence over the free group $F_{k}$ if and only if it
is a truth sentence over random groups of density $d<\frac{1}{2}$. 
\end{abstract}
\tableofcontents{}

\section{Introduction}\label{sec:Introduction}

Around 1945, a well-known question was presented by Alfred Tarski
on the first order theory of free groups. He asked if every two non-abelian
finitely generated free groups are elementary equivalent. His question
was answered affirmatively by Z. Sela through his seminal work, which
was published in a series consisting of seven papers titled ``Diophantine
Geometry'' (\cite{DGI,DGII,DGIII,DGIV,DGV,DGV2,DGVI}). For a different
approach to Tarski's problem, see also \cite{Elementary=000020theory=000020of=000020free=000020non-abelian=000020groups}.

Actually, Sela obtained a classification of all the f.g. (finitley
generated) groups that are elementary equivalent to a non-abelian
f.g. free group. In \cite{DGVII}, Sela generalized his work to a
general torsion-free hyperbolic group.

Hence, all the non-abelian f.g. free groups share the same collection
of truth sentences. This collection is called the theory of free groups.
In this paper, we will be interested in the groups that can or cannot
be distinguished from the free groups by a single sentence.

In the previous paper in this sequence, we have proved that given
a first order sentence $\psi$ in the Boolean algebra of universal
sentences, the sentence $\psi$ is a truth sentence over non-abelian
free groups if and only if $\psi$ is a truth sentence over a random
group with overwhelming probability in the (Gromov) density model
with density $d<\frac{1}{2}$.

In this paper, we consider general sentences of minimal rank. Our
aim in this paper is proving that every given minimal rank sentence
is a truth sentence over free groups if and only if it is a truth
sentence over a random group of density $d<\frac{1}{2}$. 
\begin{thm*}
(Main Theorem). Let $\psi$ be a first order sentence of minimal rank.
Let $\Gamma$ be a random group of density $d<\frac{1}{2}$. Then
$\psi$ is a truth sentence over the free group $F_{k}$ if and only
if $\psi$ is a truth sentence over $\Gamma$ (in overwhelming probability). 
\end{thm*}
In \cite{DGV,DGV2}, Z. Sela obtained a quantifier elimination procedure
for general formulas over free groups. He proved that every definable
set belongs to the Boolean algebra of $AE$-sets (All-Exist) over
free groups. The tools needed for that quantifier elimination procedure
are presented in \cite{DGI,DGII,DGIII,DGIV}.

For our aim in this paper, we prove the existence of a uniform quantifier
elimination for formulas of minimal rank. Namely, for a given minimal
rank formula $V(p)$, we prove the existence of a formula $\varphi(p)$
that belongs to the Boolean algebra of $AE$-formulas, so that the
two formulas $V(p)$ and $\varphi(p)$ are equivalent (i.e., define
the same set) over the free group $F_{k}$ (over which we define the
density model; $k$ is the rank of the free group $F_{k}$), and are
equivalent over almost all the groups of density $d<\frac{1}{2}$. 
\begin{thm*}
Let $0\leq d<\frac{1}{2}$, and let $V(w,p)$ be a minimal rank formula
in the Boolean algebra of $AE$-formulas. Let $\pi V(p)$ be the projection
of $V(w,p)$. There exists a formula $R(p)$ in the Boolean algebra
of $AE$-formulas so that: 
\begin{enumerate}
\item The formulas $\pi V(p)$ and $R(p)$ define the same set over $F_{k}$. 
\item Let $\Gamma$ be a random group of density $d$. With overwhelming
probability, the formulas $\pi V(p)$ and $R(p)$ define the same
set over $\Gamma$. 
\end{enumerate}
\end{thm*}
The uniform quantifier elimination procedure that we present in this
paper is essentially the quantifier elimination procedure for formulas
over free groups (presented in \cite{DGV}). In other words, the quantifier
elimination procedure (for minimal rank formulas) presented in \cite{DGV}
is valid over almost all the groups of density $d<\frac{1}{2}$, for
every minimal rank formula given in advance.

The paper is organized as follows.

In \secref{Limit-Groups-over-Sequences-of-Random-Groups}, we prove
that every group obtained as a limit of a convergent sequence over
an ascending (w.r.t. the levels of the model) sequence of random groups
is a limit group (over free groups). Due to the descending chain condition
of limit groups, this observation guarantees that the iterative procedures
that we will apply over random groups along the next sections must
terminate in a finite time, since these procedures run over ascending
sequences of random groups, rather than over a single one.

In \secref{Graded-Test-Sequences-over-Random-Groups}, we define graded
test sequences over ascending sequences of random groups. The main
importance of using test sequences is to preserve the structure of
the resolutions that we use, during the constructions explained in
the next sections.

\Secref{Rigid-and-Solid-Limit-Groups} is devoted to rigid and solid
graded limit groups. In this section, we recall the definition of
graded formal Makanin-Razborov diagrams, as well as the definitions
of rigid and solid graded limit groups (introduced in \cite{DGI}).
We extend the definition of exceptional families of solutions for
rigid and solid limit groups over an arbitrary group. In addition,
we prove the existence of a uniform global bound on the number of
distinct exceptional families of a given rigid or solid limit group,
over a random group (with overwhelming probability).

The approach to quantifier elimination that we adopt, requires the
development of some algebraic objects in advance. Sections 5-7 are
devoted to constructing uniform objects, which will be applied in
the next sections in order to obtain a uniform quantifier elimination
procedure over $F_{k}$ and over almost all the groups of density
$d<\frac{1}{2}$.

In particular, in the classical approach to quantifier elimination
for positive formulas over free groups, the notion of formal solutions
was used. This notion was generalized in \cite{DGII}, in order to
prove the existence of formal solutions over Diophantine sets (resolutions)
of specializations over free groups. Moreover, for obtaining a quantifier
elimination for general formulas over free groups (as presented in
\cite{DGV,DGV2}), it was necessary to encode the whole collection
of formal solutions of every given graded system of equations, in
a single Makanin-Razborov diagram, called the graded formal MR diagram
of that given graded system.

In \secref{Formal-Solutions-over-Random-Groups}, we prove that the
graded formal Makanin-Razborov diagram of a given graded system of
equations, constructed over the fixed free group $F_{k}$ (\cite{DGII}),
encodes also the collections of formal solutions over almost all the
groups of density $d<\frac{1}{2}$.

In \secref{Proof-Systems}, we explain the construction of proof systems
for a given $EAE$-formula ($\exists\,\forall\,\exists$) of minimal
rank. The construction of proof systems (tree of stratified sets)
was introduced in \cite{DGV} for $EAE$-formulas over free groups.

Given an $EAE$-formula $T(p)$ of minimal rank, we construct finitely
many proof systems that uniformly validate the truth of the sentence
$T(p_{0})$, for each value $p_{0}$ in the set defined by the formula
$T(p)$ over the free group $F_{k}$, and for each value $p_{0}$
in the set defined by the formula $T(p)$ over a random group $\Gamma$,
for almost all the groups $\Gamma$ of density $d<\frac{1}{2}$.

This validation is interpreted in terms of positive and non-positive
Diophantine conditions, together with the requirement of the existence
of a full families of exceptional solutions for some (finite) set
of solid and rigid limit groups, specified by the (finitely many)
proof systems, so that these exceptional solutions satisfy some diophantine
conditions (determined also by the proof systems).

Of course, once a quantifier elimination for positive formulas over
free groups is given, the substantial difficulty in obtaining a quantifier
elimination for general formulas over free groups is due to the existence
of inequalities.

The collection of the proof systems that we associate with a given
$EAE$-formula (of minimal rank) can be thought of as an attempt to
get rid of the inequalities, by collecting them in a formula which
belongs to the Boolean algebra of universal formulas. Unfortunately,
this cannot be achieved in general. But in \secref{An-Approximation-to-Abstract-EAE},
we will prove that the finitely many proof systems associated with
an $EAE$-formula (of minimal rank) can be used in order to construct
a formula in the Boolean algebra of $AE$-formulas that is equivalent
(define the same set) to the given $EAE$-formula, over $F_{k}$ and
over almost all the groups of density $d<\frac{1}{2}$.

Consequently, we conclude the existence of a uniform quantifier elimination
procedure for every given formula of minimal rank, over $F_{k}$ and
over almost all the groups of density $d<\frac{1}{2}$. The method
explained in \secref{An-Approximation-to-Abstract-EAE}, introduced
originally in \cite{DGV,DGV2} for $EAE$-formulas over free groups,
is called an approximation of an $EAE$-set.

Thus, the construction of proof systems for a given formula should
conceptually precede the approximation of that formula. However, in
contrast to \cite{DGV,DGV2}, in this paper we choose to reverse the
order of presentation, because we explain the approximation method
from an abstract point of view.

In \secref{Truth-MR-Sentences-over-Random-Groups}, we conclude this
work by proving that every given sentence of minimal rank, is a truth
sentence over the free group $F_{k}$ if and only if it is a truth
sentence over a random group of density $d<\frac{1}{2}$ (with overwhelming
probability).\\
 \\

\textbf{Acknowledgment.} I am indebted to my advisor Zlil Sela who
introduced me to this problem. He shared his knowledge and ideas with
me, and without his input, I couldn't have completed this work.

\newpage{}

\section{Limit Groups over Sequences of Random Groups}\label{sec:Limit-Groups-over-Sequences-of-Random-Groups}

Along all the sections of this paper, we fix an integer $k\geq2$,
a free group $F_{k}$, and a basis $a=(a_{1},...,a_{k})$ for $F_{k}$.
We also fix a real number $0\leq d<\frac{1}{2}$, and consider the
Gromov density model of density $d$. For the definition of Gromov
density model, see section 2 in the previous paper in this series.

According to theorem 10.1 in the previous paper in this series, we
know that given a (finite) system of equations $\Sigma(x,a)$, maybe
with constants (in $F_{k}$), the probability that the random group
$\Gamma$ satisfies the property that every solution of $\Sigma(x,a)$
in $\Gamma$ can be lifted to a solution of $\Sigma(x,a)$ in the
free group $F_{k}$ equals one. We use this property along all the
coming chapters, and for brevity we will refer to it by saying that
$\Gamma$ satisfies the $\Sigma(x,a)$-l.p. (lifting property). A
direct consequence of this fact is that we can lift solutions for
all the systems, but gradually and not all at once. 
\begin{thm}
\label{thm:1} Let $H$ be the collection of all the finite systems
of equations. Then, we can order the systems in $H$ in an ascending
subcollections $H_{1}\subset H_{2}\subset H_{3}\subset...$, so that
if we denote by $p_{l}$ the probability that the random group of
level $l$, $\Gamma_{l}$, satisfies the $\Sigma$-l.p. for all $\Sigma\in H_{l}$,
then $p_{l}$ converges to one as $l$ approaches $\infty$. 
\end{thm}

\begin{proof}
Let $q_{l}=1-p_{l}$. Of course, $q_{l}$ equals the probability that
there exists some system $\Sigma$ in $H_{l}$, so that the random
group of level $l$, $\Gamma_{l}$, does not satisfy the $\Sigma$-l.p..
Given a system $\Sigma\in H$, we denote by $q_{l}^{\Sigma}$, the
probability that $\Gamma_{l}$ does not satisfy the $\Sigma$-l.p..
Fixing a system $\Sigma$, we already know that $q_{l}^{\Sigma}\rightarrow0$
as $l\rightarrow\infty$. Since there exists only a countable amount
of finite systems, we can write them in a sequence $\Sigma_{1},\Sigma_{2},\Sigma_{3},...$.
Now let $1=l_{1}\leq l_{2}\leq l_{3},...$ be an ascending sequence
of integers, so that for all $n\in\mathbb{N}$, we have that $q_{l}^{\Sigma_{1}}+q_{l}^{\Sigma_{2}}...+q_{l}^{\Sigma_{n}}\leq\frac{1}{n}$
for all $l\geq l_{n}$.

For all $l\in\mathbb{N}$, let $n$ be the maximal integer for which
$l_{n}\leq l$, and we define $H_{l}$ to be $H_{l}=\{\Sigma_{1},...,\Sigma_{n}\}$.
Of course, the probability that there exists some system $\Sigma$
in $H_{l}$, so that the random group of level $l$, $\Gamma_{l}$,
does not satisfy the $\Sigma$-l.p, $q_{l}$, satisfies $q_{l}\leq q_{l}^{\Sigma_{1}}+q_{l}^{\Sigma_{2}}...+q_{l}^{\Sigma_{n}}\leq\frac{1}{n}$.

Hence $q_{l}\rightarrow0$ as $l\rightarrow\infty$. As required. 
\end{proof}
A simple but important conclusion of the last theorem is that every
limit group over any ascending sequence $\{\Gamma_{l_{n}}\}_{n}$
of random groups, is actually a limit group over the free group $F_{k}$.
More precisely: 
\begin{thm}
\label{thm:2} For each $l$, there exists a subcollection $\mathcal{N}_{l}$
of groups in the level $l$, $\mathcal{M}_{l}$, of the density model
of density $d$, so that the following conditions are satisfied: 
\begin{enumerate}
\item $\mathcal{N}_{l}$ is negligible, i.e., $\frac{|\mathcal{N}_{l}|}{|\mathcal{M}_{l}|}\rightarrow0$
as $l$ approaches $\infty$. 
\item For every convergent sequence $\{h_{l_{n}}:F_{q}\rightarrow\Gamma_{l_{n}}\}$,
where $q$ is some integer, $l_{n}$ is an increasing unbounded sequence
of integers, and $\Gamma_{l_{n}}\in\mathcal{M}_{l_{n}}\backslash\mathcal{N}_{l_{n}}$
for all $n$, the obtained limit group $F_{q}/\underrightarrow{\ker}h_{l_{n}}$
is a limit group over $F_{k}$. 
\end{enumerate}
\end{thm}

\begin{proof}
We start by defining $\mathcal{N}_{l}$ to be the collection of all
the groups of level $l$ that do not satisfy the $\Sigma$-l.p. for
some system $\Sigma$ in the collection $H_{l}$ ($H_{l}$ is defined
in \thmref{1}). Then, by \thmref{1}, the first condition is satisfied,
i.e., $\frac{|\mathcal{N}_{l}|}{|\mathcal{M}_{l}|}\rightarrow0$ as
$l$ approaches $\infty$. For proving the second condition, let $\{h_{l_{n}}:F_{q}\rightarrow\Gamma_{l_{n}}\}$
be a convergent sequence, where $q$ is some integer, $l_{n}$ is
an increasing unbounded sequence of integers, and $\Gamma_{l_{n}}\in\mathcal{M}_{l_{n}}\backslash\mathcal{N}_{l_{n}}$.
We consider the resulted limit group $F_{q}/\underrightarrow{\ker}h_{l_{n}}$,
and we want to prove that it is a limit group over $F_{k}$, i.e.,
that there exists a convergent sequence of homomorphisms from $F_{q}$
to $F_{k}$, so that this sequence converges to $F_{q}/\underrightarrow{\ker}h_{l_{n}}$.

Given an integer $n_{0}$, and a collection of elements $A\subset F_{q}$,
we say that the sequence $\{h_{l_{n}}:F_{q}\rightarrow\Gamma_{l_{n}}\}$
\emph{stabilizes $A$ $n_{0}$-eventually}, if for all $w\in A$,
we have that $h_{l_{n}}(w)=1$ for all $n\ge n_{0}$, or $h_{l_{n}}(w)\neq1$
for all $n\geq n_{0}$. Note that, by definition, a sequence $\{h_{l_{n}}:F_{q}\rightarrow\Gamma_{l_{n}}\}$
is convergent, if for every finite subset $A\subset F_{q}$, there
exists $n_{0}=n_{0}(A)$, so that the sequence stabilizes $A$ $n_{0}$-eventually.

Given an integer $m$, we denote by $B_{m}$ the ball of radius $m$
in $F_{q}$. For all $m$, let $n(m)$ be the minimal integer for
which the sequence $\{h_{l_{n}}:F_{q}\rightarrow\Gamma_{l_{n}}\}$
stabilizes the ball $B_{m}$ $n(m)$-eventually, and s.t. each system
of equations composed from conjunctions of words in $B_{m}$ is contained
in $H_{l_{n(m)}}$.

Now we consider the subsequence $\{h_{l_{n(m)}}:F_{q}\rightarrow\Gamma_{l_{n(m)}}\}_{m}$
of the sequence $\{h_{l_{n}}:F_{q}\rightarrow\Gamma_{l_{n}}\}$. Obviously,
the sequence $\{h_{l_{n(m)}}:F_{q}\rightarrow\Gamma_{l_{n(m)}}\}_{m}$
is convergent, and the limit group defined by it is the same limit
group defined by the original sequence $\{h_{l_{n}}:F_{q}\rightarrow\Gamma_{l_{n}}\}$.

For all $m$, let $W_{m}$ be the collection of words in $B_{m}$
who are mapped to $1$ by $h_{l_{n(m)}}$. Since the system of equations
defined by all the words in $W_{m}$ belongs to $H_{l_{n(m)}}$, and
since $\Gamma_{l_{n(m)}}\in\mathcal{M}_{l_{n(m)}}\backslash\mathcal{N}_{l_{n(m)}}$,
we get that there exists a lift $\tilde{h}_{l_{n(m)}}$ of $h_{l_{n(m)}}$
to $F_{k}$, i.e., $\tilde{h}_{l_{n(m)}}$ is a homomorphism from
$F_{q}$ to $F_{k}$ with $h_{l_{n(m)}}=\pi\circ\tilde{h}_{l_{n(m)}}$,
where $\pi=\pi_{l_{n(m)}}$ is the given natural quotient map from
$F_{k}$ onto $\Gamma_{l_{n(m)}}$, so that $\tilde{h}_{l_{n(m)}}(w)=1$
for all $w\in W_{m}$.

We obtain a sequence of homomorphisms from $F_{q}$ into $F_{k}$,
$\{\tilde{h}_{l_{n(m)}}:F_{q}\rightarrow F_{k}\}_{m}$. By the way
this sequence was constructed, it is a convergent sequence whose limit
is the limit group defined by the sequence $\{h_{l_{n(m)}}:F_{q}\rightarrow\Gamma_{l_{n(m)}}\}_{m}$,
$F_{q}/\underrightarrow{\ker}h_{l_{n}}$. We conclude that $F_{q}/\underrightarrow{\ker}h_{l_{n}}$
is indeed a limit group defined over $F_{k}$. This implies the second
condition in the claim. 
\end{proof}
The meaning of the last theorem, is that we may assume along all the
way that any convergent sequence of homomorphisms from a given f.g.
group into a sequence of groups from the density model that are ordered
with respect to their levels in the model, converges into a group
which is a limit group over $F_{k}$ and which is defined by lifts
of the sequence of the homomorphisms in the original sequence. We
may assume this property, without changing the model, i.e., we may
drop all the groups in the collection $\underset{l}{\cup}\mathcal{N}_{l}$
from the model, without changing its (asymptotic) functionality. 
\begin{thm}
(noetherianity of random groups). Let $\Sigma_{\infty}$ be an infinite
system of equations (with finite number of variables and constants).
Let $\Sigma_{f}$ be a finite system equivalent to $\Sigma_{\infty}$
over $F_{k}$ (Guba's theorem). Then, the systems $\Sigma_{\infty}$
and $\Sigma_{f}$ are equivalent over a random group (with overwhelming
probability). 
\end{thm}

\begin{proof}
Let $\Gamma$ be a random group. Let $x\in\Gamma$ be a solution of
$\Sigma_{f}$ over $\Gamma$. Since every solution of $\Sigma_{f}$
over $\Gamma$ can be lifted to a solution of $\Sigma_{f}$ over $F_{k}$,
there exists a lift $\tilde{x}$ of $x$ to $F_{k}$, s.t. $\tilde{x}$
is a solution of $\Sigma_{f}$ over $F_{k}$. Since $\Sigma_{\infty}$
and $\Sigma_{f}$ are equivalent over $F_{k}$, we have that $\tilde{x}$
is a solution of $\Sigma_{\infty}$ over $F_{k}$. Hence, its projection
$x$ into $\Gamma$, is a solution of $\Sigma_{\infty}$ over $\Gamma$. 
\end{proof}
We end this section by recalling the definition of a sequence of homomorphisms
over an ascending sequence of groups in the model, and the fact that
the random group of level $l$ is hyperbolic with a hyperbolicity
constant linear in $l$. 
\begin{defn}
An \emph{ascending sequence of groups in the model of density $d$},
is a sequence of groups $\{\Gamma_{l_{n}}\}_{n}$, so that: 
\begin{enumerate}
\item The sequence $\{l_{n}\}_{n}$ is an unbounded increasing sequence
of integers. 
\item For all $n$, the group $\Gamma_{l_{n}}$ belongs to the level $l_{n}$
of the model (of density $d$). 
\end{enumerate}
\end{defn}

\begin{thm}
\label{thm:3} (\cite{Sharp=000020phase=000020transition=000020theorems=000020for=000020hyperbolicity=000020of=000020random=000020groups,Some=000020small=000020cancellation=000020properties=000020of=000020random=000020groups}).
At density $d<\frac{1}{2}$, for every $\epsilon>0$, with overwhelming
probability, the random group $\Gamma$ of level $l$ is torsion-free
$\frac{12l}{(1-2d-\epsilon)^{2}}$-hyperbolic.

In particular, $\Gamma$ is $\alpha_{0}l$-hyperbolic, where 
\[
\alpha_{0}=\frac{48}{(1-2d)^{2}}\,.
\]
\end{thm}

\section{Graded Test Sequences over Random Groups}\label{sec:Graded-Test-Sequences-over-Random-Groups}

\subsection{Limit Actions of Sequences over Random Groups}

Let $F_{m}$ be a free group with a fixed basis $x=(x_{1},...,x_{m})$. 
\begin{thm}
\label{thm:4} Let $\{h_{l_{n}}:F_{m}\rightarrow\Gamma_{l_{n}}\}$
be a sequence of homomorphisms over an ascending sequence of groups
in the model. Assume that the stretching factors $\mu_{n}=\underset{f_{n}\in\Gamma_{l_{n}}}{\min}\underset{1\le u\leq m}{\max}d_{\Gamma_{l_{n}}}(1,h_{l_{n}}(x_{u}))$,
satisfies that $\mu_{n}\geq l_{n}^{2}$ for all $n$, where $X_{n}$
is the Cayley graph of $\Gamma_{l_{n}}$ w.r.t. the generating set
$a$. Then, the sequence $\tau_{f_{n}}\circ h_{l_{n}}$, where $\tau_{f_{n}}$
is the corresponding inner automorphism of $\Gamma_{l_{n}}$ (see
section 9 in the previous paper), subconverges to an isometric non-trivial
action of $F_{m}$ on a pointed real tree $(Y,y_{0})$. 
\end{thm}

\begin{proof}
For all $n$, let $\delta_{n}$ be the hyperbolicity constant of the
group $\Gamma_{l_{n}}$. Since the $\delta_{n}$ is linear in $l_{n}$,
then,

\[
\frac{\delta_{n}}{\mu_{n}}\overset{n\rightarrow\infty}{\longrightarrow}0\,.
\]
The remaining details follows as in \cite{DGI} Proposition 1.1. 
\end{proof}
\begin{thm}
\label{thm:5} Keep the notation of \ref{thm:4}, and assume that
the sequence $h_{l_{n}}$ converges into an isometric non-trivial
action of $F_{m}$ on a pointed real tree $(Y,y_{0})$. We denote
$K_{h_{l_{n}},\infty}=\{g\in F_{m}:\forall y\in Y\quad gy=y\}$. Let
$L_{h_{l_{n}},\infty}$ be the quotient group $F_{m}/K_{h_{l_{n}},\infty}$.
Then, 
\begin{enumerate}
\item $L_{h_{l_{n}},\infty}$ is a f.g. group. 
\item $L_{h_{l_{n}},\infty}$ is torsion-free. 
\item If $Y$ is a real line, then $L_{h_{l_{n}},\infty}$ is free abelian. 
\item If $g\in F_{m}$ stabilizes (pointwise) a non-trivial tripod in $Y$,
then $h_{l_{n}}(g)=1$ eventually. 
\item Let $g\in F_{m}$ be an element that does not belong to $K_{h_{l_{n}},\infty}$.
Then, $h_{l_{n}}(g)\neq1$ eventually. 
\item Let $[y_{1},y_{2}]\subset[y_{3},y_{4}]$ be a pair of non-degenerate
segments in $Y$, and assume that $stab([y_{3},y_{4}])$, the (pointwise)
stabilizer of $[y_{3},y_{4}]$ is non-trivial. Then, $stab([y_{3},y_{4}])$
is abelian, and 
\[
stab([y_{1},y_{2}])=stab([y_{3},y_{4}])\,.
\]
\end{enumerate}
\end{thm}

\begin{proof}
The same proof of theorem 9.2 in the previous paper. 
\end{proof}
\begin{thm}
\label{thm:6} Keep the notation of \ref{thm:5}. The group $L_{h_{l_{n}},\infty}$
is a limit group. In particular, the homomorphisms $h_{l_{n}}$ factor
through $L_{h_{l_{n}},\infty}$ eventually. 
\end{thm}

\begin{proof}
By \thmref{2} and \thmref{5}, the group $L_{h_{l_{n}},\infty}$
is a limit group. Since limit groups are finitely presented, $h_{l_{n}}$
must factor through $L_{h_{l_{n}},\infty}$ for all large enough $n$. 
\end{proof}
\begin{thm}
\label{thm:7} Keep the notation of \ref{thm:5}. If the limit group
$L_{h_{l_{n}},\infty}$ is freely indecomposable, then the homomorphisms
$h_{l_{n}}$ can be shortened eventually, by precomposing with a modular
automorphism in the JSJ decomposition of $L_{h_{l_{n}},\infty}$. 
\end{thm}

\begin{proof}
According to \thmref{6}, the homomorphisms $h_{l_{n}}$ factor through
$L_{h_{l_{n}},\infty}$ eventually. Hence, the limit action of $L_{h_{l_{n}},\infty}$
on the limit real tree $Y$, is the limit of the actions of $L_{h_{l_{n}},\infty}$
on the Cayley graphs of the groups $\Gamma_{l_{n}}$. Hence, the theorem
follows by the standard shortening argument presented in \cite{Structure=000020and=000020rigidity=000020in=000020hyperbolic=000020groups=000020I}. 
\end{proof}

\subsection{Graded Test Sequences over Random Groups}

We start by defining a test sequence over an ascending (w.r.t the
levels) sequence of groups in the model. The notion of a test sequence
of a resolution was introduced in \cite{DGII} over free groups. From
now on, we will make an extensive use in the objects defined in \cite{DGII}.
In particular, we will assume that the reader is familiar with the
definitions of resolutions, graded resolutions, completions, closures,
graded completions, graded closures, and covering closures. 
\begin{defn}
\label{def:8} Let $Glim(y,p,a)$ be a graded limit group (w.r.t.
the parameters set $p$), and let $GRes(y,p,a)$ be a well-structured
graded resolution of $Glim(y,p,a)$ that terminates in a group $B(y,p,a)$
which is either $p$-rigid or $p$-solid limit group. Let $Y$ be
the Cayley graph of $Comp(z,y,p,a)$, the graded completion of the
resolution $GRes(y,p,a)$, and denote by $d_{Y}$ the corresponding
metric on $Y$. Let $\{\Gamma_{l_{n}}\}$ be an ascending sequence
of groups in the model, let $X_{n}$ be the Cayley graph of $\Gamma_{l_{n}}$
w.r.t. to the generating set $a$, and denote by $d_{n}$ the corresponding
metric on $X_{n}$. For all $n$, let $(y_{n},p_{n},a)$ be a $\Gamma_{l_{n}}$-exceptional
specialization of $B$ (see \defref[s]{12} and \ref{def:13}),
and denote by $GRes(y,p_{n},a)$ the induced ungraded resolution (ungraded
resolution over $\Gamma_{l_{n}}$). We say that the sequence $\{(y_{n},p_{n},a)\}$,
together with a sequence of automorphisms $\{(\nu_{n}^{i},\tau_{n}^{i})\}$
of the $QH$ subgroups that appear in the decompositions associated
with the various levels of the completed resolutions $Comp(GRes)(z,y,p_{n},a)$,
and a sequence of homomorphisms

\[
\{\lambda_{n}:Comp(GRes)(z,y,p_{n},a)\rightarrow\Gamma_{l_{n}}\},
\]
is a \emph{graded test sequence of the resolution $Comp(z,y,p,a)$
over the ascending sequence of groups in the model $\{\Gamma_{l_{n}}\}$},
if the following conditions hold:

For every index $n$: 
\begin{enumerate}
\item The ungraded resolution $GRes(y,p_{n},a)$ is non-degenerate. 
\item Conditions $(i)-(xiv)$ of Definition 1.20 in \cite{DGII} (w.r.t.
the metrics $d_{n}$ of $(\Gamma_{l_{n}},a)$, where $a$ is the fixed
generating set of $\Gamma_{l_{n}}$) are satisfied. 
\item If $g\in Comp(GRes)(z,y,p,a)$, $d_{Y}(g,1)\leq l_{n}$, and $g$
is not elliptic in at least one of the decompositions associated with
the various levels of the completed resolution $Comp(GRes)(z,y,p_{n},a)$,
then 
\[
l_{n}^{2}\cdot\max d_{n}(p_{n},1)<d_{n}(\lambda_{n}(g),1)\,.
\]
\end{enumerate}
\end{defn}

\begin{lem}
\label{lem:9} Keep the assumptions of \ref{def:8}. Assume further
that for all $n$, the induced ungraded resolution $GRes(y,p_{n},a)$
(ungraded resolution over $\Gamma_{l_{n}}$), is non-degenerate. Then,
there exist a sequence of automorphisms $\{(\nu_{n}^{i},\tau_{n}^{i})\}$
of the $QH$ subgroups that appear in the decompositions associated
with the various levels of the completed resolutions $Comp(GRes)(z,y,p_{n},a)$,
and a sequence of homomorphisms 
\[
\lambda_{n}:Comp(GRes)(z,y,p_{n},a)\rightarrow\Gamma_{l_{n}},
\]
so that these sequences form a graded test sequence of $Comp(z,y,p,a)$
over the ascending sequence of groups in the model $\{\Gamma_{l_{n}}\}$.
Furthermore, for every $n$, given any two integers $s_{1}<s_{2}$,
and given a basis element $q_{i}^{t}$ - which is not a peg - of some
of the abelian groups that appear in a decomposition associated with
some level of $Comp(GRes)(z,y,p_{n},a)$, we can choose that $\lambda_{n}(q_{i}^{t})=h_{n}^{ms_{2}+s_{1}}$
for some integer $m$, and so that the element $h_{n}\in\Gamma_{l_{n}}$
has no non-trivial roots in $\Gamma_{l_{n}}$. 
\end{lem}

\begin{lem}
\label{lem:10} Let $Glim(y,p,a)$ be a graded limit group, and let
$GRes(y,p,a)$ be a well-structured graded resolution of $Glim(y,p,a)$
that terminates in a group $B(y,p,a)$ which is either $p$-rigid
or $p$-solid limit group. Let $\{\Gamma_{l_{n}}\}$ be an ascending
sequence of groups in the model, and assume that for all $n$, we
are given a graded test sequence $\{\lambda_{s}^{n}:Comp(GRes)(z,y,p,a)\rightarrow\Gamma_{l_{n}}\}_{s}$
(with corresponding automorphisms $\{(\nu_{s}^{n,i},\tau_{s}^{n,i})\}_{s}$),
of the completed resolution $Comp(GRes)(z,y,p,a)$ over the specific
torsion-free hyperbolic group $\Gamma_{l_{n}}$. Then, if for all
$n$, $s_{n}$ is large enough, then the sequence $\{\lambda_{s_{n}}^{n}:Comp(GRes)(z,y,p,a)\rightarrow\Gamma_{l_{n}}\}_{n}$
(with the corresponding automorphisms $\{(\nu_{s_{n}}^{n,i},\tau_{s_{n}}^{n,i})\}_{n}$),
is a graded test sequence of $Comp(GRes)(z,y,p,a)$ over the ascending
sequence of groups in the model $\{\Gamma_{l_{n}}\}$. 
\end{lem}

The following lemma guarantees that the non-degenerations of a graded
resolution over any group in the model, are encoded by the singular
locus, defined over $F_{k}$, of that resolution. 
\begin{lem}
\label{lem:11} Let $Glim(y,p,a)$ be a graded limit group, and let
$GRes(y,p,a)$ be a graded resolution of $Glim(y,p,a)$ that terminates
in a group $B(y,p,a)$ which is either $p$-rigid or $p$-solid limit
group. Let $S_{1}(y,p,a),...,S_{r}(y,p,a)$ be the finitely many limit
groups that form the singular locus of the resolution $GRes(y,p,a)$
(constructed over $F_{k}$). Let $\Gamma$ be a group in the model.
Then, the $\Gamma$-specializations $(y_{0},p_{0},a)$ of $B(y,p,a)$
for which the induced resolution $GRes(y,p_{0},a)$ is degenerated
(over $\Gamma$), factors through the singular locus of the resolution
$GRes(y,p,a)$. 
\end{lem}

\begin{proof}
By definition, for an ungraded resolution $GRes(y,p_{0},a)$ induced
from a $\Gamma$-specialization $(y_{0},p_{0},a)$ of $B$, to be
degenerated, the specialization $(y_{0},p_{0},a)$ has to satisfy
one of finitely many systems $\Sigma_{1}(y,p,a),...,\Sigma_{s}(y,p,a)$
that was constructed over $F_{k}$. Given such a $\Gamma$-specialization
$(y_{0},p_{0},a)$, We lift it to a specialization $(\tilde{y}_{0},\tilde{p}_{0},a)$
of $B(y,p,a)$ that satisfies one of the systems $\Sigma_{1}(y,p,a),...,\Sigma_{s}(y,p,a)$. 
\end{proof}

\section{Rigid and Solid Limit Groups}\label{sec:Rigid-and-Solid-Limit-Groups}

Let $\Sigma(y,p,a)$ be a $p$-graded system of equations, that is
$p$ is the parameter set, $y$ is the variable set, and the constants
appearing in $\Sigma$ belong to the free group $F_{k}=F(a)$. According
to \cite{DGI}, section 10, we construct the graded MR diagram $K$
of the system $\Sigma(y,p,a)$. This is a finite diagram that encodes
all the solutions of the graded system $\Sigma(y,p,a)$, so that the
JSJ-decompositions associated with the various groups along it, are
$p$-graded ($p$-JSJ's). That is, the subgroups $\langle a,p\rangle$
of each of the groups along the diagram, is contained in a distinguished
vertex of the $p$-JSJ associated with the corresponding group, and
the modular automorphisms of that group, that are induced from the
associated $p$-JSJ, must fix (elementwise) the distinguished vertex
(by definition).

Each resolution in the graded MR diagram, terminates in a limit group
which is either $p$-rigid or $p$-solid. The $p$-rigid limit groups,
denoted usually by $Rgd(y,p,a)$, are by definition, those graded
limit groups that admit a trivial $p$-JSJ. And $p$-solid limit groups,
denoted usually by $Sld(y,p,a)$, are defined as the graded limit
groups who admit a non-trivial $p$-JSJ, but admit an isomorphic shortening
quotient.

Every $p$-rigid or $p$-solid limit group in the diagram $K$, admits
a finite set (maybe empty) of maximal flexible quotients $Flx(y,p,a)$,
which are proper quotients. 
\begin{defn}
\label{def:12} Let $Rgd(y,p,a)$ be a $p$-rigid limit group. Let
$\Gamma=\langle a:\mathcal{R}\rangle$ be a group (presentation).
A homomorphism $h:Rgd(y,p,a)\rightarrow\Gamma$ is called a \emph{$\Gamma$-rigid
homomorphism of $Rgd(y,p,a)$}, if it does not factor through a flexible
quotient of $Rgd(y,p,a)$. A $\Gamma$-rigid homomorphism is called
also a \emph{$\Gamma$-exceptional solution of $Rgd(y,p,a)$}. 
\end{defn}

\begin{defn}
\label{def:13} Let $Sld(y,p,a)$ be a $p$-solid limit group.

If $Sld(y,p,a)$ admits no flexible quotients, then every homomorphism
$Sld(y,p,a)\rightarrow\Gamma$ is defined to be a \emph{$\Gamma$-strictly
solid homomorphism of $Sld(y,p,a)$}.

In general, let $\eta_{j}:Sld(y,p,a)\rightarrow Flx_{j}(y,p,a)$ be
the canonical quotient maps between the $p$-solid group onto each
of its (finitely many) maximal flexible quotients. According to {[}\cite{DGIII},
Definition 1.5 p. 15, 16{]}, we associate to each one-step resolution
$\eta_{j}:Sld(y,p,a)\rightarrow Flx_{j}(y,p,a)$ a graded completion
$Comp_{j}(z,y,p,a)$, so that the $p$-solid group $Sld(y,p,a)$ is
mapped canonically into each of these completions $\tau_{j}:Sld(y,p,a)\rightarrow Comp_{j}$.
Moreover, we consider the graded completion $Comp_{T}$ of the one-step
resolution $Sld(y,p,a)\rightarrow Sld(y,p,a)$ (the identity map).
The solid limit group $Sld(y,p,a)$ admits two natural embeddings
$\iota_{i}:Sld(y,p,a)\rightarrow Comp_{T}$, $i=1,2$.

Let $\Gamma=\langle a:\mathcal{R}\rangle$ be a group (presentation).
A homomorphism $h:Sld(y,p,a)\rightarrow\Gamma$ is called a \emph{$\Gamma$-strictly
solid homomorphism of $Sld(y,p,a)$}, if $h$ is non-degenerate (see
\lemref{11}) and does not factor through the map $\tau_{j}$ for
all $j$. Moreover, we define an equivalence relation on the set of
homomorphisms $h:Sld(y,p,a)\rightarrow\Gamma$ as follows. Two such
homomorphisms $h_{1},h_{2}$ are called equivalent, if there exists
a homomorphism $u:Comp_{T}\rightarrow\Gamma$ so that $u\circ\iota_{i}=h_{i}$,
for $i=1,2$. We call the equivalence class of a $\Gamma$-strictly
solid homomorphism a \emph{$\Gamma$-strictly solid family}. Note
that all the non-degenerate homomorphisms in a $\Gamma$-strictly
solid family, are $\Gamma$-strictly solid.

A $\Gamma$-strictly solid homomorphism is called also a \emph{$\Gamma$-exceptional
solution of $Sld(y,p,a)$}, and a $\Gamma$-strictly solid family
is called also a \emph{$\Gamma$-exceptional family of $Sld(y,p,a)$}.

When $\Gamma=F_{k}$, we remove $\Gamma$ from the notation, and write
simply strictly solid or exceptional homomorphisms, and strictly solid
or exceptional families. 
\end{defn}

\begin{thm}
\label{thm:14}(\cite{DGIII}, theorem 2.3). Let $Rgd(y,p,a)$ be
a $p$-rigid limit group. There exists a global bound $b_{R}(Rgd)$
for which, for any particular value of the defining parameters $p_{0}$,
there are at most $b_{R}$ distinct rigid homomorphisms $h:Rgd(y,p,a)\rightarrow F_{k}$
satisfying $h(p)=p_{0}$. 
\end{thm}

\begin{thm}
\label{thm:15} Let $Rgd(y,p,a)$ be a $p$-rigid limit group. There
exists a global bound $b_{R}(Rgd)$ for which the following property
is satisfied. Let $\Gamma_{l}$ be a group of level $l$ in the model.

If $l$ is large enough, then, for any particular value of the defining
parameters $p_{0}\in\Gamma_{l}$, there are at most $b_{R}$ distinct
rigid homomorphisms $h:Rgd(y,p,a)\rightarrow\Gamma_{l}$ satisfying
$h(p)=p_{0}$. 
\end{thm}

\begin{proof}
Let $b_{R}$ be the one in \thmref{14}, and assume by contradiction
that $\Gamma=\Gamma_{l}$ admits more than $b_{R}$ pairwise distinct
$\Gamma$-rigid homomorphisms $h_{j}:Rgd(y,p,a)\rightarrow\Gamma$.
For each $j$, we lift the homomorphism $h_{j}$ into a homomorphism
$\tilde{h}_{j}:Rgd(y,p,a)\rightarrow F_{k}$.

Then, since the homomorphisms $h_{j}$ are pairwise distinct, so are
the homomorphisms $\tilde{h}_{j}$. Moreover, for all $j$, the homomorphism
$\tilde{h}_{j}$ is rigid, since $h_{j}$ is. A contradiction. 
\end{proof}
\begin{thm}
\label{thm:16} (\cite{DGIII}, theorem 2.5). Let $Sld(y,p,a)$ be
a $p$-solid limit group. There exists a global bound $b_{S}(Sld)$
for which, for any particular value of the defining parameters $p_{0}$,
there are at most $b_{S}$ distinct strictly solid families of homomorphisms
$h:Sld(y,p,a)\rightarrow F_{k}$ satisfying $h(p)=p_{0}$. 
\end{thm}

\begin{thm}
\label{thm:17} Let $Sld(y,p,a)$ be a $p$-solid limit group. There
exists a global bound $b_{S}(Sld)$ for which the following property
is satisfied. Let $\Gamma_{l}$ be a group of level $l$ in the model.

If $l$ is large enough, then, for any particular value of the defining
parameters $p_{0}\in\Gamma_{l}$, there are at most $b_{S}$ distinct
strictly solid families of homomorphisms $h:Sld(y,p,a)\rightarrow\Gamma_{l}$
satisfying $h(p)=p_{0}$. 
\end{thm}

\begin{proof}
Let $b_{S}$ be the one in \thmref{16}, and assume by contradiction
that $\Gamma=\Gamma_{l}$ admits more $b_{S}$ pairwise distinct $\Gamma$-strictly
solid families. Let $h_{j}$ be strictly solid representatives for
these families. For each $j$, we lift the homomorphism $h_{j}$ into
a homomorphism $\tilde{h}_{j}:Sld(y,p,a)\rightarrow F_{k}$.

Then, since the homomorphisms $h_{j}$ represent pairwise distinct
$\Gamma$-strictly solid families, the homomorphisms $\tilde{h}_{j}$
must represent pairwise distinct strictly solid families. A contradiction. 
\end{proof}
\begin{lem}
\label{lem:18} Let $B(y,p,a)$ be a $p$-rigid or a $p$-solid limit
group. Let $Glim(x,y,p,a)$ be a $p$-graded limit group, and let
$GRes(x,y,p,a)$ be a $p$-graded resolution of $Glim(x,y,p,a)$.
Assume that $B(y,p,a)$ is mapped canonically into $Glim(x,y,p,a)$,
and denote by $\eta_{j}:B(y,p,a)\rightarrow Glim_{j}(x,y,p,a)$ the
corresponding homomorphisms of $B(y,p,a)$ into the various levels
of $GRes(x,y,p,a)$. If $B(y,p,a)$ admits an exceptional specialization
$(y_{0},p_{0},a)$ that factors through the resolution $GRes(x,y,p,a)$,
then for all $j$, the $p$-JSJ of $B(y,p,a)$ is mapped compatibly
into the $p$-JSJ of $Glim_{j}(x,y,p,a)$. 
\end{lem}

\begin{proof}
Assume by contradiction it is not the case. Let $Flx(c,y,p,a)=\{FlxR_{1},...,FlxR_{m}\}$
be the collection of the relations of the graded completions of the
graded resolutions $B(y,p,a)\rightarrow Flx_{1}(y,p,a),...,B(y,p,a)\rightarrow Flx_{m}(y,p,a)$,
where $Flx_{1}(y,p,a),...,Flx_{m}(y,p,a)$ are the maximal flexible
quotients of $B(y,p,a)$. Let $Comp(z,x,y,p,a)$ be the completion
of the resolution $GRes(x,y,p,a)$, and let $Comp(z,x,y,p_{0},a)$
be the ungraded resolution corresponding to a factorization of an
exceptional specialization $(y_{0},p_{0},a)$ of $B(y,p,a)$ through
$Comp(z,x,y,p,a)$.

We Construct the formal MR diagram $D$ of the systems $Flx(c,y,p_{0},a)$
w.r.t. the resolution $Comp(z,x,y,p_{0},a)$ (see \cite{DGII}, section
2, or \secref{Formal-Solutions-over-Random-Groups} below). Since
the $p$-JSJ of $B(y,p,a)$ is not mapped compatibly into the $p$-JSJ
of $Glim_{j}(x,y,p,a)$, for some $j$, we conclude that every test
sequence of $Comp(z,x,y,p_{0},a)$, sub-factors through one of the
terminal formal closures $FCl(c,z,x,y,p_{0},a)$ in the formal MR
diagram $D$. Hence, These formal closures form a covering closure
for the resolution $Comp(z,x,y,p_{0},a)$. In particular, the specialization
$(y_{0},p_{0},a)$ of $B(y,p,a)$, factors through one of these formal
closures. Thus, $(y_{0},p_{0},a)$ is not an exceptional specialization
of $B(y,p,a)$, a contradiction. 
\end{proof}
As stated in the top of this section, all the solutions of the graded
system $\Sigma(y,p,a)$, factors through the graded MR diagram $K$
of the graded system $\Sigma(y,p,a)$. By definition, a homomorphism
factors through a graded resolution in $K$, if it can be expressed
as a composition of graded modular automorphisms of the different
levels of the resolution, the canonical epimorphisms between consecutive
levels, and a rigid or a strictly solid homomorphism of the terminal
$p$-rigid or $p$-solid limit group of the resolution.

Actually, in order to cover all the solutions of the graded system
$\Sigma(y,p,a)$ by the graded resolutions of the diagram $K$, it
is not sufficient to consider only strictly solid solutions of the
terminal $p$-solid limit groups. But we need further to consider
strictly solid solutions with respect to finitely many sets of graded
closures of the graded resolutions of the graded MR diagram associated
with each $p$-solid group that appear in the diagram $K$. The details
of this more general definition are stated in definition 2.12 in \cite{DGIII}.
This definition can be converted naturally in order to define a $\Gamma$-strictly
solid solution of a $p$-solid limit group with respect to a given
set of closures for a general group $\Gamma=\langle a:\mathcal{R}\rangle$.

According to theorem 2.13 in \cite{DGIII}, the number of distinct
families of strictly solid homomorphisms of a $p$-solid group with
respect to a given set of closures, is globally bounded. Using this
theorem, a similar argument given in the proof of \thmref{17}, proves
the existence of a uniform bound on the number of distinct families
of $\Gamma$-strictly solid homomorphisms of a $p$-solid group with
respect to a given set of closures, for almost all the groups $\Gamma$
in the model.

Hence, for simplicity along the next sections, we will use the terminology
of strictly solid homomorphisms, without stating the corresponding
set of closures.

\section{Formal Solutions over Random Groups}\label{sec:Formal-Solutions-over-Random-Groups}

Along this section, we fix a $p$-graded resolution $Res(y,p,a)$,
and we consider its completion $Comp=Comp(Res)(z,y,p,a)$. We also
fix a system of equations $\Sigma(x,y,p,a)$, where $x,y$ are the
variables, and $p$ denotes the parameter set ($a$ is the fixed basis
of the fixed free group $F_{k}$). We denote by $B$ the $p$-graded
base group of $Comp(z,y,p,a)$.

Our aim is to collect all the formal solutions of $\Sigma(x,y,p,a)$,
w.r.t. to the resolution $Res(y,p,a)$, over almost all the groups
in the model, in a uniform diagram.

For that, we construct an auxiliary graded formal MR diagram that
encodes all of these formal solutions. Then, we will use it in order
to prove that the graded formal MR diagram of the system $\Sigma(x,y,p,a)$
w.r.t. the resolution $Res(y,p,a)$ over the free group $F_{k}$,
that encodes all the formal solutions of the system $\Sigma(x,y,p,a)$
over $F_{k}$, encodes also all the formal solutions of $\Sigma(x,y,p,a)$
w.r.t. the resolution $Res(y,p,a)$ over almost all the groups in
the model.

\subsection{The Construction of Graded Formal MR Diagram over Random Groups}

Let $\mathfrak{T}$ be the collection of all the sequences $\{(x_{n},z_{n},y_{n},p_{n},a)\}$
for which the following conditions hold: 
\begin{enumerate}
\item The sequence $\{(x_{n},z_{n},y_{n},p_{n},a)\}$ is defined over some
ascending sequence $\{\Gamma_{l_{n}}\}$ of groups in the model. 
\item The $\Gamma_{l_{n}}$-specialization $(x_{n},z_{n},y_{n},p_{n},a)$
satisfies the system $\Sigma$, i.e., $\Sigma(x_{n},z_{n},y_{n},p_{n},a)=1$
(in $\Gamma_{l_{n}}$), for all $n$. 
\item The restricted sequence $\{(z_{n},y_{n},p_{n},a)\}$ is a graded test
sequence over the (same) ascending sequence $\{\Gamma_{l_{n}}\}$
of groups in the model. 
\end{enumerate}
The sequences in the collection $\mathfrak{T}$, factor through finitely
many maximal limit groups $H$ (see \thmref{2}). The canonical image
of the subgroup $Comp(z,y,p,a)$ in $H$ is denoted briefly by $Comp$.
Such an $H$ could be freely decomposable w.r.t. the subgroup $Comp\leq H$.
We consider the most refined $Comp$-free factorization of $H$. We
call the free factors that do not include (a conjugate of) the subgroup
$Comp\leq H$, a \emph{standard free factors}.

We keep the standard factors, extend them by their standard MR diagrams,
and continue the construction with the free factor that contains $Comp$.
We denote this factor by $M=M(x,z,y,p,a)$, which is a limit group.
We consider the $Comp$-JSJ decomposition of the limit group $M(x,z,y,p,a)$,
i.e., the JSJ decomposition with respect to the subgroup $Comp$.
If this decomposition is trivial, we call the group $M(x,z,y,p,a)$
a \emph{$Comp$-rigid} group, and terminate the construction for this
stage. In this case, the subcollection of $\mathfrak{T}$ consisting
of all the sequences $\{(x_{n},z_{n},y_{n},p_{n},a)\}$ that factor
through $M$, is denoted by $\mathfrak{T}_{M}$.

In the case that the $Comp$-JSJ of $M(x,z,y,p,a)$ is non-trivial,
we denote by $\mathfrak{T}_{M}$ the subcollection of $\mathfrak{T}$,
consisting of all the sequences $\{(x_{n},z_{n},y_{n},p_{n},a)\}$
that satisfy the following conditions: 
\begin{enumerate}
\item The sequence $\{(x_{n},z_{n},y_{n},p_{n},a)\}$, that is defined over
an ascending sequence $\{\Gamma_{l_{n}}\}$ of groups in the model,
factors through the group $M$. 
\item For all $n$, the specialization $(x_{n},z_{n},y_{n},p_{n},a)$ cannot
be shortened (in the metric of $(\Gamma_{l_{n}},a)$, where $a$ is
the fixed generating set of $\Gamma_{l_{n}}$) by precomposing with
a modular automorphism of $M$ that is induced from its $Comp$-JSJ. 
\end{enumerate}
The sequences in the collection $\mathfrak{T}_{M}$, factor through
finitely many maximal limit quotients of $M$, that we denote - while
abusing the notation - by $H$. If $H$ is not a strict quotient of
$M$, we call $M$ a \emph{$Comp$-solid} group, and terminate the
construction for this stage.

We continue the construction iteratively. Since all the groups that
are obtained along the construction, are limit groups, then, according
to the descending chain condition of limit groups, the construction
must terminate after finitely many steps. The result is a diagram
$D$ that includes finitely many resolutions. Each of these resolutions
terminates in a free product $\langle f\rangle\ast M$, where $\langle f\rangle$
is a free group, and $M$ is either a $Comp$-rigid group or a $Comp$-solid
group. 
\begin{lem}
\label{lem:19} Assume that $M(x,z,y,p,a)$ is a $Comp$-rigid group
or a $Comp$-solid group. Then, $M$ admits a structure of a $p$-graded
tower over a $p$-graded base limit group $Term(M)$, so that: 
\begin{enumerate}
\item up to adding roots to abelian groups, this structure is similar to
the one associated with $Comp(z,y,p,a)$, 
\item this structure is compatible to the one associated with $Comp(z,y,p,a)$,
that is, the canonical map $Comp\rightarrow M$ maps the decomposition
associated with each level of $Comp$, except maybe the bottom level
of $Comp$, compatibly into the decomposition associated with the
corresponding level of $M$, 
\item the canonical image of the $p$-graded base group $B$ of $Comp(z,y,p,a)$
lies entirely in $Term(M)$. 
\end{enumerate}
\end{lem}

\begin{proof}
Let $(x_{n},z_{n},y_{n},p_{n},a)$ be a sequence in $\mathfrak{T}_{M}$,
that converges into a non-trivial action of $M(x,z,y,p,a)$ on a real
tree. Consider the abelian decomposition $\Pi$ that this action induces
for $M$.

If $Comp$ is contained in a proper subgraph of $\Pi$, then the $Comp$-JSJ
of $M(x,z,y,p,a)$ is non-trivial, and, according to the shortening
argument over ascending sequence of groups in the model (\thmref{7}),
the (generators of some part of the $Comp$-JSJ of $M$ induced by
the) specializations in the sequence $(x_{n},z_{n},y_{n},p_{n},a)$
can eventually be shortened by precomposing with modular automorphisms
in the $Comp$-JSJ of $M(x,z,y,p,a)$, a contradiction. Hence, the
components of the action of $M$ on the real tree are of the same
forms as those of $Comp$. Using the properties of test sequences
over ascending sequence of groups in the model, we can apply the analysis
in the proofs theorem 3.7, and theorem 3.8 of \cite{DGII}, in order
to analyze the action of $M$ on the limit real tree.

We deduce that the action of $M$ contains either a single IET, which
is the ``highest'' surface in the decomposition associated with
the completion $Comp(z,y,p,a)$, and finitely many point stabilizers,
or the action of $M$ is discrete with a unique point stabilizer and
one edge, which (iteratively as in the proof of the existence of a
formal solution over abelian groups in \cite{DGII}, Proposition 1.8)
gives a decomposition for $M$ of the form $V\ast_{Ab_{1}}Ab$, where
$Ab$ is obtained from the ``highest'' abelian group of the decomposition
associated with the completion $Comp(z,y,p,a)$ by adding some roots
for non-pegs generators.

Continuing with each of the vertices (point stabilizers) iteratively,
this uncovering procedure explains that $M$ admits a tower structure
whose upper levels are similar to those of the decomposition associated
with the highest level of the completion $Comp(z,y,p,a)$, up to adding
roots for abelian groups along this decomposition.

Finally, this procedure ends up when we get a point stabilizer $Term(M)$
whose intersection with $Comp$ is of the form $\bar{B}\ast F$, where
$\bar{B}$ is the image of $B$ and $F$ is a free group. 
\end{proof}
In addition to \lemref{19}, we want that the $p$-graded base group
$B$ of $Comp(z,y,p,a)$ to be mapped compatibly into the $p$-graded
base group $Term(M)$ of $M$.

For that, we note that at this point, we can consider the finitely
many base groups $Term(M)$ that we have constructed so far, as were
given in advance. According to \thmref{2}, each of these groups is
a limit group, and in particular a finitely presented group. Hence,
according to \thmref{1}, we can assume that if $\Gamma$ is any group
in the model, then every $\Gamma$-specialization of $Term(M)$ can
be lifted to a specialization of $Term(M)$ in $F_{k}$. As a consequence,
we obtain the compatibility that we wanted: 
\begin{lem}
\label{lem:20} Assume that $M(x,z,y,p,a)$ is a $Comp$-rigid group
or a $Comp$-solid group. Then, $M$ admits a structure of a graded
formal closure of $Comp(z,y,p,a)$ over a $p$-graded limit group
$Term(M)$. That is, $M$ admits a structure of a $p$-graded tower
over a $p$-graded base group $Term(M)$, so that: 
\begin{enumerate}
\item up to adding roots to abelian groups, this structure is similar to
the one associated with $Comp(z,y,p,a)$, 
\item this structure is compatible to the one associated with $Comp(z,y,p,a)$, 
\item the canonical image of the $p$-graded base group $B$ of $Comp(z,y,p,a)$
lies entirely in $Term(M)$, and $B$ is mapped compatibly into $Term(M)$.
That is, the minimal subgraph of groups of the $p$-JSJ of $Term(M)$
containing the canonical image $\bar{B}$ of $B$, consists only from
parts that are taken from the $p$-JSJ of $B$. 
\end{enumerate}
\end{lem}

\begin{proof}
Let $(x_{n},z_{n},y_{n},p_{n},a)$ be a sequence in $\mathfrak{T}_{M}$.
Since the sequence $(x_{n},z_{n},y_{n},p_{n},a)$ factors through
$M$, and since the graded limit quotient $\bar{B}$ of $B$ lies
entirely in $Term(M)$, we conclude that there exists a $\Gamma$-exceptional
specialization of $B$ that factors through $Term(M)$, for some group
$\Gamma=\Gamma_{n}$ in the model. We lift this specialization to
a specialization of $B$ (in $F_{k}$) that factors through $Term(M)$.
Note that this lift must be an exceptional specialization of $B$
(over $F_{k}$). Hence, according to \lemref{18}, the $p$-JSJ of
$Term(M)$ must be compatible with that of $B$. We conclude that
$M$ is a graded closure of $Comp(z,y,p,a)$ over some $p$-graded
limit group $Term(M)$. 
\end{proof}
We denote by $D$ the diagram obtained by our iterative construction.
In light of \lemref{20}, every resolution of the diagram $D$ that
we have constructed above, is defined over a graded formal closure
of the graded completion $Comp(z,y,p,a)$. And according to the construction
of $D$, every sequence $\{(x_{n},z_{n},y_{n},p_{n},a)\}$ of specializations
over an ascending sequence $\{\Gamma_{l_{n}}\}$ of groups in the
model, that satisfies the system $\Sigma(x,y,p,a)$, and restricts
to a test sequence $\{(z_{n},y_{n},p_{n},a)\}$ over the ascending
sequence $\{\Gamma_{l_{n}}\}$, subfactors through at least one of
the resolutions in $D$. Note that factorization through such a resolution
is defined using the canonical maps, factorization through the standard
MR diagrams associated with the standard free factors, together with
modular automorphisms induced by the $Comp$-JSJ's of the corresponding
(non-standard) groups along the diagram.

The obtained diagram $D$ is called \emph{the auxiliary graded formal
MR diagram over random groups of the system $\Sigma(x,y,p,a)$ over
the graded resolution $Res(y,p,a)$.}

Our aim is to show that the graded formal MR diagram over $F_{k}$
of the system $\Sigma(x,y,p,a)$ over the graded resolution $Res(y,p,a)$
(that we explain in the next subsection), encodes all the formal solutions
of a random group $\Gamma$ (in probability $1$) for any value of
the parameters set $p_{0}\in\Gamma$, for which the sentence

\[
\forall y\in Res(y,p_{0},a)\quad\Sigma(x,y,p_{0},a)=1
\]
is a truth sentence over $\Gamma$. The precise statement formulating
this property, will be given below (\thmref{27}). For that aim, we
start with the following lemma: 
\begin{lem}
\label{lem:21} Let $\Gamma_{l}$ be a group of level $l$ in the
model. If $l$ is large enough, then the following property is satisfied
in $\Gamma_{l}$.

Every sequence of specializations $\{(x_{n},z_{n},y_{n},p_{n},a)\}$
over $\Gamma_{l}$, for which $\Sigma(x_{n},y_{n},p_{n},a)=1$ in
$\Gamma_{l}$ for all $n$, and the restricted sequence $\{(z_{n},y_{n},p_{n},a)\}$
is a test sequence of $Comp(z,y,p,a)$ over $\Gamma_{l}$, subfactors
through the auxiliary graded formal MR diagram over random groups
$D$. 
\end{lem}

\begin{proof}
Assume by contradiction the statement is false. Then, there will be
an ascending sequence $\{\Gamma_{l_{n}}\}$ of groups in the model,
so that for all $n$, there exists a sequence $\{(x_{m}^{n},z_{m}^{n},y_{m}^{n},p_{m}^{n},a)\}_{m}$
of $\Gamma_{l_{n}}$-specializations, for which: 
\begin{enumerate}
\item For all $m$, the $\Gamma_{l_{n}}$-specialization $(x_{m}^{n},z_{m}^{n},y_{m}^{n},p_{m}^{n},a)$
satisfies the system $\Sigma$, i.e., $\Sigma(x_{m}^{n},z_{m}^{n},y_{m}^{n},p_{m}^{n},a)=1$
(in $\Gamma_{l_{n}}$). 
\item The restricted sequence $\{(z_{m}^{n},y_{m}^{n},p_{m}^{n},a)\}_{m}$
is a test sequence of $Comp(z,y,p,a)$ over $\Gamma_{l_{n}}$. 
\item For all $m$, the $\Gamma_{l_{n}}$-specialization $(x_{m}^{n},z_{m}^{n},y_{m}^{n},p_{m}^{n},a)$
does not factor through $D$. 
\end{enumerate}
Hence, using \lemref{10}, there exists a sequence $(m_{n})_{n}$
of integers, so that the sequence

\[
\{(x_{m_{n}}^{n},z_{m_{n}}^{n},y_{m_{n}}^{n},p_{m_{n}}^{n},a)\}_{n}
\]
which is defined over the ascending sequence $\{\Gamma_{l_{n}}\}$
of groups in the model, satisfies the system $\Sigma(x,y,p,a)=1$,
restricts to a test sequence $\{(z_{m_{n}}^{n},y_{m_{n}}^{n},p_{m_{n}}^{n},a)\}_{n}$
over the ascending sequence $\{\Gamma_{l_{n}}\}$, but it does not
subfactor through $D$. This contradicts the construction of $D$. 
\end{proof}

\subsection{The Construction of the Graded Formal MR Diagram over Free Groups}

In this subsection we will construct the graded formal MR diagram
over the free group $F_{k}$ of the system $\Sigma(x,y,p,a)$ w.r.t.
the graded resolution $Res(y,p,a)$. The construction is essentially
similar to the construction that we have explained above over random
groups.

Let $\mathfrak{J}$ be the collection of all the sequences $\{(x_{n},z_{n},y_{n},p_{n},a)\}$
of specializations (over $F_{k}$), for which $\Sigma(x_{n},z_{n},y_{n},p_{n},a)=1$
(in $F_{k}$) for all $n$, and for which the restricted sequence
$\{(z_{n},y_{n},p_{n},a)\}$ is a test sequence of $Comp(z,y,p,a)$
(over $F_{k}$).

The sequences in the collection $\mathfrak{J}$, factor through finitely
many maximal limit groups $H$. We consider the most refined $Comp$-free
factorization of $H$.

We keep the free factors that do not include (a conjugate of) the
subgroup $Comp\leq H$, extend them by their standard MR diagrams,
and continue the construction with the free factor that contains $Comp$.
We denote this factor by $M=M(x,z,y,p,a)$. We consider the $Comp$-JSJ
decomposition of the limit group $M(x,z,y,p,a)$. If $M(x,z,y,p,a)$
is $Comp$-rigid, i.e., if its $Comp$-JSJ is trivial, we terminate
the construction for this stage. In this case, the subcollection of
$\mathfrak{J}$ consisting of all the sequences $\{(x_{n},z_{n},y_{n},p_{n},a)\}$
that factor through $M$, is denoted by $\mathfrak{J}_{M}$.

In the case that the $Comp$-JSJ of $M(x,z,y,p,a)$ is non-trivial,
we consider the subcollection $\mathfrak{J}_{M}$ of $\mathfrak{J}$,
consisting of all sequences $\{(x_{n},z_{n},y_{n},p_{n},a)\}$ that
factor through $M$, and so that for all $n$, the specialization
$(x_{n},z_{n},y_{n},p_{n},a)$ cannot be shortened (in the metric
of $(F_{k},a)$, where $a$ is the fixed generating set of $F_{k}$)
by precomposing with a modular automorphism of $M$ that is induced
from its $Comp$-JSJ. The sequences in the collection $\mathfrak{J}_{M}$,
factor through finitely many maximal limit quotients of $M$, that
we denote again by $H$. If the obtained $H$ is not a strict quotient
of $M$, i.e., if $M$ is a $Comp$-solid group, we terminate the
construction for this stage.

We continue the construction iteratively. According to the descending
chain condition of limit groups, the construction must terminate after
finitely many steps. The result is a diagram that we denote by $K$,
that includes finitely many resolutions. Each of these resolutions
terminates in a free product $\langle f\rangle\ast M$, where $\langle f\rangle$
is a free group, and $M$ is either a $Comp$-rigid group or a $Comp$-solid
group. 
\begin{lem}
\label{lem:22} Assume that $M(x,z,y,p,a)$ is a $Comp$-rigid group
or a $Comp$-solid group. Then, $M$ admits a structure of a graded
formal closure of $Comp(z,y,p,a)$ over a $p$-graded limit group
$Term(M)$. 
\end{lem}

\begin{proof}
The proof is essentially the same as in \lemref{19} and \lemref{20}.
For more details see \cite{DGII} theorems 3.7, and 3.8. 
\end{proof}
\begin{rem}
We note that the above construction of a graded formal MR diagram
over $F_{k}$ of a given graded system $\Sigma(x,y,p,a)$ w.r.t. a
given graded resolution $Res(y,p,a)$, can be applied also in the
case that the terminal group of $Res(y,p,a)$ is not rigid or solid.
This means that we can consider the completion $Comp$ of $Res(y,p,a)$,
and then to consider the collection of all sequences $\{(x_{n},z_{n},y_{n},p_{n},a)\}$
of specializations, for which $\Sigma(x_{n},z_{n},y_{n},p_{n},a)=1$
for all $n$, and for which the sequence $\{(z_{n},y_{n},p_{n},a)\}$
is a test sequence for $Comp$ (test sequence over $Comp$ is defined
now over any sequence of specializations of $Term(Comp)$).

This way, and using the shortening procedure iterative construction,
we can construct a finite diagram $K'$, so that each of its resolutions
terminates in a graded tower (over some $p$-graded limit group $Glim_{t}(b,p,a)$
into which the base group of $Comp$ is mapped - maybe non-compatibly)
of the initial graded resolution $Res(y,p,a)$, and so that the following
property is satisfied. If $\{(x_{n},z_{n},y_{n},p_{n},a)\}$ is a
sequence of specializations for which $\Sigma(x_{n},z_{n},y_{n},p_{n},a)=1$
for all $n$, and for which the sequence $\{(z_{n},y_{n},p_{n},a)\}$
is a test sequence for $Comp$, then the sequence $\{(x_{n},z_{n},y_{n},p_{n},a)\}$
subfactors through $K'$. And in particular, each of the terminal
groups of the resolutions in $K'$ that contains the subgroup $Comp$
- which are graded towers of the resolution $Res(y,p,a)$ as aforementioned
- contains a formal solution $x=x(s,z,y,p,a)$ of the system $\Sigma(x,y,p,a)$,
where the $s$'s stand for the roots that were added for abelian groups
along the decomposition of $Comp$. 
\end{rem}

In order to complete the construction of the graded formal MR diagram
over $F_{k}$, in contrast of the construction of the auxiliary graded
formal MR diagram over random groups, we further extend the resolutions
of $K$. For every resolution in $K$ over a graded formal closure
$M$, we perform the following extension. We consider the subcollection
$\mathfrak{BJ}_{M}$ of $\mathfrak{J}_{M}$ consisting of all the
sequences $\{(x_{n},z_{n},y_{n},p_{n},a)\}$, that satisfies the following: 
\begin{enumerate}
\item If $Term(M)$ is neither $p$-rigid nor $p$-solid, then, for all
$n$, the specialization $(x_{n},z_{n},y_{n},p_{n},a)$ restricts
to a specialization of $Term(M)$ that cannot be shortened by precomposing
with a modular automorphism induced from the $p$-JSJ of $Term(M)$. 
\item If $Term(M)$ is $p$-rigid or $p$-solid, then, for all $n$, the
restriction of the specialization $(x_{n},z_{n},y_{n},p_{n},a)$ factors
through some flexible quotient of $Term(M)$. 
\end{enumerate}
The sequences in $\mathfrak{BJ}_{M}$ factor through finitely many
maximal limit groups that we denote by $H$. According to \cite{DGI},
Lemma10.4, each of the limit quotients $H$ are proper quotients of
$M$. We continue the construction iteratively. According to the descending
chain condition of limit groups, this construction must terminate
after finitely many steps. The result is a diagram $K_{M}$ with finitely
many resolutions, each of which terminates in a graded formal closure
$M'(x,z,y,p,a)$, whose terminal group is either a $p$-rigid or a
$p$-solid limit group.

Note that factorization through a resolution in the diagram $K_{M}$
is defined using the $p$-JSJ's of the base groups of the graded formal
closures along that resolution (in addition to the canonical maps,
standard MR diagrams, and the $Comp$-JSJ's of the corresponding free
factors along the resolution). 
\begin{lem}
\label{lem:23} Let $t_{0}$ be a specialization of $Term(M)$, the
$p$-graded base group of $M$, so that $t_{0}$ restricts to an exceptional
specialization of $B$, the $p$-graded base group of $Comp(z,y,p,a)$.
Denote by $M_{t_{0}}$ the induced ungraded formal closure, and let
$(x_{0},z_{0},y_{0},p_{0},a)$ be a specialization of $M_{t_{0}}$.

Then, the specialization $(x_{0},z_{0},y_{0},p_{0},a)$ factors through
one of the resolutions $GFRes(x,z,y,p,a)$ of $K_{M}$. Moreover,
if $Term(M)$ is either a $p$-rigid or a $p$-solid group, and $t_{0}$
is not an exceptional specialization of $Term(M)$, then we can assume
that the terminal graded formal closure of $GFRes(x,z,y,p,a)$ is
a proper quotient of $M$. 
\end{lem}

\begin{proof}
Since the $p$-JSJ of $Term(M)$ is compatible with that of $B$,
we may assume without loosing the generality that the given solution
$t_{0}$ is short in the case that $Term(M)$ is neither $p$-rigid,
nor $p$-solid, and that $t_{0}$ factors through a flexible quotient
of $Term(M)$, in the case that it is $p$-rigid or $p$-solid.

We consider the ungraded formal closure $M_{t_{0}}=M(x,z,y,p_{0},a)$,
and the ungraded completion $Comp(z,y,p_{0},a)$, that are obtained
from $M$ and $Comp(z,y,p,a)$ by specializing their base groups $Term(M)$
and $B$ according to the given solution $t_{0}$, respectively.

Let $\{(x_{n},z_{n},y_{n},p_{0},a)\}$ be a test sequence of $M_{t_{0}}$.
That is, the sequence $\{(z_{n},y_{n},p_{0},a)\}$ is a test sequence
of $Comp(z,y,p_{0},a)$, and the specialization $x_{n}$ is the induced
specialization of the element $x$ of $M_{t_{0}}$, for all $n$.

Then, the sequence $\{(x_{n},z_{n},y_{n},p_{0},a)\}$ belongs to $\mathfrak{BJ}_{M}$.
Hence, using a simple induction, the sequence $\{(x_{n},z_{n},y_{n},p_{0},a)\}$
subfactors through the diagram $K_{M}$ by the construction of $K_{M}$.

This implies that there exists a collection of (boundedly many) ungraded
closures $FCl(x,z,y,p_{0},a)$, each of which is obtained by specializing
the terminal group of some graded closure in $K_{M}$ (other than
$M$ itself), so that this collection forms a covering closure of
$M_{t_{0}}$, and the Lemma follows. 
\end{proof}
\begin{lem}
\label{lem:24} Let $\Gamma_{l}$ be a group of level $l$ in the
model. If $l$ is large enough, then the following property is satisfied
in $\Gamma_{l}$.

Let $t_{0}$ be a $\Gamma_{l}$-specialization of $Term(M)$, so that
$t_{0}$ restricts to a $\Gamma_{l}$-exceptional specialization of
$B$. Denote by $M_{t_{0}}$ the induced ungraded formal closure,
and let $(x_{0},z_{0},y_{0},p_{0},a)$ be a $\Gamma_{l}$-specialization
of $M_{t_{0}}$.

Then, the $\Gamma_{l}$-specialization $(x_{0},z_{0},y_{0},p_{0},a)$
factors through one of the resolutions $GFRes(x,z,y,p,a)$ of $K_{M}$.
Moreover, if $Term(M)$ is either a $p$-rigid or a $p$-solid group,
and $t_{0}$ is not a $\Gamma_{l}$-exceptional specialization of
$Term(M)$, then we can assume that the terminal graded formal closure
of $GFRes(x,z,y,p,a)$ is a proper quotient of $M$. 
\end{lem}

\begin{proof}
We lift $(x_{0},z_{0},y_{0},p_{0},a)$ to a specialization $(\tilde{x}_{0},\tilde{z}_{0},\tilde{y}_{0},\tilde{p}_{0},a)$
of $M$ in $F_{k}$, so that in the case that $Term(M)$ is either
$p$-rigid or $p$-solid, and the restriction $t_{0}$ of the $\Gamma_{l}$-specialization
$(x_{0},z_{0},y_{0},p_{0},a)$ to $Term(M)$ is not $\Gamma_{l}$-exceptional,
then the restriction $\tilde{t}_{0}$ of the lift $(\tilde{x}_{0},\tilde{z}_{0},\tilde{y}_{0},\tilde{p}_{0},a)$
to $Term(M)$ is not $F_{k}$-exceptional. Note that the restriction
of $(\tilde{x}_{0},\tilde{z}_{0},\tilde{y}_{0},\tilde{p}_{0},a)$
to $B$ must be an $F_{k}$-exceptional solution of $B$.

Hence, according to \lemref{23}, the lift $(\tilde{x}_{0},\tilde{z}_{0},\tilde{y}_{0},\tilde{p}_{0},a)$
factors through some resolution $GFRes(x,z,y,p,a)$ of $K_{M}$. Thus,
the $\Gamma_{l}$-specialization $(x_{0},z_{0},y_{0},p_{0},a)$ factors
through the resolution $GFRes(x,z,y,p,a)$ too. 
\end{proof}
Finally, We extend each of the resolutions in the diagram $K$ by
their corresponding diagram $K_{M}$. We call the obtained diagram
\emph{the graded formal MR diagram of the system $\Sigma(x,y,p,a)$
w.r.t. the resolution $Res(y,p,a)$}, and we denote it briefly by
$GFMRD$.

\subsection{General Properties of the Graded Formal MR Diagram}
\begin{thm}
\label{thm:25} Let

\[
\forall y\;\exists x\quad\Sigma(x,y,p,a)=1\,\wedge\,\Psi(x,y,p,a)\neq1\,,
\]

be an $AE$-formula. Let $Res(y,p,a)$ be a graded resolution, and
denote its completion by $Comp(z,y,p,a)$. Denote by $B$ the $p$-rigid
or $p$-solid base group of $Comp(z,y,p,a)$. Let $GFMRD$ be the
graded formal MR diagram of the system $\Sigma(x,y,p,a)$ w.r.t. the
resolution $Res(y,p,a)$. Let $p_{0}\in F_{k}$, and let $h_{p_{0}}$
be an exceptional solution of $B$ that maps $p$ to $p_{0}$. Denote
the ungraded resolution obtained from $Res(y,p,a)$ by specializing
its terminal group according to $h_{p_{0}}$ by $Res(y,p_{0},a)$,
and the corresponding ungraded completion by $Comp(z,y,p_{0},a)$.

Assume that the sentence: 
\[
\forall y\in Res(y,p_{0},a)\;\exists x\quad\Sigma(x,y,p_{0},a)=1\,\wedge\,\Psi(x,y,p_{0},a)\neq1\,,
\]
is a truth sentence over $F_{k}$.

Then, there exist some resolutions $GFRes_{1}(x,z,y,p,a),...,GFRes_{m}(x,z,y,p,a)$
in $GFMRD$, with corresponding terminal graded formal closures $GFCl_{1}(x,z,y,p,a),...,GFCl_{m}(x,z,y,p,a)$,
together with some exceptional solutions $g_{p_{0}}^{1},...,g_{p_{0}}^{m}$
for the $p$-rigid or $p$-solid base groups $Term(GFCl_{1}),...,Term(GFCl_{m})$,
so that: 
\begin{enumerate}
\item For all $i=1,...,m$, the solution $g_{p_{0}}^{i}$ restricts to an
exceptional solution of $B$ that belongs to the same exceptional
family of $h_{p_{0}}$. 
\item Every sequence of specializations $\{(x_{n},z_{n},y_{n},p_{0},a)\}$,
for which the sequence $\{(z_{n},y_{n},p_{0},a)\}$ is a test sequence
through $Comp(z,y,p_{0},a)$, and $\Sigma(x_{n},y_{n},p_{0},a)=1\,\wedge\,\Psi(x_{n},y_{n},p_{0},a)\neq1$
for all $n$, subfactors through one of the resolutions $GFRes_{1}(x,z,y,p,a),...,GFRes_{m}(x,z,y,p,a)$. 
\item The set of induced ungraded formal closures $FCl_{1}(x,z,y,p_{0},a),...,FCl_{m}(x,z,y,p_{0},a)$
forms a covering closure for $Comp(z,y,p_{0},a)$. 
\item For all $i=1,...,m$, there exists a formal solution $x_{i}=x_{i}(s,z,y,p,a)\in GFCl_{i}$
that factors through the resolution $GFRes_{i}(x,z,y,p,a)$, so that
the words $\Sigma(x_{i},y,p,a)$ represent the trivial element in
$GFCl_{i}$, and each of the words $\Psi(x_{i},y,p_{0},a)$ is non-trivial
in $FCl_{i}(x,z,y,p_{0},a)$. 
\end{enumerate}
\end{thm}

\begin{proof}
Let $GFRes_{1}(x,z,y,p,a),...,GFRes_{m'}(x,z,y,p,a)$ be a minimal
collection of resolutions in $GFMRD$, through which every sequence
of specializations $\{(x_{n},z_{n},y_{n},p_{0},a)\}$, for which the
sequence $\{(z_{n},y_{n},p_{0},a)\}$ is a test sequence of $Comp(z,y,p_{0},a)$,
and $\Sigma(x_{n},y_{n},p_{0},a)=1\,\wedge\,\Psi(x_{n},y_{n},p_{0},a)\neq1$
for all $n$, subfactors.

For every $i=1,...,m'$, we consider a (boundedly many) specializations

\[
(x_{0}^{i,1},z_{0}^{i,1},y_{0}^{i,1},p_{0},a),...,(x_{0}^{i,r(i)},z_{0}^{i,r(i)},y_{0}^{i,r(i)},p_{0},a)
\]
that factor through $GFRes_{i}(x,z,y,p,a)$, satisfy $\Sigma(x_{0}^{i,s},y_{0}^{i,s},p_{0},a)=1\,\wedge\,\Psi(x_{0}^{i,s},y_{0}^{i,s},p_{0},a)\neq1$,
and so that the restriction $(z_{0}^{i,s},y_{0}^{i,s},p_{0},a)$ is
a specialization of $Comp(z,y,p_{0},a)$, for all $s=1,...,r(i)$.
In light of \thmref[s]{14} and \ref{thm:16}, we can assume that
every sequence of specializations $\{(x_{n},z_{n},y_{n},p_{0},a)\}$,
for which the sequence $\{(z_{n},y_{n},p_{0},a)\}$ is a test sequence
through $Comp(z,y,p_{0},a)$, and $\Sigma(x_{n},y_{n},p_{0},a)=1\,\wedge\,\Psi(x_{n},y_{n},p_{0},a)\neq1$
for all $n$, subfactors through one of the formal resolutions $FRes_{i}^{s}(x,z,y,p_{0},a)$
induced from the sepecializations $(x_{0}^{i,s},z_{0}^{i,s},y_{0}^{i,s},p_{0},a)$. 

For every $s=1,...,r(i)$, the factorization of $(x_{0}^{i,s},z_{0}^{i,s},y_{0}^{i,s},p_{0},a)$
through $GFRes_{i}(x,z,y,p,a)$, induces a formal solution $x_{i}^{s}\in GFCl_{i}$
(which depends also on the specializations of the other free factors
along the resolution), together with an ungraded formal closure $FCl_{i}^{s}(x,z,y,p_{0},a)$
obtained by specializing the terminal group of the graded formal closure
$GFCl_{i}$, $Term(GFCl_{i})$, according to some exceptional solution
$g_{p_{0}}^{i,s}$. Since $\Psi(x_{0}^{i,s},y_{0}^{i,s},p_{0},a)\neq1$,
we deduce point 4.

Finally, if, by contradiction, the set of the obtained formal closures
$\{FCl_{i}^{s}(x,z,y,p_{0},a)\}_{i,s}$ did not form a covering closure
for $Comp(z,y,p_{0},a)$, then there will exist a test sequence of
specializations $\{(z_{n},y_{n},p_{0},a)\}$ of $Comp(z,y,p_{0},a)$,
so that the sepecialization $(z_{n},y_{n},p_{0},a)$ is not covered
by the closures $\{FCl_{i}^{s}(x,z,y,p_{0},a)\}$, for all $n$.

For all $n$, let $x_{n}\in F_{k}$ be such that $\Sigma(x_{n},y_{n},p_{0},a)=1\,\wedge\,\Psi(x_{n},y_{n},p_{0},a)\neq1$.
Then, by construction, the sequence $\{(x_{n},z_{n},y_{n},p_{0},a)\}$
subfactors through one of the formal resolutions $FRes_{i}^{s}(x,z,y,p_{0},a)$.
Hence the sequence $(z_{n},y_{n},p_{0},a)$ subfactors through the
ungraded formal closure $FCl_{i}^{s}$ associated with $FRes_{i}^{s}(x,z,y,p_{0},a)$,
a contradiction. 
\end{proof}
The graded formal MR diagram encodes all the formal solutions in the
following sense. 
\begin{lem}
\label{lem:26} Let $Res(y,p,a)$ be a graded resolution, and denote
its completion by $Comp(z,y,p,a)$. Denote by $B$ the $p$-rigid
or $p$-solid base group of $Comp(z,y,p,a)$. Let $GFMRD$ be the
graded formal MR diagram of the system $\Sigma(x,y,p,a)$ w.r.t. the
resolution $Res(y,p,a)$.

Let $GFCl(b,z,y,p,a)$ be a graded closure of $Comp(z,y,p,a)$, and
denote the $p$-graded base limit group of $GFCl(b,z,y,p,a)$ by $Term(GFCl)$.
Let $x$ be an element of $GFCl(b,z,y,p,a)$ so that the words $\Sigma(x,y,p,a)$
represent the trivial element in $GFCl(b,z,y,p,a)$. Then, for every
specialization of $GFCl(b,z,y,p,a)$, that restricts to an exceptional
specialization of $B$, the induced specialization $(x_{0},z_{0},y_{0},p_{0},a)$
of the tuple $(x,z,y,p,a)$, factors through one of the resolutions
in $GFMRD$. 
\end{lem}

\begin{proof}
The idea of proving the property of the graded formal MR diagram stated
in the lemma, is actually to use the method of its construction.

Let $GFRes_{1}(x,z,y,p,a),...,GFRes_{r}(x,z,y,p,a)$ be all the graded
formal resolutions in $GFMRD$. We consider their completions $Comp(GFRes_{1})(w,x,z,y,p,a),...,Comp(GFRes_{r})(w,x,z,y,p,a)$.
Let

\[
\Omega_{1}(w,x,z,y,p,a),...,\Omega_{r}(w,x,z,y,p,a)
\]
be the respective (finite) sets of defining relations of $Comp(GFRes_{1}),...,Comp(GFRes_{r})$.

Let $g_{p_{0}}$ be a specialization of $Term(GFCl)$, that maps $p$
to some $p_{0}\in F_{k}$, and restricts to an exceptional specialization
of $B$. The solution $g_{p_{0}}$ induces an ungraded closure $GFCl(b,z,y,p_{0},a)$
that is covered by the given graded closure $GFCl(b,z,y,p,a)$. We
emphasize the image of the given element $x$ of $GFCl(b,z,y,p,a)$
in $GFCl(b,z,y,p_{0},a)$, by writing $GFCl(x,b,z,y,p_{0},a)$.

In a similar way for constructing the graded formal MR diagram, from
the whole collection of sequences of specializations $\{(w_{n},x_{n},b_{n},z_{n},y_{n},p_{0},a)\}$,
for which the sequences of restricted specializations $\{(x_{n},b_{n},z_{n},y_{n},p_{0},a)\}$
are test sequences of the ungraded closure $GFCl(x,b,z,y,p_{0},a)$
(i.e., $\{(z_{n},y_{n},p_{0},a)\}$ is a test sequence of $Comp(z,y,p,a)$
so that $(x_{n},b_{n},z_{n},y_{n},p_{0},a)$ is a specialization of
$GFCl(x,b,z,y,p_{0},a)$ for all $n$), and that satisfy one of the
systems $\Omega_{1}(w,x,z,y,p_{0},a),...,\Omega_{r}(w,x,z,y,p_{0},a)$,
we construct the formal MR diagram $D$ of the systems $\Omega_{1}(w,x,z,y,p_{0},a),...,\Omega_{r}(w,x,z,y,p_{0},a)$
w.r.t. the closure $GFCl(x,b,z,y,p_{0},a)$.

Every resolution in $D$, has to terminate in an ungraded formal closure
$GFCl'(w,x,b,z,y,p_{0},a)$ of $GFCl(x,b,z,y,p_{0},a)$, so that the
words in at least one of the systems $\Omega_{1}(w,x,z,y,p_{0},a),...,\Omega_{r}(w,x,z,y,p_{0},a)$
represent the trivial element in $GFCl'(w,x,b,z,y,p_{0},a)$.

Now we explain the reason that the set of all the terminal closures
in the diagram $D$, forms a covering closure for the resolution $GFCl(x,b,z,y,p_{0},a)$.
For that, it suffices to show that every test sequence of $GFCl(x,b,z,y,p_{0},a)$
subfactors through some terminal formal closure $GFCl'(w,x,b,z,y,p_{0},a)$
in $D$.

Indeed, if $\{(x_{n},b_{n},z_{n},y_{n},p_{0},a)\}$ is a test sequence
of $GFCl(x,b,z,y,p_{0},a)$, then $\Sigma(x_{n},z_{n},y_{n},p_{0},a)=1$
for all $n$, and the restricted sequence $\{(z_{n},y_{n},p_{0},a)\}$
is a test sequence of $Comp(z,y,p,a)$. Hence, by the construction
of the graded formal MR diagram $GFMRD$, there exists a sequence
of specializations $\{w_{n_{s}}\}_{s}$, so that the specialization
$(w_{n_{s}},x_{n_{s}},b_{n_{s}},z_{n_{s}},y_{n_{s}},p_{0},a)$ satisfies
one of the systems $\Omega_{1}(w,x,z,y,p,a),...,\Omega_{r}(w,x,z,y,p,a)$,
for all $s$. Hence, by the construction of $D$, the sequence $(x_{n},b_{n},z_{n},y_{n},p_{0},a)$
subfactors through some terminal formal closure $GFCl'(w,x,b,z,y,p_{0},a)$.

Finally, By the construction of $D$, every terminal formal closure
$GFCl'(w,x,b,z,y,p_{0},a)$ comes with some formal solution $w$ for
one of the systems $\Omega_{1}(w,x,z,y,p,a),...,\Omega_{r}(w,x,z,y,p,a)$,
and the lemma follows. 
\end{proof}
\begin{thm}
\label{thm:27} Let

\[
\forall y\;\exists x\quad\Sigma(x,y,p,a)=1\,\wedge\,\Psi(x,y,p,a)\neq1\,,
\]

be an $AE$-formula. Let $Res(y,p,a)$ be a graded resolution, and
denote its completion by $Comp(z,y,p,a)$. Denote by $B$ the $p$-rigid
or $p$-solid base group of $Comp(z,y,p,a)$. Let $GFMRD$ be the
graded formal MR diagram of the system $\Sigma(x,y,p,a)$ w.r.t. the
resolution $Res(y,p,a)$, that was constructed over the free group
$F_{k}$.

Let $\Gamma_{l}$ be a group of level $l$ in the model. If $l$ is
large enough, then the following property is satisfied in $\Gamma_{l}$.

Every sequence $\{(x_{n},z_{n},y_{n},p_{n},a)\}$ of $\Gamma_{l}$-specializations
for which the restricted sequence $\{(z_{n},y_{n},p_{n},a)\}$ is
a $\Gamma_{l}$-test sequence of $Comp(z,y,p,a)$, and $\Sigma(x_{n},y_{n},p_{n},a)=1$
in $\Gamma_{l}$ for all $n$, factors through $GFMRD$.

Moreover, let $p_{0}\in\Gamma_{l}$, and let $h_{p_{0}}$ be a $\Gamma_{l}$-exceptional
solution of $B$ that maps $p$ to $p_{0}$. Denote by $Res(y,p_{0},a)$
the ungraded resolution over $\Gamma_{l}$ obtained from $Res(y,p,a)$
by specializing its terminal group according to $h_{p_{0}}$, and
the corresponding $\Gamma_{l}$-ungraded completion by $Comp(z,y,p_{0},a)$.

Assume that the sentence: 
\[
\forall y\in Res(y,p_{0},a)\;\exists x\quad\Sigma(x,y,p_{0},a)=1\,\wedge\,\Psi(x,y,p_{0},a)\neq1\,,
\]
is a truth sentence over $\Gamma_{l}$.

Then, there exist some resolutions $GFRes_{1}(x,z,y,p,a),...,GFRes_{m}(x,z,y,p,a)$
in $GFMRD$, with corresponding terminal graded formal closures $GFCl_{1}(x,z,y,p,a),...,GFCl_{m}(x,z,y,p,a)$,
together with some $\Gamma_{l}$-exceptional solutions $g_{p_{0}}^{1},...,g_{p_{0}}^{m}$
for the $p$-rigid or $p$-solid base groups $Term(GFCl_{1}),...,Term(GFCl_{m})$,
so that: 
\begin{enumerate}
\item For all $i=1,...,m$, the solution $g_{p_{0}}^{i}$ restricts to a
$\Gamma_{l}$-exceptional solution of $B$ that belongs to the same
$\Gamma_{l}$-exceptional family of $h_{p_{0}}$. 
\item Every sequence of $\Gamma_{l}$-specializations $\{(x_{n},z_{n},y_{n},p_{0},a)\}$,
for which $\{(z_{n},y_{n},p_{0},a)\}$ is a $\Gamma_{l}$-test sequence
through $Comp(z,y,p_{0},a)$, and $\Sigma(x_{n},y_{n},p_{0},a)=1\,\wedge\,\Psi(x_{n},y_{n},p_{0},a)\neq1$
in $\Gamma_{l}$ for all $n$, factors through one of the resolutions
$GFRes_{1}(x,z,y,p,a),...,GFRes_{m}(x,z,y,p,a)$. 
\item The induced $\Gamma_{l}$-ungraded formal closures $FCl_{1}(x,z,y,p_{0},a),...,FCl_{m}(x,z,y,p_{0},a)$
form a $\Gamma_{l}$-covering closure for $Comp(z,y,p_{0},a)$. 
\item For all $i=1,...,m$, there exists a formal solution $x_{i}=x_{i}(s,z,y,p,a)\in GFCl_{i}$
that factors through the resolution $GFRes_{i}(x,z,y,p,a)$, so that
the words $\Sigma(x_{i},y,p,a)$ represent the trivial element in
$GFCl_{i}$, and each of the words $\Psi(x_{i},y,p_{0},a)$ is non-trivial
in $FCl_{i}(x,z,y,p_{0},a)$. 
\end{enumerate}
\end{thm}

\begin{proof}
Denote by $D$ the auxiliary graded formal MR diagram over random
groups of the system $\Sigma(x,y,p,a)$ w.r.t. the graded resolution
$Res(y,p,a)$. Let $(z_{0},y_{0},p_{0},a)$ be a $\Gamma_{l}$-specialization
of $Comp(z,y,p,a)$, and let $x_{0}\in\Gamma_{l}$ so that the $\Gamma_{l}$-specialization
$(x_{0},z_{0},y_{0},p_{0},a)$ factors through a resolution $GFRes(x,z,y,p,a)$
in $D$. Denote by $h_{p_{0}}$ the $\Gamma_{l}$-exceptional solution
of $B$ induced by $(x_{0},z_{0},y_{0},p_{0},a)$. According to the
construction of $D$, the resolution $GFRes(x,z,y,p,a)$ terminates
in a limit group of the form $\langle f\rangle\ast M$, where $\langle f\rangle$
is a free group, and $M=M(x,z,y,p,a)$ is a graded formal closure
for $Comp(z,y,p,a)$. Moreover, there exists a formal solution $x\theta\in\langle f\rangle\ast M$,
that factors through the resolution $GFRes(x,z,y,p,a)$, and a $\Gamma_{l}$-specialization
$(x_{0}',z_{0},y_{0},p_{0},a)$ of $\langle f\rangle\ast M$, so that
$x\theta$ is mapped to $x_{0}$. Denote by $H_{p_{0}}$ the restriction
of $(x_{0}',z_{0},y_{0},p_{0},a)$ to $M$. Mapping $\langle f\rangle$
into the subgroup $F_{k}$ of $M$, we get a formal solution, denoted
again $x\theta$, that belongs to $M$, and that is mapped to $x_{0}$
by $H_{p_{0}}$. Note that, according to the construction of $D$,
$H_{p_{0}}$ restricts to $h_{p_{0}}$ on $B$.

In particular, $x\theta$ is an element of the graded closure $M$
of $Comp$, so that the words $\Sigma(x\theta,z,y,p,a)$ represent
the trivial element in $M$. Thus, according to \lemref{26}, for
every specialization of $M$ in $F_{k}$, that restricts to an exceptional
solution of $B$, the induced specialization of $(x\theta,z,y,p,a)$
factors through $GFMRD$.

We lift $H_{p_{0}}$ to a specialization $\tilde{H}_{p_{0}}$ of $M$
in $F_{k}$. Denote by $\tilde{h}_{p_{0}}$ the restriction of $\tilde{H}_{p_{0}}$
to $B$. In particular, the solution $\tilde{h}_{p_{0}}$ lifts $h_{p_{0}}$,
and the specialization $(\tilde{x}_{0},\tilde{z}_{0},\tilde{y}_{0},\tilde{p}_{0},a)$
in $F_{k}$ that $\tilde{H}_{p_{0}}$ induces for the tuple $(x\theta,z,y,p,a)$,
lifts the $\Gamma_{l}$-specialization $(x_{0},z_{0},y_{0},p_{0},a)$.
Since $h_{p_{0}}$ is a $\Gamma_{l}$-exceptional solution of $B$,
then $\tilde{h}_{p_{0}}$ must be an exceptional solution of $B$
over $F_{k}$.

Thus, according to \lemref{26} the induced specialization $(\tilde{x}_{0},\tilde{z}_{0},\tilde{y}_{0},\tilde{p}_{0},a)$
factors through $GFMRD$. This implies immediately that the original
$\Gamma_{l}$-specialization $(x_{0},z_{0},y_{0},p_{0},a)$ factors
through $GFMRD$. Moreover, if $(\tilde{x}_{0},\tilde{z}_{0},\tilde{y}_{0},\tilde{p}_{0},a)$
factors through the resolution $GFRes(x,z,y,p,a)$ of $GFMRD$ whose
terminal closure is $GFCl(x,z,y,p,a)$, then $(x_{0},z_{0},y_{0},p_{0},a)$
factors through the same resolution, and the induced $\Gamma_{l}$-specialization
$g_{p_{0}}$ of the base $p$-rigid or $p$-solid group $Term(GFCl)$
of $GFCl(x,z,y,p,a)$ restricts to a $\Gamma_{l}$-exceptional specialization
of $B$ that belongs to the same $\Gamma_{l}$-exceptional family
of $h_{p_{0}}$.

If the $\Gamma_{l}$-specialization $g_{p_{0}}$ is not itself a $\Gamma_{l}$-exceptional
specialization of the base $p$-rigid or $p$-solid group $Term(GFCl)$
of $GFCl(x,z,y,p,a)$, then we lift it to a non-exceptional specialization
of $Term(GFCl)$ over $F_{k}$. This replaces the specialization of
$GFCl(x,z,y,p,a)$ induced by the factorization of $(\tilde{x}_{0},\tilde{z}_{0},\tilde{y}_{0},\tilde{p}_{0},a)$
with another one, so that they are the same $\Gamma_{l}$-specializations.
Hence, we can assume that $g_{p_{0}}$ is a $\Gamma_{l}$-exceptional
specialization of $Term(GFCl)$.

Now let $\{(x_{n},z_{n},y_{n},p_{n},a)\}$ be a sequence of $\Gamma_{l}$-specializations,
and assume that $\{(z_{n},y_{n},p_{n},a)\}$ is a $\Gamma_{l}$-test
sequence through $Comp(z,y,p,a)$, and $\Sigma(x_{n},y_{n},p_{0},a)=1$
in $\Gamma_{l}$ for all $n$.

Then, according to \lemref{21}, the sequence $\{(x_{n},z_{n},y_{n},p_{n},a)\}$
subfactors over $\Gamma_{l}$ through some resolution in the auxiliary
graded formal MR diagram over random groups, that was denoted by $D$.
Hence, according to the above argument, the sequence $\{(x_{n},z_{n},y_{n},p_{n},a)\}$
subfactors through $GFMRD$.

The remaining part of the theorem follows identically as the case
over $F_{k}$ (with the exception that we use \thmref[s]{15} and
\ref{thm:17} instead of \thmref[s]{14} and \ref{thm:16}). 
\end{proof}
The following fact is useful when treating a sentence that includes
more than one inequality. 
\begin{lem}
\label{lem:28} Let $Res(y,p,a)$ be a graded resolution, and denote
its completion by $Comp(z,y,p,a)$. Denote by $B$ the $p$-rigid
or $p$-solid base group of $Comp(z,y,p,a)$. Let $GFMRD$ be the
graded formal MR diagram of the system $\Sigma(x,y,p,a)$ w.r.t. the
resolution $Res(y,p,a)$ (constructed over $F_{k}$).

Let $GFRes(x,z,y,p,a)$ be a graded formal resolution in $GFMRD$,
defined over the terminal closure $GFCl(x,z,y,p,a)$.

Let $\psi_{1}(x,y,p,a),...,\psi_{s}(x,y,p,a)$ be a finite collection
of words, and let $x\theta_{1},...,x\theta_{s}$ be formal solutions
that factor through the resolution $GFRes(x,z,y,p,a)$. Let $(x_{0},z_{0},y_{0},p_{0},a)$
be a specialization of $GFCl(x,z,y,p,a)$, so that the induced specialization
$(x\theta_{i,0},y_{0},p_{0},a)$, satisfies the inequality 
\[
\psi_{i}(x,y,p,a)\neq1\,,
\]
for all $i=1,...,s$.

Then, there exists a formal solution $x\theta$, that factor through
the resolution $GFRes(x,z,y,p,a)$, so that the induced specialization
$(x\theta_{0},y_{0},p_{0},a)$, satisfies the inequality $\psi_{i}(x,y,p,a)\neq1$,
for all $i=1,...,s$. 
\end{lem}

\begin{proof}
Let $GFRes(x,z_{0},y_{0},p_{0},a)$ be the ungraded formal resolution
induced from the specialization $(x_{0},z_{0},y_{0},p_{0},a)$ and
covered by the resolution $GFRes(x,z,y,p,a)$. Let $FComp(t,x,z_{0},y_{0},p_{0},a)$
be the completion of this resolution, and let $(t_{n},x_{n},z_{0},y_{0},p_{0},a)$
be a test sequence of $FComp(t,x,z_{0},y_{0},p_{0},a)$.

By the assumption, for all $i$, the formal solution $x\theta_{i}$
factors through $GFRes(x,z,y,p,a)$, so that

\[
\psi_{i}(x\theta_{i,0},y_{0},p_{0},a)\neq1\,,
\]
where $(x\theta_{i,0},z_{0},y_{0},p_{0},a)$ is the specialization
of $(x\theta_{i},z,y,p,a)$ induced by $(x_{0},z_{0},y_{0},p_{0},a)$.
In particular, the word $\psi_{i}(x,y,p,a)$, viewed as an element
of $FComp(t,x,z_{0},y_{0},p_{0},a)$, must be non-trivial, for all
$i$.

Hence, and since $(t_{n},x_{n},z_{0},y_{0},p_{0},a)$ is a test sequence,
for large enough $n$, we must have that

\[
\psi_{i}(x_{n},y_{0},p_{0},a)\neq1\,.
\]
\end{proof}
\begin{lem}
\label{lem:29} Let $Res(y,p,a)$ be a graded resolution, and denote
its completion by $Comp(z,y,p,a)$. Denote by $B$ the $p$-rigid
or $p$-solid base group of $Comp(z,y,p,a)$. Let $GFMRD$ be the
graded formal MR diagram of the system $\Sigma(x,y,p,a)$ w.r.t. the
resolution $Res(y,p,a)$ (constructed over $F_{k}$), and let $GFRes(x,z,y,p,a)$
be a graded formal resolution in $GFMRD$, defined over the terminal
closure $GFCl(x,z,y,p,a)$.

Let $\Gamma_{l}$ be a group of level $l$ in the model. If $l$ is
large enough, then the following property is satisfied in $\Gamma_{l}$.

Let $\psi_{1}(x,y,p,a),...,\psi_{s}(x,y,p,a)$ be a finite collection
of words, and let $x\theta_{1},...,x\theta_{s}$ be formal solutions
that factor through the resolution $GFRes(x,z,y,p,a)$. Let $(x_{0},z_{0},y_{0},p_{0},a)$
be a $\Gamma_{l}$-specialization of $GFCl(x,z,y,p,a)$, so that the
induced $\Gamma_{l}$-specialization $(x\theta_{i,0},y_{0},p_{0},a)$,
satisfies the inequality 
\[
\psi_{i}(x,y,p,a)\neq1\,,
\]
for all $i=1,...,s$.

Then, there exists a formal solution $x\theta$, that factor through
the resolution $GFRes(x,z,y,p,a)$, so that the induced $\Gamma_{l}$-specialization
$(x\theta_{0},y_{0},p_{0},a)$, satisfies the inequality $\psi_{i}(x,y,p,a)\neq1$,
for all $i=1,...,s$. 
\end{lem}

\begin{proof}
Let $GFRes(x,z_{0},y_{0},p_{0},a)$ be the ungraded formal resolution
induced from the $\Gamma_{l}$-specialization $(x_{0},z_{0},y_{0},p_{0},a)$
and covered by the resolution $GFRes(x,z,y,p,a)$. Let $FComp(t,x,z_{0},y_{0},p_{0},a)$
be the completion of this resolution, and let $(t_{n},x_{n},z_{0},y_{0},p_{0},a)$
be a $\Gamma_{l}$-test sequence of $GFComp(t,x,z_{0},y_{0},p_{0},a)$.

By the assumption, for all $i$, the formal solution $x\theta_{i}$
factors through $GFRes(x,z,y,p,a)$, so that

\[
\psi_{i}(x\theta_{i,0},y_{0},p_{0},a)\neq1\,,
\]
where $(x\theta_{i,0},z_{0},y_{0},p_{0},a)$ is the $\Gamma_{l}$-specialization
of $(x\theta_{i},z,y,p,a)$ induced by $(x_{0},z_{0},y_{0},p_{0},a)$.
In particular, the word $\psi_{i}(x,y,p,a)$, viewed as an element
of $FComp(t,x,z_{0},y_{0},p_{0},a)$, must be non-trivial, for all
$i$.

Hence, and since $(t_{n},x_{n},z_{0},y_{0},p_{0},a)$ is a $\Gamma_{l}$-test
sequence, for large enough $n$, we must have that $\psi_{i}(x_{n},y_{0},p_{0},a)\neq1$. 
\end{proof}

\section{The $ExtraPS$ MR Diagram over Random Groups  }\label{sec:ExtraPS-MR-Diagrams}

Along this section, we fix a collection $WPRS$ consisting of $wp$-rigid
and $wp$-solid limit groups

\[
G_{1}(x_{1},w,p,a),,...,G_{s}(x_{s},w,p,a)
\]
($wp$ are the parameter set for the groups $G_{i}$). For all $i=1,...,s$,
we fix an integer $e_{i}\in\{0,...,M(G_{i})\}$, where $M(G_{i})$
is the maximal number of exceptional families of specializations in
$G_{i}$ for a fixed value of the parameters.

We consider a group $PT(y,w,p,a)$ that contains the amalgamated free
product of the groups in $WPRS$ \\

\[
PT(y,w,p,a)\geq\text{\ensuremath{\stackrel[i=1]{s}{\ast_{\langle w,p,a\rangle}}\stackrel[j=1]{e_{i}}{\ast_{\langle w,p,a\rangle}}G_{i}}}
\]
amalgamated over the subgroup $\langle w,p,a\rangle$. We consider
the collection $E$ of all the specializations $(y_{0},w_{0},p_{0},a)$
of $PT(y,w,p,a)$ for which the restriction of $(y_{0},w_{0},p_{0},a)$
to $\stackrel[j=1]{e_{i}}{\ast_{\langle w,p,a\rangle}}G_{i}$ represent
a family of representatives of pairwise distinct exceptional families
of $G_{i}$ corresponding to the value $w_{0}p_{0}$ of the parameters,
for all $i=1,...,s$.

The collection $E$ factors through finitely many maximal limit groups
$PS(y,w,p,a)$. For each $PS(y,w,p,a)$, we consider the subcollection
$EPS$ of $E$ consisting of all the specializations in $E$ that
factor through $PS(y,w,p,a)$. we associate its standard $p$-graded
MR diagram (note the drop of $w$ from the parameter set). We consider
a resolution $Res(y,w,p,a)$ in this diagram (rooted at $PS(y,w,p,a)$),
through which some specialization from $E$ factors.

We consider the graded completion $Comp(z,y,w,p,a)$ of the resolution
$Res(y,w,p,a)$, and we denote the terminal $p$-rigid or $p$-solid
base group of $Comp(z,y,w,p,a)$ by $B$. 
\begin{lem}
\label{lem:30}For all $i=1,...,s$, all the mappings of $G_{i}$
into any level of $Comp(z,y,w,p,a)$ are compatible. That is, if $\bar{G}_{i}$
is one of the canonical images of $G_{i}$ in one of the levels of
$Comp(z,y,w,p,a)$, whose associated $p$-graded decomposition is
$\Pi$, for some $i=1,...,s$, then the graph of groups obtained by
collapsing the minimal subgraph of $\Pi$ containing the subgroup
$\langle w,p,a\rangle$ into the vertex $AP$ (the distinguished vertex
containing the subgroup $\langle p,a\rangle$), is compatible with
the $wp$-JSJ of $G_{i}$, i.e., the minimal subgraph of the obtained
graph containing the image $\bar{G}_{i}$ consists only from parts
of the $wp$-JSJ of $G_{i}$. 
\end{lem}

\begin{proof}
If not, then, using the same argument in the proof of \lemref{18},
we can construct a formal solution(s) in (a covering closure of) $Comp(z,y,w,p,a)$
testifying flexibility of all the specializations of $G_{i}$ that
factor through $Comp(z,y,w,p,a)$, a contradiction. 
\end{proof}
\begin{lem}
Let $\Gamma_{l}$ be a group in the model. The $p$-graded MR diagram
of the groups $PS(y,w,p,a)$ encodes all the $\Gamma_{l}$-specializations
$(y_{0},w_{0},p_{0},a)$ for which for all $i=1,...,s$, the restriction
of $(y_{0},w_{0},p_{0},a)$ to $\stackrel[j=1]{e_{i}}{\ast_{\langle w,p,a\rangle}}G_{i}$
represent a family of representatives of pairwise distinct $\Gamma_{l}$-exceptional
families of $G_{i}$ corresponding to the value $w_{0}p_{0}$ of the
parameters. 
\end{lem}

\begin{proof}
Any lift of $(y_{0},w_{0},p_{0},a)$ factors through this diagram. 
\end{proof}

\subsection{The Construction of $ExtraPS$ MR Diagram over Random Groups }

Let $Comp(z,y,w,p,a)$ be a $p$-graded completion, and assume that
the mappings of the $wp$-rigid or $wp$-solid limit groups $G_{i}$
into $Comp(z,y,w,p,a)$ are compatible, for all $i=1,...,s$ (see
\lemref{30}). We are going to construct the auxiliary $ExtraPS$
MR diagram over ascending sequences of groups in the model, for the
collection $WPRS$ and the resolution $Comp(z,y,w,p,a)$.

We look at the collection $\mathfrak{T}$ of all the sequences of
specializations $\{(u_{n},z_{n},y_{n},w_{n},p_{n},a)\}$ over an ascending
sequence $\{\Gamma_{l_{n}}\}$ of groups in the model, for which: 
\begin{enumerate}
\item The sequence $\{(u_{n},z_{n},y_{n},w_{n},p_{n},a)\}$ restricts to
a $p$-graded test sequence of $Comp(z,y,w,p,a)$ over the ascending
sequence $\{\Gamma_{l_{n}}\}$ of groups in the model. 
\item For all $n$, the specialization $(u_{n},z_{n},y_{n},w_{n},p_{n},a)$
is a specialization of the amalgamated product 
\[
\left(\stackrel[i=1]{s}{\ast_{\langle w,p,a\rangle}}\stackrel[j=1]{f_{i}}{\ast_{\langle w,p,a\rangle}}G_{i}\right)\ast_{\langle w,p,a\rangle}Comp(z,y,w,p,a)\,,
\]
where $f_{i}$ is an integer in $f_{i}\in\{0,1\}$. 
\item For all $n$, the specialization $(u_{n},z_{n},y_{n},w_{n},p_{n},a)$
is \emph{non-collapsed}, i.e., for all $i=1,...,s$, the $f_{i}+e_{i}$
specializations of $G_{i}$ that the specialization $(u_{n},z_{n},y_{n},w_{n},p_{n},a)$
induces, form a family of representatives of pairwise distinct $\Gamma_{l_{n}}$-exceptional
families of $G_{i}$ corresponding to the value $w_{n}p_{n}$ of the
parameters (see \defref[s]{12} and \ref{def:13}). 
\end{enumerate}
We apply the method of constructing the auxiliary graded formal MR
diagram over random groups (explained in \secref{Formal-Solutions-over-Random-Groups}),
on the collection $\mathfrak{T}$.

The sequences in the collection $\mathfrak{T}$, factor through finitely
many limit groups $M(u,z,y,w,p,a)$. The canonical image of the subgroup
$Comp(z,y,p,a)$ in $M$ is denoted briefly by $Comp$. In light of
point 3, the limit group $M$ cannot be freely decomposable w.r.t.
the subgroup $Comp\leq M$. We associate to $M$ its relative $Comp$-JSJ
decomposition. Point 3 implies that for all $i=1,...,s$, the mapping
of $G_{i}$ into $M$ is compatible, i.e., the $Comp$-JSJ of $M$
is compatible with the $wp$-JSJ of $G_{i}$.

If the $Comp$-JSJ of $M$ is trivial, we call the group $M(u,z,y,w,p,a)$
a $Comp$-rigid group, and terminate the construction for this stage.
In this case, the subcollection of $\mathfrak{T}$ consisting of all
the sequences $\{(u_{n},z_{n},y_{n},w_{n},p_{n},a)\}$ that factor
through $M$, is denoted by $\mathfrak{T}_{M}$.

In the case that the $Comp$-JSJ of $M(u,z,y,w,p,a)$ is non-trivial,
we consider the subcollection $\mathfrak{T}_{M}$ of $\mathfrak{T}$,
consisting of all the sequences $\{(u_{n},z_{n},y_{n},w_{n},p_{n},a)\}$
that factor through $M$, and so that for all $n$, the specialization
$\{(u_{n},z_{n},y_{n},w_{n},p_{n},a)\}$ cannot be shortened (in the
metric of $(\Gamma_{l_{n}},a)$, where $a$ is the fixed generating
set of $\Gamma_{l_{n}}$) by precomposing with a modular automorphism
of $M$ that is induced from its $Comp$-JSJ.

The sequences in the collection $\mathfrak{T}_{M}$, factor through
finitely many limit quotients of $M$, that we denote by $H$. If
$H$ is not a proper quotient of $M$, we call $M$ a $Comp$-solid
group, and terminate the construction for this stage.

We continue iteratively. Since all the groups that are obtained along
the construction, are limit groups, then, according to the descending
chain condition of limit groups, the construction must terminate after
finitely many steps. The result is a diagram $D$ that includes finitely
many resolutions. Each of these resolutions terminates in a limit
group $M(u,z,y,w,p,a)$ which is either $Comp$-rigid or $Comp$-solid. 
\begin{lem}
\label{lem:31} Assume that $M(u,z,y,w,p,a)$ is a $Comp$-rigid group
or a $Comp$-solid group. Then, $M$ admits a structure of a graded
closure of $Comp(z,y,w,p,a)$ over a $p$-graded limit group $Term(M)$.
That is, $M$ admits a structure of a $p$-graded tower over a $p$-graded
base group $Term(M)$, so that: 
\begin{enumerate}
\item up to adding roots to abelian groups, this structure is similar to
the one associated with $Comp(z,y,w,p,a)$, 
\item this structure is compatible to the one associated with $Comp(z,y,w,p,a)$ 
\item the canonical image of the $p$-graded base group $B$ of $Comp(z,y,w,p,a)$
lies entirely in $Term(M)$, and $B$ is mapped compatibly into $Term(M)$.
That is, the minimal subgraph of groups of the $p$-JSJ of $Term(M)$
containing the canonical image $\bar{B}$ of $B$, consists only from
parts that are taken from the $p$-JSJ of $B$. 
\end{enumerate}
\end{lem}

\begin{proof}
This was proved in \lemref{19}. 
\end{proof}
In ligh of the previous lemma, every resolution of the diagram $D$
that we have constructed above, is defined over a graded closure of
the graded completion $Comp(z,y,w,p,a)$. And according to the construction
of $D$, every sequence $\{(u_{n},z_{n},y_{n},w_{n},p_{n},a)\}$ of
non-collapsed specializations over an ascending sequence of groups
in the model, that restricts to a test sequence $\{(z_{n},y_{n},w_{n},p_{n},a)\}$
through $Comp(z,y,w,p,a)$ over that ascending sequence of groups
in the model, subfactors through at least one of the resolutions in
$D$.

We call this diagram \emph{the auxiliary $ExtraPS$ MR diagram over
random groups of the collection $WPRS$ w.r.t. the graded resolution
$Comp(z,y,w,p,a)$}. Our aim in this section is to show that the $ExtraPS$
MR diagram over $F_{k}$ of the collection $WPRS$ w.r.t. the graded
resolution $Comp(z,y,w,p,a)$ (which we explain in the next subsection),
encodes all the sequences $\{(u_{n},z_{n},y_{n},w_{n}p_{n},a)\}$
of non-collapsed $\Gamma$-specializations, that restrict to $\Gamma$-test
sequences $\{(z_{n},y_{n},w_{n},p_{n},a)\}$, for a random group $\Gamma$
(in probability $1$). The precise statement formulating this property,
will be given below (\thmref{35}). For that aim, we start with the
following lemma: 
\begin{lem}
Let $\Gamma_{l}$ be a group of level $l$ in the model. If $l$ is
large enough, then every sequence of non-collapsed $\Gamma_{l}$-specializations
$\{(u_{n},z_{n},y_{n},w_{n}p_{n},a)\}$, that restricts to a $\Gamma_{l}$-test
sequence $\{(z_{n},y_{n},w_{n},p_{n},a)\}$ of $Comp(z,y,w,p,a)$,
subfactors through the auxiliary $ExtraPS$ MR diagram over random
groups $D$. 
\end{lem}

\begin{proof}
Assume by contradiction that the statement is false. Then, there will
be an ascending sequence $\{\Gamma_{l_{n}}\}$ of groups in the model,
so that for all $n$, there exists a sequence $\{(u_{m}^{n},z_{m}^{n},y_{m}^{n},w_{m}^{n},p_{m}^{n},a)\}_{m}$
of non-collapsed $\Gamma_{l_{n}}$-specializations, so that the restricted
sequence $\{(z_{m}^{n},y_{m}^{n},w_{m}^{n},p_{m}^{n},a)\}_{m}$ is
a $\Gamma_{l}$-test sequence of $Comp(z,y,w,p,a)$, and the specialization
$(u_{m}^{n},z_{m}^{n},y_{m}^{n},w_{m}^{n},p_{m}^{n},a)$ does not
factor through $D$ for all $m$.

By the definition of a test sequence over an ascending sequence of
groups in the model, there exists a sequence $(m_{n})_{n}$ of integers,
so that the sequence $\{(u_{m_{n}}^{n},z_{m_{n}}^{n},y_{m_{n}}^{n},w_{m_{n}}^{n},p_{m_{n}}^{n},a)\}_{n}$,
which is defined over the ascending sequence $\{\Gamma_{l_{n}}\}$
of groups in the model, consists of non-collapsed specializations,
restricts to a test sequence $\{(z_{m_{n}}^{n},y_{m_{n}}^{n},w_{m_{n}}^{n},p_{m_{n}}^{n},a)\}_{n}$
over the ascending sequence $\{\Gamma_{l_{n}}\}$, but it does not
subfactor through $D$. This contradicts the construction of $D$. 
\end{proof}

\subsection{The Construction of the $ExtraPS$ MR Diagram over Free Groups}

Now we perform the parallel construction over $F_{k}$. We consider
all the sequences of non-collapsed specializations $\{(u_{n},z_{n},y_{n},w_{n}p_{n},a)\}$
(over $F_{k}$), that restrict to test sequences $\{(z_{n},y_{n},w_{n},p_{n},a)\}$
of $Comp(z,y,w,p,a)$. We denote the collection of all such sequences
$\{(u_{n},z_{n},y_{n},w_{n}p_{n},a)\}$, by $\mathfrak{J}$. Applying
the shortening procedure iterative construction on the sequences in
$\mathfrak{J}$, we construct an MR diagram, in a similar way to the
construction that we explained above. The result of this construction
over $F_{k}$ is a finite diagram $K$, each of its resolutions $Res(u,z,y,w,p,a)$
terminates in a graded closure $M(u,z,y,w,p,a)$, and each of the
limit groups along $Res(u,z,y,w,p,a)$ is $Comp$-freely indecomposable,
and admits a $Comp$-JSJ that is compatible with the $wp$-JSJ of
$G_{i}$, for all $i=1,...,s$.

For a resolution $Res(u,z,y,w,p,a)$, we consider its terminal graded
closure $M(u,z,y,w,p,a)$, and denote its $p$-graded base group by
$Term(M)$. We want to extend the resolution $Res(u,z,y,w,p,a)$ into
finitely many resolutions, that terminate in a graded closure over
a $p$-rigid or $p$-solid limit base group. For that, we consider
the subcollection $\mathfrak{J}_{M}$ of $\mathfrak{J}$ consisting
of all the sequences $\{(u_{n},z_{n},y_{n},w_{n}p_{n},a)\}$ for which: 
\begin{enumerate}
\item If $Term(M)$ is either $p$-rigid or $p$-solid, then for all $n$,
the specialization $(u_{n},z_{n},y_{n},w_{n}p_{n},a)$ restricts to
a specialization of $Term(M)$ that factors through some flexible
quotient of $Term(M)$. 
\item Otherwise, for all $n$, the restriction of the specialization $(u_{n},z_{n},y_{n},w_{n}p_{n},a)$
into $Term(M)$ cannot be shortened by precomposing with a modular
automorphism induced by the $p$-JSJ of $Term(M)$. 
\end{enumerate}
The subcollection $\mathfrak{J}_{M}$ factors through finitely many
maximal limit groups $H$. We continue with $H$ by applying the shortening
procedure iterative construction that was presented in the beginning
of the construction. According to the descending chain condition of
limit groups, this construction must terminate after finitely many
steps. The result is a diagram $K_{M}$ with finitely many resolutions,
each of which terminates in a graded closure of $Comp(z,y,w,p,a)$. 
\begin{lem}
\label{lem:32} Let $(u_{0},z_{0},y_{0},w_{0},p_{0},a)$ be a non-collapsed
specialization that factors through $M$. Then, there exists a specialization
$(u_{0}',z_{0},y_{0},w_{0},p_{0},a)$ that factors through $M$ so
that: 
\begin{enumerate}
\item The specializations $(u_{0},z_{0},y_{0},w_{0},p_{0},a)$ and $(u_{0}',z_{0},y_{0},w_{0},p_{0},a)$
restricts to the same modular family of $Term(M)$ (and thus to the
same exceptional family of $B$). 
\item The two families of representatives that the specializations $(u_{0},z_{0},y_{0},w_{0},p_{0},a)$
and $(u_{0}',z_{0},y_{0},w_{0},p_{0},a)$ induce, represent the same
modular families of the group $G_{i}$, for all $i=1,...,s$. 
\item The specialization $(u_{0}',z_{0},y_{0},w_{0},p_{0},a)$ factors through
one of the resolutions of $K_{M}$. 
\end{enumerate}
\end{lem}

\begin{proof}
Let $t_{0}$ be a specialization of $Term(M)$, that maps $p$ to
some $p_{0}$ in $F_{k}$. Assume that $t_{0}$ is short when $Term(M)$
is neither $p$-rigid, nor $p$-solid, and that $t_{0}$ factors through
a flexible quotient of $Term(M)$ otherwise. Let $M_{t_{0}}$ be the
induced ungraded closure.

Let $\Omega$ be a collection of (positive) Diophantine conditions
$\Omega_{1}(q,u,z,y,w,p_{0},a),...,\Omega_{r}(q,u,z,y,w,p_{0},a)$
stating that the specialization $(u,z,y,w,p_{0},a)$ of $M_{t_{0}}$
factors through one of the resolutions of $K_{M}$. And let $\Sigma$
be a collection of (positive) Diophantine conditions $\Sigma_{1}(c,u,z,y,w,p_{0},a),...,\Sigma_{t}(c,u,z,y,w,p_{0},a)$
stating that the specialization $(u,z,y,w,p_{0},a)$ of $M_{t_{0}}$
is collapsed.

We construct two formal MR diagrams $FMRD(\Omega)$ and $FMRD(\Sigma)$
over the ungraded closure $M_{t_{0}}$, the first of the collection
$\Omega$, and the second for the collection $\Sigma$.

Let $(u_{n},z_{n},y_{n},w_{n},p_{0},a)$ be a test sequence of $M_{t_{0}}$.
Then, according to the construction of $K_{M}$, either $(u_{n},z_{n},y_{n},w_{n},p_{0},a)$
contains a subsequence that is collapsed, or it subfactors through
one of the resolutions of $K_{M}$. Hence, the collection of terminal
ungraded closures in the two formal MR diagrams $FMRD(\Omega)$ and
$FMRD(\Sigma)$, forms a covering closure for $M_{t_{0}}$.

Assume that a specialization $(u_{0},z_{0},y_{0},w_{0},p_{0},a)$
of $M_{t_{0}}$ is non-collapsed. Then, it cannot factor through $FMRD(\Sigma)$.
Hence, it must factor through one of the ungraded closures in $FMRD(\Omega)$.
Due to the existence of a formal solution in that ungraded closure,
we conclude that $(u_{0},z_{0},y_{0},w_{0},p_{0},a)$ factors through
one of the resolutions of $K_{M}$.

Finally, if $(u_{0},z_{0},y_{0},w_{0},p_{0},a)$ is a specialization
of $M_{t_{0}'}$, for some specialization $t_{0}'$ of $Term(M)$
that represents the same modular family of $t_{0}$, then there exists
some $u_{0}''$ so that the specialization $(u_{0}'',z_{0},y_{0},w_{0},p_{0},a)$
factors through $M_{t_{0}}$, and the two families of representatives
that the specializations $(u_{0},z_{0},y_{0},w_{0},p_{0},a)$ and
$(u_{0}'',z_{0},y_{0},w_{0},p_{0},a)$ induce, represent the same
modular families of the group $G_{i}$, for all $i=1,...,s$. 
\end{proof}
We continue by extending each of the terminal graded closures in the
diagram $K_{M}$, with a new diagram, in a similar way that $K_{M}$
was constructed. We continue the construction iteratively. According
to the descending chain condition of limit groups, this construction
must terminate after finitely many steps. The result is an MR diagram,
that extends the diagram $K$, and admits finitely many resolutions,
each of which terminates in a graded formal closure $M'(u,z,y,w,p,a)$,
whose terminal group is either a $p$-rigid or a $p$-solid limit
group.

We call the obtained diagram \emph{the $ExtraPS$ MR diagram of the
collection $WPRS$ over the graded resolution $Comp(z,y,w,p,a)$},
and we denote it briefly by $ExtraPSMRD$.

\subsection{General Properties of the $ExtraPS$ MR Diagram}
\begin{thm}
\label{thm:33} Let $WPRS$ be a given collection of $wp$-rigid and
$wp$-solid limit groups $G_{1}(x_{1},w,p,a),,...,G_{s}(x_{s},w,p,a)$,
and let $Comp(z,y,w,p,a)$ be a $p$-graded completion whose base
group is denoted by $B$ which is a $p$-rigid or a $p$-solid limit
group. Assume that the $wp$-rigid or $wp$-solid limit group $G_{i}$
admits $e_{i}\in\{0,...,M(G_{i})\}$ mappings into $Comp(z,y,w,p,a)$,
for all $i=1,...,s$. Assume that there exists a non-collapsed specialization
$(y_{0},w_{0},p_{0},a)$ that factors through $Comp(z,y,w,p,a)$.

Let $ExtraPSMRD$ be the $ExtraPS$ MR diagram of the collection $WPRS$
w.r.t. the resolution $Comp(z,y,w,p,a)$. Let $ExtraPSGCl_{1}(u,z,y,w,p,a),...,ExtraPSGCl_{m}(u,z,y,w,p,a)$
be the terminal graded closures of the various resolutions of $ExtraPSMRD$.

For all $n$, let $(u_{n},z_{n},y_{n},w_{n},p_{n},a)$ be a non-collapsed
specialization of 
\[
\stackrel[i=1]{s}{\ast_{\langle w,p,a\rangle}}\stackrel[j=1]{f_{i}}{\ast_{\langle w,p,a\rangle}}G_{i}\ast_{\langle w,p,a\rangle}Comp(z,y,w,p,a)\,,
\]
where $f_{i}$ is an integer in $f_{i}\in\{0,1\}$, and assume that
the restricted sequence $\{(z_{n},y_{n},w_{n},p_{n},a)\}$ is a test
sequence of $Comp(z,y,w,p,a)$.

Then, for some $k=1,...,m$, and for all $n$ in a subsequence of
the integers, there exists a non-collapsed specialization $(u'_{n},z_{n},y_{n},w_{n},p_{n},a)$,
so that: 
\begin{enumerate}
\item The specializations $(u_{n},z_{n},y_{n},w_{n},p_{n},a)$ and $(u'_{n},z_{n},y_{n},w_{n},p_{n},a)$
restrict to solutions of $B$ that represent the same exceptional
family. 
\item The two families of representatives that $(u_{n},z_{n},y_{n},w_{n},p_{n},a)$
and $(u'_{n},z_{n},y_{n},w_{n},p_{n},a)$ induce, represent the same
exceptional families of the group $G_{i}$, for all $i=1,...,s$. 
\item The specialization $(u'_{n},z_{n},y_{n},w_{n},p_{n},a)$ factors through
the graded closure $ExtraPSGCl_{k}(u,z,y,w,p,a)$, so that the restriction
of $(u'_{n},z_{n},y_{n},w_{n},p_{n},a)$ to $Term(ExtraPSGCl_{k})$
is an exceptional solution for it. 
\end{enumerate}
\end{thm}

\begin{proof}
By the construction of $ExtraPSMRD$, the sequence $(u_{n},z_{n},y_{n},w_{n},p_{n},a)$
subfactors through some resolution of $ExtraPSMRD$, defined over
some graded closure $M(u,z,y,w,p,a)$ in the diagram $ExtraPSMRD$.
The remaining of the statement follows from \lemref{30} and \lemref{31}. 
\end{proof}
The $ExtraPS$ MR diagram encodes all the non-collapsed specializations
of any graded closure of $Comp(z,y,w,p,a)$, in the following sense: 
\begin{lem}
\label{lem:34} Let $GFCl(b,z,y,w,p,a)$ be a graded closure of $Comp(z,y,w,p,a)$,
and denote the $p$-graded base limit group of $GFCl(b,z,y,w,p,a)$
by $Term(GFCl)$. Assume that the amalgamated free product 
\[
\stackrel[i=1]{s}{\ast_{\langle w,p,a\rangle}}\stackrel[j=1]{f_{i}}{\ast_{\langle w,p,a\rangle}}G_{i}\ast_{\langle w,p,a\rangle}Comp(z,y,w,p,a)\,,
\]
where $f_{i}$ is an integer in $f_{i}\in\{0,1\}$, is mapped into
$GFCl(b,z,y,w,p,a)$, so that the subgroup $Comp(z,y,w,p,a)$ is mapped
canonically. Let $u$ be the images of the variables $x_{1},...,x_{s}$
of the given $wp$-rigid and $wp$-solid groups $G_{1}(x_{1},w,p,a),,...,G_{s}(x_{s},w,p,a)$.

Then, for every specialization of $GFCl(b,z,y,w,p,a)$, that restricts
to an exceptional specialization of $B$, if the induced specialization
$(u_{0},z_{0},y_{0},w_{0},p_{0},a)$ of the tuple $(u,z,y,w,p,a)$
is non-collapsed, then $(u_{0},z_{0},y_{0},w_{0},p_{0},a)$ factors
through one of the resolutions in the $ExtraPS$ MR diagram $ExtrPSMRD$. 
\end{lem}

\begin{proof}
Let $g_{p_{0}}$ be a specialization of $Term(GFCl)$, that maps $p$
to some $p_{0}\in F_{k}$, and restricts to an exceptional specialization
of $B$. The solution $g_{p_{0}}$ induces an ungraded closure $GFCl_{g_{p_{0}}}$
that is covered by the given graded closure $GFCl(b,z,y,w,p,a)$.
Note that $GFCl_{g_{p_{0}}}$ is generated by the (image of the) elements
$b,z,y,w,a$. In particular the element $u\in GFCl_{g_{p_{0}}}$ is
generated by these elements, $u=u(b,z,y,w,a)$.

Let $\Omega$ be a collection of (positive) Diophantine conditions

\[
\Omega_{1}(q,u(b,z,y,w,a),z,y,w,p_{0},a),...,\Omega_{r}(q,u(b,z,y,w,a),z,y,w,p_{0},a)
\]
stating that the specialization $(u(b,z,y,w,a),z,y,w,p_{0},a)$ factors
through the $ExtraPS$ MR diagram $ExtraPSMRD$. And let $\Sigma$
be a collection of (positive) Diophantine conditions

\[
\Sigma_{1}(c,u(b,z,y,w,a),z,y,w,p_{0},a),...,\Sigma_{t}(c,u(b,z,y,w,a),z,y,w,p_{0},a)
\]
stating that the specialization $(u(b,z,y,w,a),z,y,w,p_{0},a)$ is
collapsed.

We construct two formal MR diagrams $FMRD(\Omega)$ and $FMRD(\Sigma)$
over the ungraded closure $GFCl_{g_{p_{0}}}$, the first of the collection
$\Omega$, and the second for the collection $\Sigma$.

Let $\{(b_{n},z_{n},y_{n},w_{n},p_{0},a)\}$ be a test sequence of
$GFCl_{g_{p_{0}}}$, and consider the induced specializations $u_{n}$
of the element $u\in GFCl_{g_{p_{0}}}$. If for all $n$, the specialization
$(u_{n},z_{n},y_{n},w_{n},p_{0},a)$ is non-collapsed, then, according
to the definition of the $ExtraPS$ MR diagram $ExtrPSMRD$, the sequence
$\{(u_{n},z_{n},y_{n},w_{n},p_{0},a)\}$ subfactors through $ExtrPSMRD$.
Hence, in this case, the sequence $\{(b_{n},z_{n},y_{n},w_{n},p_{0},a)\}$
subfactors through $FMRD(\Omega)$.

If for all $n$, the specialization $(u_{n},z_{n},y_{n},w_{n},p_{0},a)$
is collapsed, then the sequence $\{(b_{n},z_{n},y_{n},w_{n},p_{0},a)\}$
subfactors through $FMRD(\Sigma)$.

We deduce that the collection of formal closures of $GFCl_{g_{p_{0}}}$,
that are terminal closures in the diagrams $FMRD(\Omega)$ and $FMRD(\Sigma)$,
form a covering closure for $GFCl_{g_{p_{0}}}$.

Finally, let $(b_{0},z_{0},y_{0},w_{0},p_{0},a)$ be a specialization
of $GFCl_{g_{p_{0}}}$, and assume that the induced specialization
$(u_{0},z_{0},y_{0},w_{0},p_{0},a)$ is non-collapsed. Then, due to
the formal solutions in the terminal closures of $FMRD(\Sigma)$,
the specialization $(b_{0},z_{0},y_{0},w_{0},p_{0},a)$ cannot factor
through a terminal closure of $FMRD(\Sigma)$. Hence, $(b_{0},z_{0},y_{0},w_{0},p_{0},a)$
factors through a terminal closure of $FMRD(\Omega)$. The formal
solution in that closure, testifies that the specialization $(u_{0},z_{0},y_{0},w_{0},p_{0},a)$
factors through the $ExtraPS$ MR diagram $ExtrPSMRD$. As required. 
\end{proof}
\begin{thm}
\label{thm:35} Let $WPRS$ be a given collection of $wp$-rigid and
$wp$-solid limit groups $G_{1}(x_{1},w,p,a),,...,G_{s}(x_{s},w,p,a)$,
and let $Comp(z,y,w,p,a)$ be a $p$-graded completion whose base
group is denoted by $B$ which is a $p$-rigid or a $p$-solid limit
group. Assume that the $wp$-rigid or $wp$-solid limit group $G_{i}$
admits $e_{i}\in\{0,...,M(G_{i})\}$ mappings into $Comp(z,y,w,p,a)$,
for all $i=1,...,s$. 

Let $\Gamma_{l}$ be a group of level $l$ in the model. If $l$ is
large enough, then the following properties are satisfied in $\Gamma_{l}$.

Assume that there exists a non-collapsed $\Gamma_{l}$-specialization
$(y_{0},w_{0},p_{0},a)$ that factors through $Comp(z,y,w,p,a)$.
Then there exists a non-collapsed $F_{k}$-specialization $(\tilde{y}_{0},\tilde{w}_{0},\tilde{p}_{0},a)$
that factors through the resolution $Comp(z,y,w,p,a)$.

Moreover, let $ExtraPSMRD$ be the $ExtraPS$ MR diagram of the collection
$WPRS$ w.r.t. the resolution $Comp(z,y,w,p,a)$. Let $ExtraPSGCl_{1}(u,z,y,w,p,a),...,ExtraPSGCl_{m}(u,z,y,w,p,a)$
be the terminal graded closures of the various resolutions of $ExtraPSMRD$.

For all $n$, let $(u_{n},z_{n},y_{n},w_{n},p_{n},a)$ be a non-collapsed
$\Gamma_{l}$-specialization of 
\[
\stackrel[i=1]{s}{\ast_{\langle w,p,a\rangle}}\stackrel[j=1]{f_{i}}{\ast_{\langle w,p,a\rangle}}G_{i}\ast_{\langle w,p,a\rangle}Comp(z,y,w,p,a)\,,
\]
where $f_{i}$ is an integer in $f_{i}\in\{0,1\}$, so that the restricted
sequence $(z_{n},y_{n},w_{n},p_{n},a)$ is a $\Gamma_{l}$-test sequence
of $Comp(z,y,w,p,a)$.

Then, for some $k=1,...,m$, and for all $n$ in a subsequence of
the integers, there exists a non-collapsed $\Gamma_{l}$-specialization
$(u'_{n},z_{n},y_{n},w_{n},p_{n},a)$, so that: 
\begin{enumerate}
\item The $\Gamma_{l}$-specializations $(u_{n},z_{n},y_{n},w_{n},p_{n},a)$
and $(u'_{n},z_{n},y_{n},w_{n},p_{n},a)$ restrict to solutions of
$B$ that represent the same $\Gamma_{l}$-exceptional family. 
\item The two families of representatives that $(u_{n},z_{n},y_{n},w_{n},p_{n},a)$
and $(u'_{n},z_{n},y_{n},w_{n},p_{n},a)$ induce, represent the same
$\Gamma_{l}$-modular families of the group $G_{i}$, for all $i=1,...,s$. 
\item The $\Gamma_{l}$-specialization $(u'_{n},z_{n},y_{n},w_{n},p_{n},a)$
factors through the graded closure $ExtraPSGCl_{k}(u,z,y,w,p,a)$,
so that the restriction of $(u'_{n},z_{n},y_{n},w_{n},p_{n},a)$ to
$Term(ExtraPSGCl_{k})$ is a $\Gamma_{l}$-exceptional solution for
it. 
\end{enumerate}
\end{thm}

\begin{proof}
The first part of the claim follows immediately from that if $(\tilde{y}_{0},\tilde{w}_{0},\tilde{p}_{0},a)$
is any lift of $(y_{0},w_{0},p_{0},a)$ that factors through $Comp(z,y,w,p,a)$,
then, $(\tilde{y}_{0},\tilde{w}_{0},\tilde{p}_{0},a)$ is non-collapsed
(over $F_{k}$).

Denote by $D$ the auxiliary $ExtraPS$ MR diagram of the collection
$WPRS$ w.r.t. the resolution $Comp(z,y,w,p,a)$. Let $(u_{0},z_{0},y_{0},w_{0},p_{0},a)$
be a $\Gamma_{l}$-specialization of $Comp(z,y,w,p,a)$, that factors
through a resolution $GRes(u,z,y,w,p,a)$ in $D$. Denote by $h_{p_{0}}$
the $\Gamma_{l}$-exceptional solution of $B$ induced by $(u_{0},z_{0},y_{0},w_{0},p_{0},a)$.
According to the construction of $D$, the resolution $GRes(u,z,y,w,p,a)$
terminates in a graded closure $M(u,z,y,w,p,a)$. Moreover, there
exists an element $u\theta\in M$, that factors through the resolution
$GFRes(u,z,y,w,p,a)$ as the image of $u$, and a $\Gamma_{l}$-specialization
$H_{p_{0}}=(u_{0}',z_{0},y_{0},w_{0},p_{0},a)$ of $M$, so that $u\theta$
is mapped to $u_{0}$. Note that, according to the construction of
$D$, $H_{p_{0}}$ restricts to $h_{p_{0}}$ on $B$.

In particular, the amalgamated free product 
\[
\stackrel[i=1]{s}{\ast_{\langle w,p,a\rangle}}\stackrel[j=1]{f_{i}}{\ast_{\langle w,p,a\rangle}}G_{i}\ast_{\langle w,p,a\rangle}Comp(z,y,w,p,a)\,,
\]
 where $f_{i}\in\{0,1\}$ are the corresponding integers, is mapped
into the graded closure $M$, so that the subgroup $Comp(z,y,w,p,a)$
is mapped canonically, and the elements $u\theta$ are the images
of the variables $x_{1},...,x_{s}$ of the given groups $G_{1}(x_{1},w,p,a),,...,G_{s}(x_{s},w,p,a)$.

Thus, according to \lemref{34}, for every specialization of $M$
in $F_{k}$, that restricts to an exceptional solution of $B$, the
induced specialization of $(u\theta,z,y,w,p,a)$ factors through $ExtraPSMRD$.

We lift $H_{p_{0}}$ to a specialization $\tilde{H}_{p_{0}}$ of $M$
in $F_{k}$. Denote by $\tilde{h}_{p_{0}}$ the restriction of $\tilde{H}_{p_{0}}$
to $B$. In particular, the solution $\tilde{h}_{p_{0}}$ lifts $h_{p_{0}}$,
and the specialization $(\tilde{u}_{0},\tilde{z}_{0},\tilde{y}_{0},\tilde{w}_{0},\tilde{p}_{0},a)$
in $F_{k}$ that $\tilde{H}_{p_{0}}$ induces for the tuple $(u\theta,z,y,w,p,a)$,
lifts the $\Gamma_{l}$-specialization $(u_{0},z_{0},y_{0},w_{0},p_{0},a)$.
Since $h_{p_{0}}$ is a $\Gamma_{l}$-exceptional solution of $B$,
then $\tilde{h}_{p_{0}}$ must be an exceptional solution of $B$
over $F_{k}$. Thus, the induced specialization $(\tilde{u}_{0},\tilde{z}_{0},\tilde{y}_{0},\tilde{w}_{0},\tilde{p}_{0},a)$
factors through $ExtraPSMRD$. This implies immediately that the original
$\Gamma_{l}$-specialization $(u_{0},z_{0},y_{0},w_{0},p_{0},a)$
factors through $ExtraPSMRD$. Moreover, if $(\tilde{u}_{0},\tilde{z}_{0},\tilde{y}_{0},\tilde{w}_{0},\tilde{p}_{0},a)$
factors through the resolution $ExtraPSGRes(u,z,y,w,p,a)$ of $ExtraPSMRD$
whose terminal closure is\\
 $ExtraPSGCl(u,z,y,w,p,a)$, then $(u_{0},z_{0},y_{0},w_{0},p_{0},a)$
factors through the same resolution, and the induced $\Gamma_{l}$-specialization
$g_{p_{0}}$ of the base $p$-rigid or $p$-solid group $Term(ExtraPSGCl)$
of $ExtraPSGCl(u,z,y,w,p,a)$ restricts to a $\Gamma_{l}$-exceptional
specialization of $B$ that belongs to the same $\Gamma_{l}$-exceptional
family of $h_{p_{0}}$.

If the $\Gamma_{l}$-specialization $g_{p_{0}}$ is not itself a $\Gamma_{l}$-exceptional
specialization of the base $p$-rigid or $p$-solid group $Term(ExtraPSGCl)$
of $ExtraPSGCl(u,z,y,w,p,a)$, then we lift it to a non-exceptional
specialization of $Term(ExtraPSGCl)$ over $F_{k}$. This replaces
the non-collapsed specialization $ExtraPSGCl(u,z,y,w,p,a)$ induced
by the factorization of $(\tilde{u}_{0},\tilde{z}_{0},\tilde{y}_{0},\tilde{w}_{0},\tilde{p}_{0},a)$
with another non-collapsed one, so that they are the same $\Gamma_{l}$-specializations.
Hence, we can assume that $g_{p_{0}}$ is a $\Gamma_{l}$-exceptional
specialization of $Term(ExtraPSGCl)$.

Now let $\{(u_{n},z_{n},y_{n},w_{n},p_{n},a)\}$ be a sequence of
non-collapsed $\Gamma_{l}$-specializations, and assume that $\{(z_{n},y_{n},w_{n},p_{n},a)\}$
is a $\Gamma_{l}$-test sequence through $Comp(z,y,w,p,a)$. Then,
the sequence $\{(u_{n},z_{n},y_{n},w_{n},p_{n},a)\}$ subfactors through
some resolution in the auxiliary $ExtraPS$ MR diagram over random
groups, that was denoted by $D$. Hence, according to the above argument,
the sequence $\{(u_{n},z_{n},y_{n},w_{n},p_{n},a)\}$ subfactors through
$ExtraPSMRD$.

The remaining part of the theorem follows identically as the case
over $F_{k}$. 
\end{proof}

\section{The Construction of the CollapseExtra MR Diagram}\label{sec:CollapseExtra-MR-Diagram}

Let $ExtraPSGCl(u,z,y,w,p,a)$ be one of the terminal graded closures
in the $ExtraPS$ MR diagram $ExtraPSMRD$ of the collection $WPRS$
over the $p$-graded resolution $Res(y,w,p,a)$ (constructed in \secref{ExtraPS-MR-Diagrams}).
Let $Comp(z,y,w,p,a)$ be the graded completion of $Res(y,w,p,a)$,
and denote its terminal $p$-rigid or $p$-solid base group by $B$.
Denote the $p$-graded base group of $ExtraPSGCl(u,z,y,w,p,a)$, which
is either $p$-rigid or $p$-solid, by $Term(ExtraPSGCl)$.

Let $\Sigma_{i}(c,u,z,y,w,p,a)$, $i=1,...,r$, be the finitely many
systems that testify that the specialization $(u,z,y,w,p,a)$ is collapsed.
We are going to construct a graded formal MR diagram of the systems
$\Sigma_{i}(c,u,z,y,w,p,a)$ over the graded closure $ExtraPSGCl(u,z,y,w,p,a)$
(over $F_{k}$).

We consider the collection $\mathfrak{J}$ of all the sequences of
specializations $\{(c_{n},u_{n},z_{n},y_{n},w_{n},p_{n},a)\}$, for
which the restricted sequence $\{(u_{n},z_{n},y_{n},w_{n},p_{n},a)\}$
is a sequence of specializations of $ExtraPSGCl(u,z,y,w,p,a)$, the
restricted sequence $\{(z_{n},y_{n},w_{n},p_{n},a)\}$ is a test sequence
of $Comp(z,y,w,p,a)$, and for all $n$: 
\begin{enumerate}
\item The restricted specialization $(y_{n},w_{n},p_{n},a)$ is non-collapsed. 
\item The specialization $(c_{n},u_{n},z_{n},y_{n},w_{n},p_{n},a)$ restricts
to an exceptional solution of the base group \\
 $Term(ExtraPSGCl)$ of $ExtraPSGCl$, and to an exceptional solution
of $B$. 
\item The specialization $c_{n}$ explains how the specialization $(c_{n},u_{n},z_{n},y_{n},w_{n},p_{n},a)$
is collapsed, i.e., the elements of one of the systems $\Sigma_{i}(c_{n},u_{n},z_{n},y_{n},w_{n},p_{n},a)$
represent the trivial word. 
\end{enumerate}
The sequences in the collection $\mathfrak{J}$, factor through finitely
many limit groups $H(c,u,z,y,w,p,a)$. The canonical image of the
subgroup $Comp(z,y,w,p,a)$ in $H$ is denoted briefly by $Comp$,
and the canonical image of the subgroup $ExtraPSGCl(u,z,y,w,p,a)$
in $H$ is denoted by $ExtraPSGCl$. The limit group $H$ could be
freely decomposable w.r.t. the subgroup $Comp\leq H$. We consider
the most refined $Comp$-free factorization of $H$. We keep the free
factors that do not include (a conjugate of) the subgroup $Comp\leq H$,
extend them by their standard MR diagrams, and continue the construction
with the free factor that contains $Comp$. We denote this factor
by $M(c,u,z,y,w,p,a)$, which is a limit group. We consider the $Comp$-JSJ
decomposition of the limit group $M(c,u,z,y,w,p,a)$, i.e., the JSJ
decomposition with respect to the subgroup $Comp$. 
\begin{lem}
\label{lem:36} The canonical image $ExtraPSGCl$ of $ExtraPSGCl(u,z,y,w,p,a)$
is contained in the free factor $M$ of $H$ containing $Comp$. Moreover,
the $Comp$-JSJ of $M$ is compatible with the $p$-JSJ of $Term(ExtraPSGCl)$. 
\end{lem}

\begin{proof}
According to the definition of $H$ and the sequences in $\mathfrak{J}$,
there exists a specialization $(c_{0},u_{0},z_{0},y_{0},w_{0},p_{0},a)$
of $H$ that restricts to an exceptional specialization of $Term(ExtraPSGCl)$.
Hence, the claim follows using \lemref{18}. 
\end{proof}
If the $Comp$-JSJ of $M(c,u,z,y,w,p,a)$ is trivial, we call the
group $M(c,u,z,y,w,p,a)$ a $Comp$-rigid group, and terminate the
construction for this stage. In this case, the subcollection of $\mathfrak{J}$
consisting of all the sequences $(c_{n},u_{n},z_{n},y_{n},w_{n},p_{n},a)$
that factor through $M$, is denoted by $\mathfrak{J}_{M}$.

In the case that the $Comp$-JSJ of $M(c,u,z,y,w,p,a)$ is non-trivial,
we consider the subcollection $\mathfrak{J}_{M}$ of $\mathfrak{J}$,
consisting of all the sequences $(c_{n},u_{n},z_{n},y_{n},w_{n},p_{n},a)$
that factor through $M$, and so that for all $n$, the specialization
$(c_{n},u_{n},z_{n},y_{n},w_{n},p_{n},a)$ cannot be shortened (in
the metric of $(F_{k},a)$) by precomposing with a modular automorphism
of $M$ that is induced from its $Comp$-JSJ. The sequences in the
collection $\mathfrak{J}_{M}$, factor through finitely many maximal
limit quotients of $M$, that we denote - while abusing the notation
- by $H$. If $H$ is not a proper quotient of $M$, we call $M$
a $Comp$-solid group, and terminate the construction for this stage. 
\begin{lem}
Let $\{(c_{n},u_{n},z_{n},y_{n},w_{n},p_{n},a)\}$ be a sequence in
the collection $\mathfrak{J}$ that factors through $M$. Then, the
shortened sequence $\{((c\theta)_{n},(u\theta)_{n},z_{n},y_{n},w_{n},p_{n},a)\}$
belongs to $\mathfrak{J}$. 
\end{lem}

\begin{proof}
For all $n$, let $((c\theta)_{n},(u\theta)_{n},z_{n},y_{n},w_{n},p_{n},a)$
be a specialization of $M$ that belongs to the same modular family
of $(c_{n},u_{n},z_{n},y_{n},w_{n},p_{n},a)$. According to \lemref{36},
and according to the construction of the $ExtraPS$ MR diagram, the
two families of representatives, that the specializations $(c_{n},u_{n},z_{n},y_{n},w_{n},p_{n},a)$
and \\
 $((c\theta)_{n},(u\theta)_{n},z_{n},y_{n},w_{n},p_{n},a)$ induce,
represent the same exceptional families for the group $G_{i}\in WPRS$,
for all $i=1,...,s$. 
\end{proof}
We continue the construction iteratively. Since all the groups that
are obtained along the construction, are limit groups, then, according
to the descending chain condition of limit groups, the construction
must terminate after finitely many steps. The result is a diagram
$K$ that includes finitely many resolutions. Each of these resolutions
terminates in a free product $\langle f\rangle\ast M$, where $\langle f\rangle$
is a free group, and $M$ is either a $Comp$-rigid or a $Comp$-solid
group. 
\begin{lem}
Assume that $M(c,u,z,y,w,p,a)$ is a $Comp$-rigid group or a $Comp$-solid
group. Then, $M$ admits a structure of a graded formal closure of
$Comp(z,y,w,p,a)$ over a $p$-graded limit group $Term(M)$. That
is, $M$ admits a structure of a $p$-graded tower over a $p$-graded
base group $Term(M)$, so that: 
\begin{enumerate}
\item up to adding roots to abelian groups, this structure is similar to
the one associated with $Comp(z,y,w,p,a)$, 
\item this structure is compatible to the one associated with $Comp(z,y,w,p,a)$ 
\item the canonical image of the $p$-graded base group $B$ of $Comp(z,y,w,p,a)$
lies entirely in $Term(M)$, and $B$ is mapped compatibly into $Term(M)$.
That is, the minimal subgraph of groups of the $p$-JSJ of $Term(M)$
containing the canonical image $\bar{B}$ of $B$, consists only from
parts that are taken from the $p$-JSJ of $B$. 
\end{enumerate}
Moreover, the mapping of $Term(ExtraPSGCl)$ into any level of $M$
is compatible. 
\end{lem}

\begin{proof}
The proof of the first part of the claim is done in the proof of \lemref{20}.

If the last statement was wrong, then, using the same argument in
the proof of \lemref{18}, we can construct a formal solution(s) in
(a covering closure of) $Comp(z,y,w,p,a)$ testifying flexibility
of all the specializations of $Term(ExtraPSGCl)$ that factor through
$M$. 
\end{proof}
\begin{lem}
\label{lem:37} Let $t_{0}$ be a specialization of $Term(M)$, the
$p$-graded base group of $M$. Consider the induced ungraded closure
$M_{t_{0}}$, and assume that there exists a specialization $(c_{0},u_{0},z_{0},y_{0},w_{0},p_{0},a)$
that factors through the ungraded closure $M_{t_{0}}$ and satisfies
the conditions 1-3 in the top of the section.

Then, there exists a test sequence $\{(z_{n},y_{n},w_{n},p_{0},a)\}$
of $Comp(z,y,w,p,a)$, so that for every specialization $t_{0}'$
of $Term(M)$ that belongs to the same modular family of $t_{0}$,
the specialization $(z_{n},y_{n},w_{n},p_{0},a)$ factors through
the ungraded closure $M_{t_{0}'}$, and the induced specialization
$(c_{n}',u_{n}',z_{n},y_{n},w_{n},p_{0},a)$ satisfies the conditions
1-3, for all $n$. 
\end{lem}

\begin{proof}
Since $Term(ExtraPSGCl)$ is mapped compatibly into $M$, and $(c_{0},u_{0},z_{0},y_{0},w_{0},p_{0},a)$
is a specialization of $M_{t_{0}}$ that restricts to an exceptional
solution for $Term(ExtraPSGCl)$ and $B$, we conclude that every
specialization of $M_{t_{0}'}$, where $t_{0}'$ is a specialization
of $Term(M)$ in the same modular family of $t_{0}$, satisfies the
conditions 2-3 in the beginning of the section.

Now a similar argument to the proof of \lemref{32}, implies that
there must exist some test sequence $\{(z_{n},y_{n},w_{n},p_{0},a)\}$
of $Comp(z,y,w,p,a)$ consisting of non-collapsed specializations,
that factor through $M_{t_{0}}$. Hence, if $t_{0}'$ is a specialization
of $Term(M)$ that belongs to the same modular family of $t_{0}$,
the specialization $(z_{n},y_{n},w_{n},p_{0},a)$ factors through
$M_{t_{0}'}$, and the induced specialization $(c_{n}',u_{n}',z_{n},y_{n},w_{n},p_{0},a)$
satisfies the conditions 1-3, for all $n$. 
\end{proof}
We continue the construction of our graded formal MR diagram. For
a resolution $GFRes(c,u,z,y,w,p,a)$ in the diagram $K$ constructed
so far, we consider its terminal graded formal closure $M(c,u,z,y,w,p,a)$,
and denote its $p$-graded base group by $Term(M)$. We want to extend
the resolution $GFRes(c,u,z,y,w,p,a)$ into finitely many resolutions,
that terminate in a graded closure over a $p$-rigid or a $p$-solid
base limit group. For that, we consider the subcollection $\mathfrak{J}_{M}$
of $\mathfrak{J}$ consisting of all the sequences of specializations
$\{(c_{n},u_{n},z_{n},y_{n},w_{n},p_{n},a)\}$ of $M$, for which: 
\begin{enumerate}
\item If $Term(M)$ is either $p$-rigid or $p$-solid, then, for all $n$,
the specialization $(c_{n},u_{n},z_{n},y_{n},w_{n},p_{n},a)$ restricts
to a specialization of $Term(M)$ that factors through some flexible
quotient of $Term(M)$. 
\item Otherwise, for all $n$, the restriction of the specialization $(c_{n},u_{n},z_{n},y_{n},w_{n},p_{n},a)$
into $Term(M)$ cannot be shortened by precomposing with a modular
automorphism induced by the $p$-JSJ of $Term(M)$. 
\end{enumerate}
The subcollection $\mathfrak{J}_{M}$ factors through finitely many
maximal limit groups $H$. We continue with $H$ by applying the shortening
procedure iterative construction that was presented in the beginning
of the construction. According to the descending chain condition of
limit groups, this construction must terminate after finitely many
steps. The result is a diagram $K_{M}$ with finitely many resolutions,
each of which terminates in a graded formal closure of $Comp(z,y,w,p,a)$. 
\begin{lem}
\label{lem:38} Let $(c_{0},u_{0},z_{0},y_{0},w_{0},p_{0},a)$ be
a specialization of $M$ that satisfies the conditions 1-3 in the
top of the section. Then, there exists a specialization $(c_{0}',u_{0}',z_{0},y_{0},w_{0},p_{0},a)$
that factors through $M$ so that: 
\begin{enumerate}
\item The specializations $(c_{0},u_{0},z_{0},y_{0},w_{0},p_{0},a)$ and
$(c_{0}',u_{0}',z_{0},y_{0},w_{0},p_{0},a)$ restricts to the same
modular family of $Term(M)$ (and thus to the same exceptional family
of $Term(ExtraPSGCl)$, and to the same exceptional family of $B$). 
\item The two families of representatives that the specializations $(c_{0},u_{0},z_{0},y_{0},w_{0},p_{0},a)$
and $(c_{0}',u_{0}',z_{0},y_{0},w_{0},p_{0},a)$ induce, represent
the same modular families of the group $G_{i}$, for all $i=1,...,s$. 
\item The specialization $(c_{0}',u_{0}',z_{0},y_{0},w_{0},p_{0},a)$ satisfies
the conditions 1-3 in the top of the section, and factors through
one of the resolutions $GFRes(c,u,z,y,w,p,a)$ of $K_{M}$. 
\end{enumerate}
Moreover, if $Term(M)$ is either $p$-rigid or $p$-solid, and the
restriction of the specialization $(c_{0},u_{0},z_{0},y_{0},w_{0},p_{0},a)$
to $Term(M)$ is not exceptional, then we can assume that the terminal
graded formal closure of the resolution $GFRes(c,u,z,y,w,p,a)$ is
a proper quotient of $M$. 
\end{lem}

\begin{proof}
Let $t_{0}$ be a specialization of $Term(M)$, that maps $p$ to
some $p_{0}$ in $F_{k}$. Assume that $t_{0}$ is short when $Term(M)$
is neither $p$-rigid, nor $p$-solid, and that $t_{0}$ factors through
a flexible quotient of $Term(M)$ otherwise. Let $M_{t_{0}}$ be the
induced ungraded closure, and assume that $(c_{0},u_{0},z_{0},y_{0},w_{0},p_{0},a)$
is a specialization of $M_{t_{0}}$ that satisfies the conditions
1-3.

Let $\Omega$ be a collection of (positive) Diophantine conditions
$\Omega_{1}(q,c,u,z,y,w,p_{0},a),...,\Omega_{r}(q,c,u,z,y,w,p_{0},a)$
stating that the specialization $(c,u,z,y,w,p_{0},a)$ of $M$ factors
through one of the resolutions of $K_{M}$ (without $M$ itself).
And let $\Sigma$ be a collection of (positive) Diophantine conditions
$\Sigma_{1}(f,c,u,z,y,w,p_{0},a),...,\Sigma_{t}(f,c,u,z,y,w,p_{0},a)$
stating that the specialization $(c,u,z,y,w,p_{0},a)$ of $M$ satisfies
that the restriction $(z,y,w,p_{0},a)$ is collapsed.

We construct two formal MR diagrams $FMRD(\Omega)$ and $FMRD(\Sigma)$
over the ungraded closure $M_{t_{0}}$, the first for the collection
$\Omega$, and the second for the collection $\Sigma$.

Let $\{(c_{n},u_{n},z_{n},y_{n},w_{n},p_{0},a)\}$ be a test sequence
of $M_{t_{0}}$. Note that, for all $n$, if the restriction $(z_{n},y_{n},w_{n},p_{0},a)$
is non-collapsed, then the specialization $(c_{n},u_{n},z_{n},y_{n},w_{n},p_{0},a)$
satisfies the conditions 1-3. Thus, according to the construction
of $K_{M}$, the sequence $\{(c_{n},u_{n},z_{n},y_{n},w_{n},p_{0},a)\}$
either contains a subsequence for which the restricted sequence $\{(z_{n},y_{n},w_{n},p_{0},a)\}$
consists of collapsed specializations, or it subfactors through one
of the resolutions of $K_{M}$. Hence, the collection of terminal
ungraded closures in the two formal MR diagrams $FMRD(\Omega)$ and
$FMRD(\Sigma)$, forms a covering closure for $M_{t_{0}}$.

Now consider the specialization $(c_{0},u_{0},z_{0},y_{0},w_{0},p_{0},a)$
of $M_{t_{0}}$ that satisfies conditions 1-3. It cannot factor through
$FMRD(\Sigma)$. Hence, it must factor through one of the ungraded
closures in $FMRD(\Omega)$. Due to the existence of a formal solution
in that ungraded closure, we conclude that the specialization $(c_{0},u_{0},z_{0},y_{0},w_{0},p_{0},a)$
factors through one of the resolutions of $K_{M}$.

Finally, if $(c_{0}',u_{0}',z_{0},y_{0},w_{0},p_{0},a)$ is a specialization
of $M_{t_{0}'}$, for some specialization $t_{0}'$ of $Term(M)$
that represents the same modular family of $t_{0}$, then there exists
some $c_{0}'',u_{0}''\in F_{k}$ so that the specialization $(c_{0}'',u_{0}'',z_{0},y_{0},w_{0},p_{0},a)$
factors through $M_{t_{0}}$, and the two families of representatives
that the specializations $(c_{0}',u_{0}',z_{0},y_{0},w_{0},p_{0},a)$
and $(c_{0}'',u_{0}'',z_{0},y_{0},w_{0},p_{0},a)$ induce, represent
the same modular families of the group $G_{i}$, for all $i=1,...,s$. 
\end{proof}
\begin{lem}
\label{lem:39} Let $\Gamma_{l}$ be a group of level $l$ in the
model. If $l$ is large enough, then the following property is satisfied
for $\Gamma_{l}$.

Let $(c_{0},u_{0},z_{0},y_{0},w_{0},p_{0},a)$ be a $\Gamma_{l}$-specialization
of $M$ that satisfies the conditions 1-3 in the top of the section
(w.r.t. $\Gamma_{l}$). Then, there exists a $\Gamma_{l}$-specialization
$(c_{0}',u_{0}',z_{0},y_{0},w_{0},p_{0},a)$ that factors through
$M$ so that: 
\begin{enumerate}
\item The $\Gamma_{l}$-specializations $(c_{0},u_{0},z_{0},y_{0},w_{0},p_{0},a)$
and $(c_{0}',u_{0}',z_{0},y_{0},w_{0},p_{0},a)$ restrict to the same
$\Gamma_{l}$-modular family of $Term(M)$ (and thus to the same $\Gamma_{l}$-exceptional
family of $Term(ExtraPSGCl)$, and to the same $\Gamma_{l}$-exceptional
family of $B$). 
\item The two families of representatives that the $\Gamma_{l}$-specializations
$(c_{0},u_{0},z_{0},y_{0},w_{0},p_{0},a)$ and $(c_{0}',u_{0}',z_{0},y_{0},w_{0},p_{0},a)$
induce, represent the same $\Gamma_{l}$-modular families of the group
$G_{i}$, for all $i=1,...,s$. 
\item The $\Gamma_{l}$-specialization $(c_{0}',u_{0}',z_{0},y_{0},w_{0},p_{0},a)$
satisfies the conditions 1-3 in the top of the section, and factors
through one of the resolutions $GFRes(c,u,z,y,w,p,a)$ of $K_{M}$. 
\end{enumerate}
Moreover, if $Term(M)$ is either $p$-rigid or $p$-solid, and the
restriction of the $\Gamma_{l}$-specialization $(c_{0},u_{0},z_{0},y_{0},w_{0},p_{0},a)$
to $Term(M)$ is not $\Gamma_{l}$-exceptional, then we can assume
that the terminal graded formal closure of the resolution $GFRes(c,u,z,y,w,p,a)$
is a proper quotient of $M$. 
\end{lem}

\begin{proof}
Every lift $(\tilde{c}_{0},\tilde{u}_{0},\tilde{z}_{0},\tilde{y}_{0},\tilde{w}_{0},\tilde{p}_{0},a)$
of $(c_{0},u_{0},z_{0},y_{0},w_{0},p_{0},a)$ to a specialization
of $M$ in $F_{k}$, satisfies the conditions 1-3. We lift $(c_{0},u_{0},z_{0},y_{0},w_{0},p_{0},a)$
to a specialization $(\tilde{c}_{0},\tilde{u}_{0},\tilde{z}_{0},\tilde{y}_{0},\tilde{w}_{0},\tilde{p}_{0},a)$
of $M$ in $F_{k}$, so that in the case that $Term(M)$ is either
$p$-rigid or $p$-solid, and the restriction of the $\Gamma_{l}$-specialization
$(c_{0},u_{0},z_{0},y_{0},w_{0},p_{0},a)$ to $Term(M)$ is not $\Gamma_{l}$-exceptional,
the restriction of the lift $(\tilde{c}_{0},\tilde{u}_{0},\tilde{z}_{0},\tilde{y}_{0},\tilde{w}_{0},\tilde{p}_{0},a)$
to $Term(M)$ is not $F_{k}$-exceptional.

Hence, according to \lemref{38}, there exists an $F_{k}$-specialization
$(\tilde{c}_{0}',\tilde{u}_{0}',\tilde{z}_{0},\tilde{y}_{0},\tilde{w}_{0},\tilde{p}_{0},a)$
of $M$ that restricts to the modular family of $Term(M)$ represented
by $(\tilde{c}_{0},\tilde{u}_{0},\tilde{z}_{0},\tilde{y}_{0},\tilde{w}_{0},\tilde{p}_{0},a)$,
and factors through one of the resolutions $GFRes(c,u,z,y,w,p,a)$
of $K_{M}$. Let $(c_{0}',u_{0}',z_{0},y_{0},w_{0},p_{0},a)$ be the
projection of $(\tilde{c}_{0}',\tilde{u}_{0}',\tilde{z}_{0},\tilde{y}_{0},\tilde{w}_{0},\tilde{p}_{0},a)$
to $\Gamma_{l}$. Then, $(c_{0},u_{0},z_{0},y_{0},w_{0},p_{0},a)$
is a $\Gamma_{l}$-specialization of $M$ that restricts to the $\Gamma_{l}$-modular
family of $Term(M)$ represented by $(c_{0},u_{0},z_{0},y_{0},w_{0},p_{0},a)$,
and factors through the resolution $GFRes(c,u,z,y,w,p,a)$ of $K_{M}$.
The other details of the claim follows immediately from this. 
\end{proof}
Finally, we continue by extending each of the terminal graded closures
in the diagram $K_{M}$, with a new diagram, in a similar way that
$K_{M}$ was constructed. We continue the construction iteratively.
According to the descending chain condition of limit groups, this
construction must terminate after finitely many steps. The result
is an MR diagram, that extends the diagram $K$, and admits finitely
many resolutions, each of which terminates in a graded formal closure
$M'(u,z,y,w,p,a)$, whose terminal group is either a $p$-rigid or
a $p$-solid limit group.

We call the obtained diagram \emph{the $CollapseExtra$ MR diagram
of the collection $WPRS$ over the graded closure $ExtraPSGCl(u,z,y,w,p,a)$},
and we denote it by $CollapseExtraMRD(ExtraPSGCl)$, or briefly $CollapseExtraMRD$.
We summarize the general properties of the construction in the next
theorem.

\subsection{General Properties of the $CollapseExtra$ MR Diagram}
\begin{thm}
\label{thm:40} Let $ExtraPSGCl(u,z,y,w,p,a)$ be a terminal graded
closure in $ExtraPSMRD$.

Let $CollapseExtraMRD$ be the $CollapseExtra$ MR diagram of the
collection $WPRS$ over the graded closure $ExtraPSGCl(u,z,y,w,p,a)$.

Let $\{(c_{n},u_{n},z_{n},y_{n},w_{n},p_{n},a)\}$ be a sequence of
specializations, for which the restricted sequence $\{(u_{n},z_{n},y_{n},w_{n},p_{n},a)\}$
is a sequence of specializations of $ExtraPSGCl(u,z,y,w,p,a)$, the
restricted sequence $\{(z_{n},y_{n},w_{n},p_{n},a)\}$ is a test sequence
of $Comp(z,y,w,p,a)$, and for all $n$: 
\begin{enumerate}
\item The restriction $(y_{n},w_{n},p_{n},a)$ is non-collapsed. 
\item The specialization $(c_{n},u_{n},z_{n},y_{n},w_{n},p_{n},a)$ restricts
to an exceptional solution of the base group \\
 $Term(ExtraPSGCl)$ of $ExtraPSGCl$, and to an exceptional solution
of $B$. 
\item The specialization $c_{n}$ explains how the specialization $(u_{n},z_{n},y_{n},w_{n},p_{n},a)$
is collapsed. 
\end{enumerate}
Then, for all $n$ in a subsequence of the integers, there exist $c_{n}',u_{n}'\in F_{k}$,
so that: 
\begin{enumerate}
\item The specialization $(c_{n}',u_{n}',z_{n},y_{n},w_{n},p_{n},a)$ factors
through one of the terminal graded formal closures $CollapseExtraGFCl(c,u,z,y,w,p,a)$
of $CollapseExtraMRD$. 
\item The specialization $(c_{n}',u_{n}',z_{n},y_{n},w_{n},p_{n},a)$ restricts
to an exceptional solution of the $p$-rigid or $p$-solid base group
$Term(CollapseExtraGFCl)$ of $CollapseExtraGFCl(c,u,z,y,w,p,a)$. 
\item The specialization $(c_{n}',u_{n}',z_{n},y_{n},w_{n},p_{n},a)$ restricts
to an exceptional solution of the $p$-rigid or $p$-solid base group
$Term(ExtraPSGCl)$ of $ExtraPSGCl(u,z,y,w,p,a)$, that belongs to
the same exceptional family of the restriction of $(c_{n},u_{n},z_{n},y_{n},w_{n},p_{n},a)$
to $Term(ExtraPSGCl)$. 
\item The specialization $(c_{n}',u_{n}',z_{n},y_{n},w_{n},p_{n},a)$ restricts
to an exceptional solution of $B$, the $p$-rigid or $p$-solid base
group of $Comp(z,y,w,p,a)$, that belongs to the same exceptional
family of the restriction of $(c_{n},u_{n},z_{n},y_{n},w_{n},p_{n},a)$
to $B$. 
\item The specializations $(c_{n},u_{n},z_{n},y_{n},w_{n},p_{n},a)$ and
$(c_{n}',u_{n}',z_{n},y_{n},w_{n},p_{n},a)$ are collapsed in the
same form. 
\end{enumerate}
\end{thm}

The $CollapseExtra$ MR diagram $CollapseExtra(ExtraPSGCl)$ encodes
all the specializations of \\
 $ExtraPSGCl(u,z,y,w,p,a)$ that are collapsed in a formal way, in
the following more precise sense: 
\begin{lem}
\label{lem:41} Let $ExtraPSGCl(u,z,y,w,p,a)$ be a terminal graded
closure in $ExtraPSMRD$.

Let $CollapseExtraMRD$ be the $CollapseExtra$ MR diagram of the
collection $WPRS$ over the graded closure $ExtraPSGCl(u,z,y,w,p,a)$.
Let $\Lambda$ be a collection of (positive) Diophantine conditions
stating that the specialization $(u,z,y,w,p,a)$ is collapsed.

Let $GCl(b,z,y,w,p,a)$ be a graded closure of $Comp(z,y,w,p,a)$,
and denote its $p$-graded base limit group by $Term(GCl)$. Assume
that the graded closure $ExtraPSGCl(u,z,y,w,p,a)$ is mapped into
$GCl(b,z,y,w,p,a)$, so that the subgroup $Comp(z,y,w,p,a)$ is mapped
canonically. Assume that there exists an element \\
 $c\in GCl(b,z,y,w,p,a)$ so that the words of one of the systems
$\Lambda_{i}(c,u,z,y,w,p,a)$ in $\Lambda$, $i=1,...,s$, represent
the trivial element in $GCl(b,z,y,w,p,a)$.

Then, every specialization $(c_{0},u_{0},z_{0},y_{0},w_{0},p_{0},a)$
that factors through $GCl(b,z,y,w,p,a)$, and satisfies the conditions
1-3 in the top of the section, factors through the $CollapseExtra$
MR diagram $CollapseExtraMRD$. 
\end{lem}

\begin{proof}
Assume that there exists a specialization $(c_{0},u_{0},z_{0},y_{0},w_{0},p_{0},a)$
that factors through $GCl$, and satisfies the conditions 1-3 in the
top of the section. Then, the mapping of $Term(ExtraPSGCl)$ into
any level of $GCl$ is compatible.

Let $g_{p_{0}}$ be the restriction of $(c_{0},u_{0},z_{0},y_{0},w_{0},p_{0},a)$
to $Term(GCl)$. The solution $g_{p_{0}}$ induces an ungraded closure
$GCl_{g_{p_{0}}}$ that is covered by the given graded closure $GCl(b,z,y,p,a)$.
Note that $GCl_{g_{p_{0}}}$ is generated by the (image of the) elements
$b,z,y,w,a$. In particular the elements $u,c\in GCl_{g_{p_{0}}}$
are generated by these elements, $u=u(b,z,y,w,a),\,c=c(b,z,y,w,a)$.
Moreover, every specialization $(c'_{0},u'_{0},z'_{0},y'_{0},w'_{0},p_{0},a)$
of $GCl_{g_{p_{0}}}$ satisfies the conditions 2-3 in the beginning
of the section.

Let $\Omega$ be a collection of (positive) Diophantine conditions
\[
\Omega_{1}(q,c(b,z,y,w,a),u(b,z,y,w,a),z,y,w,p_{0},a),...,\Omega_{r}(q,c(b,z,y,w,a),u(b,z,y,w,a),z,y,w,p_{0},a)
\]
stating that the specialization $(c(b,z,y,w,a),u(b,z,y,w,a),z,y,w,p_{0},a)$
factors through the $CollapseExtra$ MR diagram $CollapseExtraMRD$.

And let $\Sigma$ be a collection of (positive) Diophantine conditions
\[
\Sigma_{1}(f,c(b,z,y,w,a),u(b,z,y,w,a),z,y,w,p_{0},a),...,\Sigma_{t}(f,c(b,z,y,w,a),u(b,z,y,w,a),z,y,w,p_{0},a)
\]
stating that the specialization $(c,u,z,y,w,p_{0},a)$ satisfies that
the restriction $(z,y,w,p_{0},a)$ is collapsed.

We construct two formal MR diagrams $FMRD(\Omega)$ and $FMRD(\Sigma)$
over the ungraded closure $GCl_{g_{p_{0}}}$, the first of the collection
$\Omega$, and the second for the collection $\Sigma$.

Let $\{(b_{n},z_{n},y_{n},w_{n},p_{0},a)\}$ be a test sequence of
$GCl_{g_{p_{0}}}$, and consider the induced specializations $\{u_{n}\}$
and $\{c_{n}\}$ of the elements $u,c\in GFCl_{g_{p_{0}}}$. Then,
the sequence $\{(c_{n},u_{n},c_{n},y_{n},w_{n},p_{0},a)\}$ either
admits a subsequence for which the restricted specializations $\{(z_{n},y_{n},w_{n},p_{0},a)\}$
are collapsed, or it subfactors through one of the resolutions of
$CollapseExtraMRD$. Hence, the collection of terminal ungraded closures
in the two formal MR diagrams $FMRD(\Omega)$ and $FMRD(\Sigma)$,
forms a covering closure for $GCl_{g_{p_{0}}}$.

Now consider the specialization $(c_{0},u_{0},z_{0},y_{0},w_{0},p_{0},a)$
that factors through $GCl_{g_{p_{0}}}$ and satisfies the conditions
1-3. It cannot factor through $FMRD(\Sigma)$, and hence it factors
through some terminal ungraded closure of $FMRD(\Omega)$. Due to
the formal solution in that ungraded closure, we deduce that $(c_{0},u_{0},z_{0},y_{0},w_{0},p_{0},a)$
factors through $CollapseExtraMRD$.
\end{proof}
\begin{lem}
\label{lem:42} Keep the notation of \ref{lem:41}. Let $\Gamma_{l}$
be a group of level $l$ in the model. If $l$ is large enough, then
the following property is satisfied in $\Gamma_{l}$.

Every $\Gamma_{l}$-specialization $(c_{0},u_{0},z_{0},y_{0},w_{0},p_{0},a)$
that factors through $GCl(b,z,y,w,p,a)$, and satisfies the conditions
1-3 in the top of the section (w.r.t. $\Gamma_{l}$), factors through
the $CollapseExtra$ MR diagram $CollapseExtraMRD$. 
\end{lem}

\begin{proof}
We lift $(c_{0},u_{0},z_{0},y_{0},w_{0},p_{0},a)$ to a specialization
$(\tilde{c}_{0},\tilde{u}_{0},\tilde{z}_{0},\tilde{y}_{0},\tilde{w}_{0},\tilde{p}_{0},a)$
that factors through $GCl(b,z,y,w,p,a)$. Then, the specialization
$(\tilde{c}_{0},\tilde{u}_{0},\tilde{z}_{0},\tilde{y}_{0},\tilde{w}_{0},\tilde{p}_{0},a)$
satisfies the conditions 1-3, and thus, according to \lemref{41},
it factors through $CollapseExtraMRD$. Hence, the $\Gamma_{l}$-specialization
$(c_{0},u_{0},z_{0},y_{0},w_{0},p_{0},a)$ factors through $CollapseExtraMRD$. 
\end{proof}
\begin{lem}
\label{lem:43} Let $ExtraPSGCl(u,z,y,w,p,a)$ be a terminal graded
closure in $ExtraPSMRD$.

Let $CollapseExtraMRD$ be the $CollapseExtra$ MR diagram of the
collection $WPRS$ and the graded closure $ExtraPSGCl(u,z,y,w,p,a)$.

Let $CollapseExtraGFCl(c,u,z,y,w,p,a)$ be a terminal graded formal
closure in $CollapseExtraMRD$, and let $(c_{0},u_{0},z_{0},y_{0},w_{0},p_{0},a)$
be a specialization of $CollapseExtraGFCl(c,u,z,y,w,p,a)$.

Then, every specialization $(u_{0}',z_{0},y_{0},w_{0},p_{0},a)$ of
$ExtraPSGCl(u,z,y,w,p,a)$, that restricts to a solution of $Term(ExtraPSGCl)$
that belongs to the same modular family of the restriction of $(c_{0},u_{0},z_{0},y_{0},w_{0},p_{0},a)$
to $Term(ExtraPSGCl)$, is collapsed. 
\end{lem}

\begin{proof}
Since $(u_{0},z_{0},y_{0},w_{0},p_{0},a)$ factors through $CollpaseExtraMRD$,
we know that it is collapsed. Since $(u_{0},z_{0},y_{0},w_{0},p_{0},a)$
and $(u_{0}',z_{0},y_{0},w_{0},p_{0},a)$ are specializations of $ExtraPSGCl(u,z,y,w,p,a)$
that restrict to the same modular family of $Term(ExtraPSGCl)$, the
two families of representatives that they induce, represent the same
modular families of the group $G_{i}$, for all $i=1,...,s$. Hence,
$(u_{0}',z_{0},y_{0},w_{0},p_{0},a)$ is collapsed. 
\end{proof}
\begin{lem}
\label{lem:44} Let $ExtraPSGCl(u,z,y,w,p,a)$ be a terminal graded
closure in $ExtraPSMRD$.

Let $CollapseExtraMRD$ be the $CollapseExtra$ MR diagram of the
collection $WPRS$ and the graded closure $ExtraPSGCl(u,z,y,w,p,a)$.

Let $\Gamma_{l}$ be a group of level $l$ in the model. If $l$ is
large enough, then the following property is satisfied in $\Gamma_{l}$.

Let $CollapseExtraGFCl(c,u,z,y,w,p,a)$ be a terminal graded formal
closure in $CollapseExtraMRD$, and let $(c_{0},u_{0},z_{0},y_{0},w_{0},p_{0},a)$
be a $\Gamma_{l}$-specialization of $CollapseExtraGFCl(c,u,z,y,w,p,a)$.

Then, every $\Gamma_{l}$-specialization $(u_{0}',z_{0},y_{0},w_{0},p_{0},a)$
of $ExtraPSGCl(u,z,y,w,p,a)$, that restricts to a solution of $Term(ExtraPSGCl)$
that belongs to the same modular family of the restriction of $(c_{0},u_{0},z_{0},y_{0},w_{0},p_{0},a)$
to $Term(ExtraPSGCl)$, is collapsed. 
\end{lem}

\begin{proof}
The same proof as the case of $F_{k}$. 
\end{proof}

\section{Approximations of $EAE$-Sets under Minimal Rank Assumption}\label{sec:An-Approximation-to-Abstract-EAE}

Let $WPRS$ be a given collection of $wp$-rigid and $wp$-solid limit
groups $G_{1}(x_{1},w,p,a),,...,G_{s}(x_{s},w,p,a)$ ($wp$ is the
parameter set for those graded limit groups). For all $i=1,...,s$,
we fix an integer $e_{i}\in\{0,...,M(G_{i})\}$, where $M(G_{i})$
is the maximal number of exceptional families of specializations in
$G_{i}$ for a fixed value of the parameters. We consider the graded
completion $Comp(z,y,w,p,a)$ of some $p$-graded resolution $Res(y,w,p,a)$,
and we denote the terminal $p$-rigid or $p$-solid base group of
$Comp(z,y,w,p,a)$ by $B$ (note that $p$ is the parameter set for
the graded completion $Comp(z,y,w,p,a)$). We assume that the amalgamated
free product of the groups in the collection $WPRS$ 
\[
\text{\ensuremath{\stackrel[i=1]{s}{\ast_{\langle w,p,a\rangle}}\stackrel[j=1]{e_{i}}{\ast_{\langle w,p,a\rangle}}G_{i}}}
\]
amalgamated over the subgroup $\langle w,p,a\rangle$, is mapped into
each of the levels of $Comp(z,y,w,p,a)$ (in a way that respects the
canonical maps between those levels). We assume further that $Comp(z,y,w,p,a)$
admits some specialization $(z_{0},y_{0},w_{0},p_{0},a)$ so that
the restriction of $(y_{0},w_{0},p_{0},a)$ to $\stackrel[j=1]{e_{i}}{\ast_{\langle w,p,a\rangle}}G_{i}$
represent a family of representatives of pairwise distinct exceptional
families of $G_{i}$ corresponding to the value $w_{0}p_{0}$ of the
parameters, for all $i=1,...,s$.

Let $AbsSys(c,z,y,w,p,a)$ be a finite collection of systems of equations.
We add to the systems in $AbsSys(c,z,y,w,p,a)$, the finitely many
systems stating that the specialization $(z,y,w,p,a)$ of $Comp(z,y,w,p,a)$
is collapsed, i.e., that for some $i=1,...,s$, the $e_{i}$ specializations
of $G_{i}$ that the specialization $(z,y,w,p,a)$ induces, do not
form a family of representatives of pairwise distinct exceptional
families of $G_{i}$ corresponding to the value $wp$ of the parameters.

We construct the $p$-graded formal MR diagram of the systems $AbsSys(c,z,y,w,p,a)$
w.r.t. the resolution $Comp(z,y,w,p,a)$. We denote this diagram by
$PosGFMRD$. We also construct the $ExtraPS$ MR diagram $ExtraPSMRD$
of the collection $WPRS$ over the resolution $Comp(z,y,w,p,a)$. 
\begin{defn}
Let $(z_{0},y_{0},w_{0},p_{0},a)$ be a specialization of $Comp(z,y,w,p,a)$
(constructed with respect to some proof system PS). We say that $(z_{0},y_{0},w_{0},p_{0},a)$
is an \emph{abstract validPS statement}, if the following conditions
hold: 
\begin{enumerate}
\item The specialization $(z_{0},y_{0},w_{0},p_{0},a)$ is non-collapsed. 
\item For all $i=1,...,s$, the family of representatives of the exceptional
families of the $wp$-rigid or $wp$-solid group $G_{i}$, that the
specialization $(z_{0},y_{0},w_{0},p_{0},a)$ induces, forms a full
family w.r.t. the value $w_{0}p_{0}$ of the parameters. 
\item There exists no $c_{0}\in F_{k}$, so that the specialization $(c_{0},z_{0},y_{0},w_{0},p_{0},a)$
solves some of the systems in $AbsSys(c,z,y,w,p,a)$. 
\end{enumerate}
The set of all the values $p_{0}\in F_{k}$, for which there exists
an abstract validPS statement $(z_{0},y_{0},w_{0},p_{0},a)$, is denoted
by $EAE(p)$. 
\end{defn}

\begin{defn}
Let $\Gamma$ be a group in the model. Let $(z_{0},y_{0},w_{0},p_{0},a)$
be a specialization of $Comp(z,y,w,p,a)$ (constructed with respect
to some proof system PS). We say that $(z_{0},y_{0},w_{0},p_{0},a)$
is an \emph{abstract $\Gamma$-validPS statement}, if the following
conditions hold: 
\begin{enumerate}
\item The $\Gamma$-specialization $(z_{0},y_{0},w_{0},p_{0},a)$ is non-collapsed. 
\item For all $i=1,...,s$, the family of representatives of the $\Gamma$-exceptional
families of the $wp$-rigid or $wp$-solid group $G_{i}$, that the
$\Gamma$-specialization $(z_{0},y_{0},w_{0},p_{0},a)$ induces, forms
a full family w.r.t. the value $w_{0}p_{0}$ of the parameters. 
\item There exists no $c_{0}\in\Gamma$, so that the specialization $(c_{0},z_{0},y_{0},w_{0},p_{0},a)$
solves some of the systems in $AbsSys(c,z,y,w,p,a)$. 
\end{enumerate}
The set of all the values $p_{0}\in\Gamma$, for which there exists
an abstract $\Gamma$-validPS statement $(z_{0},y_{0},w_{0},p_{0},a)$,
is denoted by $EAE_{\Gamma}(p)$. 
\end{defn}

Our aim in this section, is to reduce the sets $EAE(p)$ and $EAE_{\Gamma}(p)$
uniformly, for almost all the groups $\Gamma$, into the Boolean algebra
of $AE$-sets. That is, to construct a formula $\varphi(p)$ in the
Boolean algebra of $AE$-formulas, that defines each of these sets
over the corresponding groups. We achieve this aim using a process
that terminates in finitely many steps. In each step, we define a
formula in the Boolean algebra of $AE$-formulas that approximates
our sets $EAE(p)$ and $EAE_{\Gamma}(p)$ uniformly, on one hand.
And on the other hand, every specialization $p=p_{0}$ in $EAE(p)$
or in $EAE_{\Gamma}(p)$, must be collected (over the corresponding
group) by a formula defined in some step along the process that we
are going to explain below.

\subsection{First Approximation of $EAE(p)$}
\begin{thm}
\label{thm:45} Let $p_{0}\in F_{k}$. Assume that there exists a
specialization $(z_{0},y_{0},w_{0},p_{0},a)$ that factors through
$Comp(z,y,w,p,a)$, and satisfies the following conditions: 
\begin{enumerate}
\item $(z_{0},y_{0},w_{0},p_{0},a)$ restricts to an exceptional specialization
$(b_{0},p_{0},a)$ of $B$. 
\item $(z_{0},y_{0},w_{0},p_{0},a)$ does not factor through $PosGFMRD$. 
\item For every terminal graded closure $ExtraPSGCl(u,z,y,w,p,a)$ in $ExtraPSMRD$,
for every $u_{0}\in F_{k}$, if $(u_{0},z_{0},y_{0},w_{0},p_{0},a)$
factors through $ExtraPSGCl(u,z,y,w,p,a)$, then there exists some
$u_{0}'\in F_{k}$, so that: 
\begin{enumerate}
\item The specializations $(u_{0},z_{0},y_{0},w_{0},p_{0},a)$ and $(u_{0}',z_{0},y_{0},w_{0},p_{0},a)$
restrict to the same exceptional family of $Term(ExtraPSGCl)$. 
\item The specialization $(u_{0}',z_{0},y_{0},w_{0},p_{0},a)$ factors through
the collapse-extra MR diagram associated with the graded closure $ExtraPSGCl(u,z,y,w,p,a)$. 
\end{enumerate}
\end{enumerate}
Then $p_{0}\in EAE(p)$. The set of all the specializations $p=p_{0}$
of the parameters, for which there exists some specialization $(z_{0},y_{0},w_{0},p_{0},a)$
who satisfies the conditions in the theorem, is denoted by $TSPS(p)$. 
\end{thm}

\begin{proof}
We consider the ungraded completion $Comp(z,y,w,p_{0},a)$ induced
by $(z_{0},y_{0},w_{0},p_{0},a)$.

For every terminal graded formal closure $PosGFCl(c,z,y,w,p,a)$,
in $PosGFMRD$, we consider its $p$-rigid or $p$-solid base group
$Term(PosGFCl)$. If $Term(PosGFCl)$ admits an exceptional solution
$h_{p_{0}}$, so that the restriction of $h_{p_{0}}$ to $B$ is an
exceptional solution of $B$ that belongs to the same exceptional
family of $(b_{0},p_{0},a)$, we list the ungraded closure $PosGFCl_{h_{p_{0}}}$
induced by (some exceptional solution in the exceptional family of)
$h_{p_{0}}$. The collection of all such (boundedly many) ungraded
closures, is denoted by $Pos_{p_{0}}$.

For every terminal graded closure $ExtraPSGCl(u,z,y,w,p,a)$, in $ExtraPSMRD$,
we consider its $p$-rigid or $p$-solid base group $Term(ExtraPSGCl)$.
If $Term(ExtraPSGCl)$ admits an exceptional solution $h_{p_{0}}$,
so that the restriction of $h_{p_{0}}$ to $B$ is an exceptional
solution of $B$ that belongs to the same exceptional family of $(b_{0},p_{0},a)$,
we list the ungraded closure $ExtraPSCl_{h_{p_{0}}}$ induced by (some
exceptional solution in the exceptional family of) $h_{p_{0}}$. The
collection of all such (boundedly many) ungraded closures, is denoted
by $ExtraPS_{p_{0}}$. We denote the subcollection of $ExtraPS_{p_{0}}$
consisting of the ungraded closures through which the specialization
$(z_{0},y_{0},w_{0},p_{0},a)$ factors (for some $u_{0}$), by $ExtraPS_{(z_{0},y_{0},w_{0},p_{0},a)}$.

By the assumption, and according to the construction of the collapse-extra
MR diagram, for every ungraded closure $ExtraPSCl_{h_{p_{0}}}$ in
$ExtraPS_{(z_{0},y_{0},w_{0},p_{0},a)}$, who is covered by the graded
closure \\
 $ExtraPSGCl(u,z,y,w,p,a)$, there exists a specialization $(c_{0},u_{0},z_{0},y_{0},w_{0},p_{0},a)$
of some terminal graded formal closure $CollapseExtraGFCl(c,u,z,y,w,p,a)$
in the collapse-extra MR diagram $CollapseExtraMRD$ associated with
$ExtraPSGCl(u,z,y,w,p,a)$, so that the restriction of $(c_{0},u_{0},z_{0},y_{0},w_{0},p_{0},a)$
to $B$ belongs to the exceptional family represented by $(b_{0},p_{0},a)$,
the restriction of $(c_{0},u_{0},z_{0},y_{0},w_{0},p_{0},a)$ to $Term(ExtraPSGCl)$
belongs to the exceptional family represented by $h_{p_{0}}$, and
the restriction of $(c_{0},u_{0},z_{0},y_{0},w_{0},p_{0},a)$ to \\
 $Term(CollapseExtraGFCl)$ is some exceptional solution which we
denote by $g_{p_{0}}$. We choose such a terminal graded formal closure
$CollapseExtraGFCl(c,u,z,y,w,p,a)$ and a specialization $(c_{0},u_{0},z_{0},y_{0},w_{0},p_{0},a)$
of it, and denote the induced ungraded closure by $CollapseExtraFCl_{g_{p_{0}}}$.
For every $ExtraPSCl_{h_{p_{0}}}$ in $ExtraPS_{(z_{0},y_{0},w_{0},p_{0},a)}$,
we have associated an ungraded closure $CollapseExtraFCl_{g_{p_{0}}}$.
The collection of all the (boundedly many) ungraded closures $CollapseExtraFCl_{g_{p_{0}}}$
that we obtain, is denoted briefly by $CollapseExtra_{p_{0}}$.

Let $ExtraPSCl_{h_{p_{0}}}$ be an ungraded closure in $ExtraPS_{(z_{0},y_{0},w_{0},p_{0},a)}$
who is covered by the graded closure $ExtraPSGCl(u,z,y,w,p,a)$, and
let $CollapseExtraFCl_{g_{p_{0}}}$ be the ungraded closure in $CollapseExtra_{p_{0}}$
associated with $ExtraPSCl_{h_{p_{0}}}$. Then, according to the construction
of the collapse extra MR diagram, every specialization $(u_{0}',z_{0}',y_{0}',w_{0}',p_{0},a)$
that factors through $CollapseExtraFCl_{g_{p_{0}}}$, restricts to
an exceptional specialization of $Term(ExtraPSGCl)$ that belongs
to the same exceptional family of $h_{p_{0}}$. Hence, according to
\lemref{43}, if $(z_{0}',y_{0}',w_{0}',p_{0},a)$ is a specialization
that factors through $CollapseExtraFCl_{g_{p_{0}}}$, then, for every
$u_{0}''\in F_{k}$, if $(u_{0}'',z_{0}',y_{0}',w_{0}',p_{0},a)$
is a specialization of $ExtraPSGCl(u,z,y,w,p,a)$ that restricts to
an exceptional family of $Term(ExtraPSGCl)$ that belongs to the same
exceptional family of $h_{p_{0}}$, then $(u_{0}'',z_{0}',y_{0}',w_{0}',p_{0},a)$
is collapsed.

We consider now the collection of ungraded closures $CollapseExtra_{p_{0}}$,
in one hand, and on the other hand, we consider the collection of
ungraded closures $Pos_{p_{0}}$ and $ExtraPS_{p_{0}}\backslash ExtraPS_{(z_{0},y_{0},w_{0},p_{0},a)}$.
Each of these closures, is in particular an ungraded closure of the
ungraded completion $Comp(z,y,w,p_{0},a)$. However, the specialization
$(z_{0},y_{0},w_{0},p_{0},a)$ of $Comp(z,y,w,p_{0},a)$, factors
through each of the closures in $CollapseExtra_{p_{0}}$, but it does
not factor through any of the closures in $Pos_{p_{0}}$ and $ExtraPS_{p_{0}}\backslash ExtraPS_{(z_{0},y_{0},w_{0},p_{0},a)}$.
Thus, there exits a test sequence of specializations $(z_{n},y_{n},w_{n},p_{0},a)$
of $Comp(z,y,w,p_{0},a)$, so that for all $n$, the specialization
$(z_{n},y_{n},w_{n},p_{0},a)$ factors through each of the closures
in $CollapseExtra_{p_{0}}$, and does not factor through any of the
closures in $Pos_{p_{0}}$ and $ExtraPS_{p_{0}}\backslash ExtraPS_{(z_{0},y_{0},w_{0},p_{0},a)}$.
We are going to prove that $(z_{n},y_{n},w_{n},p_{0},a)$ is eventually
a validPS statement.

By the construction of the sequence $(z_{n},y_{n},w_{n},p_{0},a)$,
and since, for all $n$, the specialization $(z_{n},y_{n},w_{n},p_{0},a)$
restricts to an exceptional solution of $B$ that belongs to the same
exceptional family of $(b_{0},p_{0},a)$, we have that for all $n$,
the specialization $(z_{n},y_{n},w_{n},p_{0},a)$ does not factor
through the diagram $PosGFMRD$, nor through any of the terminal graded
closures $ExtraPSGCl(u,z,y,w,p,a)$ in $ExtraPSMRD$ that does not
cover some closure in $ExtraPS_{p_{0}}$, and if it factors through
an ungraded closure $ExtraPSCl_{h_{p_{0}}}$ in $ExtraPS_{p_{0}}$,
then $ExtraPSCl_{h_{p_{0}}}$ belongs to $ExtraPS_{(z_{0},y_{0},w_{0},p_{0},a)}$.

By the constructions of the diagram $PosGFMRD$, this implies that
the specialization $(z_{n},y_{n},w_{n},p_{0},a)$ includes (eventually)
only exceptional solutions of the corresponding $wp$-rigid and $wp$-solid
limit groups that are mapped into $Comp(z,y,w,p,a)$, i.e., the specialization
$(z_{n},y_{n},w_{n},p_{0},a)$ is non-collapsed.

Assume by contradiction that for a subsequence $n$ of the integers,
the sequence $(z_{n},y_{n},w_{n},p_{0},a)$ consists of non-validPS
statements. Let $u_{n}\in F_{k}$ be a sequence of specializations
testifying this property, i.e., $(u_{n},z_{n},y_{n},w_{n},p_{0},a)$
is non-collapsed, for all $n$. In particular, according to \thmref{33},
there exists some terminal graded closure $ExtraPSGCl(u,z,y,w,p,a)$
in the diagram $ExtraPSMRD$, so that for all $n$ in a subsequence
of the integers, there exists $u_{n}'\in F_{k}$, so that $(u_{n}',z_{n},y_{n},w_{n},p_{0},a)$
is a specialization of $ExtraPSGCl(u,z,y,w,p,a)$, and: 
\begin{enumerate}
\item The specialization $(u_{n}',z_{n},y_{n},w_{n},p_{0},a)$ restricts
to an exceptional specialization of $B$ that belongs to the same
exceptional family of $(b_{0},p_{0},a)$. 
\item The specialization $(u_{n}',z_{n},y_{n},w_{n},p_{0},a)$ restricts
to an exceptional specialization $h_{p_{0}}^{n}$ of $Term(ExtraPSGCl)$. 
\item The specialization $(u_{n}',z_{n},y_{n},w_{n},p_{0},a)$ is a non-collapsed. 
\end{enumerate}
Denote by $ExtraPSCl_{h_{p_{0}}^{n}}$ the ungraded closure obtained
by specializing the base group $Term(ExtraPSGCl)$ of $ExtraPSGCl(u,z,y,w,p,a)$
according to the restriction $h_{p_{0}}^{n}$ of $(u_{n}',z_{n},y_{n},w_{n},p_{0},a)$
to it. By the construction of the test sequence $(z_{n},y_{n},w_{n},p_{0},a)$,
the specialization $(z_{n},y_{n},w_{n},p_{0},a)$ does not factor
through any of the closures $ExtraPS_{p_{0}}\backslash ExtraPS_{(z_{0},y_{0},w_{0},p_{0},a)}$.
Thus, the ungraded closure $ExtraPSCl_{h_{p_{0}}^{n}}$ must be represented
by some ungraded closure $ExtraPSCl_{h_{p_{0}}}$, that is covered
by the graded closure $ExtraPSGCl(u,z,y,w,p,a)$, and belongs to $ExtraPS_{(z_{0},y_{0},w_{0},p_{0},a)}$.
That is, $h_{p_{0}}^{n}$ and $h_{p_{0}}$ belong to the same exceptional
family of $Term(ExtraPSGCl)$, for all $n$.

However, by construction, the specialization $(z_{n},y_{n},w_{n},p_{0},a)$
factors through the ungraded closure \\
 $CollapseExtraFCl_{g_{p_{0}}}$ associated with $ExtraPSCl_{h_{p_{0}}}$,
so $(u_{n}',z_{n},y_{n},w_{n},p_{0},a)$ must be collapsed, for all
$n$, a contradiction. 
\end{proof}
\begin{thm}
\label{thm:46} Let $\Gamma$ be a group of level $l$ in the model.
If $l$ is large enough, then the following property is satisfied
for $\Gamma$.

Let $p_{0}\in\Gamma$. Assume that there exists a $\Gamma$-specialization
$(z_{0},y_{0},w_{0},p_{0},a)$ that factors through $Comp(z,y,w,p,a)$,
and satisfies the following conditions: 
\begin{enumerate}
\item $(z_{0},y_{0},w_{0},p_{0},a)$ restricts to a $\Gamma$-exceptional
specialization $(b_{0},p_{0},a)$ of $B$. 
\item $(z_{0},y_{0},w_{0},p_{0},a)$ does not factor through $PosGFMRD$. 
\item For every terminal graded closure $ExtraPSGCl(u,z,y,w,p,a)$ in $ExtraPSMRD$,
for every $u_{0}\in\Gamma$, if $(u_{0},z_{0},y_{0},w_{0},p_{0},a)$
factors through $ExtraPSGCl(u,z,y,w,p,a)$, then there exists some
$u_{0}'\in\Gamma$, so that: 
\begin{enumerate}
\item The $\Gamma$-specializations $(u_{0},z_{0},y_{0},w_{0},p_{0},a)$
and $(u_{0}',z_{0},y_{0},w_{0},p_{0},a)$ restrict to the same $\Gamma$-exceptional
family of $Term(ExtraPSGCl)$. 
\item The $\Gamma$-specialization $(u_{0}',z_{0},y_{0},w_{0},p_{0},a)$
factors through the collapse-extra MR diagram associated with the
graded closure $ExtraPSGCl(u,z,y,w,p,a)$. 
\end{enumerate}
\end{enumerate}
Then $p_{0}\in EAE_{\Gamma}(p)$. The set of all the $\Gamma$-specializations
$p=p_{0}$ of the parameters, for which there exists some specialization
$(z_{0},y_{0},w_{0},p_{0},a)$ who satisfies the conditions in the
theorem, is denoted by $TSPS_{\Gamma}(p)$. 
\end{thm}

\begin{proof}
In a similar way of the proof of \thmref{45}, we consider the ungraded
completion $Comp(z,y,w,p_{0},a)$ induced by $(z_{0},y_{0},w_{0},p_{0},a)$.

We list, in a collection denoted $Pos_{p_{0}}$, the (boundedly many)
ungraded closures $PosGFCl_{h_{p_{0}}}$ who are covered by terminal
graded formal closures $PosGFCl(c,z,y,w,p,a)$ in $PosGFMRD$, so
that $h_{p_{0}}$ is a $\Gamma$-exceptional solution of $Term(PosGFCl)$
whose restriction to $B$ belongs to the same $\Gamma$-exceptional
family of $(b_{0},p_{0},a)$.

We list, in a collection denoted $ExtraPS_{p_{0}}$, the (boundedly
many) ungraded closures $ExtraPSCl_{h_{p_{0}}}$ who are covered by
terminal graded closures $ExtraPSGCl(u,z,y,w,p,a)$ in $ExtraPSMRD$,
so that $h_{p_{0}}$ is a $\Gamma$-exceptional solution of $Term(ExtraPSGCl)$
whose restriction to $B$ belongs to the same $\Gamma$-exceptional
family of $(b_{0},p_{0},a)$. The subcollection of $ExtraPS_{p_{0}}$
consisting of the ungraded closures through which the specialization
$(z_{0},y_{0},w_{0},p_{0},a)$ factors, is denoted by $ExtraPS_{(z_{0},y_{0},w_{0},p_{0},a)}$.

By the assumption, and according to the construction of the collapse-extra
MR diagram, \lemref{39} and \lemref{44}, for every ungraded closure
$ExtraPSCl_{h_{p_{0}}}$ in $ExtraPS_{(z_{0},y_{0},w_{0},p_{0},a)}$,
who is covered by the graded closure $ExtraPSGCl(u,z,y,w,p,a)$, we
can associate an ungraded closure $CollapseExtraFCl_{g_{p_{0}}}$,
that is induced from some $\Gamma$-exceptional solution $g_{p_{0}}$
of the base $p$-rigid or $p$-solid group of some graded formal closure
$CollapseExtraGFCl(c,u,z,y,w,p,a)$ in the collapse-extra MR diagram
$CollapseExtraMRD$ associated with $ExtraPSGCl(u,z,y,w,p,a)$, so
that the restriction of $g_{p_{0}}$ to $B$ is a $\Gamma$-exceptional
solution of $B$ that belongs to the $\Gamma$-exceptional family
of $(b_{0},p_{0},a)$, and that for every $\Gamma$-specialization
$(z_{0}',y_{0}',w_{0}',p_{0},a)$ that factors through $CollapseExtraFCl_{g_{p_{0}}}$,
and for every $u_{0}''\in\Gamma$, if $(u_{0}'',z_{0}',y_{0}',w_{0}',p_{0},a)$
is a specialization of $ExtraPSGCl(u,z,y,w,p,a)$ that restricts to
an exceptional family of $Term(ExtraPSGCl)$ that belongs to the exceptional
family of $h_{p_{0}}$, then $(u_{0}'',z_{0}',y_{0}',w_{0}',p_{0},a)$
is collapsed. The collection of all such (boundedly many) associated
ungraded closures $CollapseExtraFCl_{g_{p_{0}}}$, is denoted briefly
by $CollapseExtra_{p_{0}}$.

Now since the $\Gamma$-specialization $(z_{0},y_{0},w_{0},p_{0},a)$
of $Comp(z,y,w,p_{0},a)$, factors through each of the ungraded closures
in ungraded closures $CollapseExtra_{p_{0}}$, but does not factor
through any of the ungraded closures in $Pos_{p_{0}}$ and $ExtraPS_{p_{0}}\backslash ExtraPS_{(z_{0},y_{0},w_{0},p_{0},a)}$,
then, there must exist a $\Gamma$-test sequence $(z_{n},y_{n},w_{n},p_{0},a)$
of $Comp(z,y,w,p_{0},a)$, so that for all $n$, the $\Gamma$-specialization
$(z_{n},y_{n},w_{n},p_{0},a)$ factors through each of the closures
in $CollapseExtra_{p_{0}}$, and does not factor through any of the
closures in $Pos_{p_{0}}$ and \\
 $ExtraPS_{p_{0}}\backslash ExtraPS_{(z_{0},y_{0},w_{0},p_{0},a)}$.
We are going to prove that $(z_{n},y_{n},w_{n},p_{0},a)$ is eventually
a $\Gamma$-validPS statement.

According to \thmref[s]{27} and \ref{thm:35}, the remaining part
of the proof is identical to the case over $F_{k}$ presented in \thmref{45}. 
\end{proof}
\begin{thm}
\label{thm:47} Let $p_{0}\in F_{k}$. There exists a formula $\varphi(p)$
that belongs to the Boolean algebra of $AE$-formulas, so that: 
\begin{enumerate}
\item The formula $\varphi(p)$ defines the set $TSPS(p)$. 
\item Let $\Gamma$ be a group of level $l$ in the model. If $l$ is large
enough, then the formula $\varphi(p)$ defines the set $TSPS_{\Gamma}(p)$. 
\end{enumerate}
\end{thm}

\begin{proof}
We show how to construct $\varphi(p)$, and explain the reason that
it defines the sets $TSPS(p)$ and $TSPS_{\Gamma}(p)$.

Let $(z_{0},y_{0},w_{0},p_{0},a)$ be a {[}$\Gamma$-{]}specialization
that factors through $Comp(z,y,w,p,a)$.

The third condition of \thmref{45}{[}\thmref{46}{]}, is equivalent
to the following: for every graded terminal closure \\
 $ExtraPSCl(u,z,y,w,p,a)$ of $ExtraMRD$, there exists a {[}$\Gamma$-{]}full
family of representatives $h_{1},...,h_{r}$, for the {[}$\Gamma$-{]}exceptional
families of its base group $Term(ExtraPSCl)$ corresponding to $p=p_{0}$,
so that, for all $i=1,...,r$, if the {[}$\Gamma$-{]}specialization
$(z_{0},y_{0},w_{0},p_{0},a)$ factors through the ungraded closure
covered by $ExtraPSCl(u,z,y,w,p,a)$ and induced from $h_{i}$, and
$u_{0}^{i}$ is the corresponding induced {[}$\Gamma$-{]}specialization
of the variables $u$, then the {[}$\Gamma$-{]}specialization $(u_{0}^{i},z_{0},y_{0},w_{0},p_{0},a)$
factors through some terminal graded formal closure in $CollapseExtraMRD$
associated with $ExtraPSCl(u,z,y,w,p,a)$.

For simplicity, we denote by $E_{1},...,E_{r}$ all the terminal graded
closures in $ExtraPSMRD$, and by $C_{1},...,C_{s}$ all the graded
formal terminal closures in $CollapseExtraMRD$. We denote also the
terminal graded formal closures of $PosGFMRD$ by $S_{1},...,S_{q}$.

Hence, given any $p=p_{0}\in F_{k}$, if for every terminal graded
closure $E_{i}$ in $ExtraPSMRD$, the exact number $n_{i}$ of exceptional
families of solutions of its $p$-rigid or $p$-solid base group corresponding
to the value $p=p_{0}$ of the parameters is known, then there will
exist a specialization $(z_{0},y_{0},w_{0},p_{0},a)$ of $Comp(z,y,w,p,a)$
that satisfies the conditions in \thmref{45}, if and only if the
following sentence is true:

\begin{align*}
 & \exists z_{0},y_{0},w_{0},\stackrel[i=1]{r}{\exists}h_{1}^{i},...,h_{n_{i}}^{i}\stackrel[i=1]{r}{\exists}c_{1}^{i},...,c_{n_{i}}^{i},\quad\text{s.t.}\\
 & \left[(z_{0},y_{0},w_{0},p_{0},a)\text{ restricts to an exceptional solution for }B\right],\text{ and}\\
 & \left[\stackrel[i=1]{q}{\wedge}\forall c_{0}\,S_{i}(c_{0},z_{0},y_{0},w_{0},p_{0},a)\neq1\right],\text{ and}\\
 & \left[\stackrel[i=1]{r}{\wedge}\stackrel[j=1]{n_{i}}{\wedge}h_{j}^{i}\text{ is an exceptional solution for }Term(E_{i})\right],\text{ and}\\
 & \left[\stackrel[i=1]{r}{\wedge}\stackrel[j\neq j']{}{\wedge}h_{j}^{i}\text{ and }h_{j'}^{i}\text{ represent distinct exceptional families of }Term(E_{i})\right],\text{ and}\\
 & \left[\stackrel[i=1]{r}{\wedge}\stackrel[j=1]{n_{i}}{\wedge}\forall u_{0}\,\left[E_{h_{j}^{i}}(u_{0},z_{0},y_{0},w_{0},p_{0},a)=1\implies\stackrel[k=1]{s}{\vee}C_{k}(c_{j}^{i},u_{0},z_{0},y_{0},w_{0},p_{0},a)=1\right]\right].
\end{align*}

Since the numbers $n_{i}$ admits a uniform upper bound for all the
values $p=p_{0}\in F_{k}$ {[}$p=p_{0}\in\Gamma${]}, and since the
values $p=p_{0}$ for which a given $p$-rigid or $p$-solid limit
group admits exactly a given number of {[}$\Gamma$-{]}exceptional
families corresponding to $p=p_{0}$ form a set that belongs to the
Boolean algebra of $AE$-sets, a finite disjunction that runs on all
the possible values of the numbers $n_{i}$, and for each possibility
it lists the corresponding formula similar to the aforementioned one,
defines the desired formula $\varphi(p)$. 
\end{proof}

\subsection{Reducing the Set $EAE(p)$ under Minimal-Rank Assumption }

Now we explain how \thmref{47} can be applied iteratively, in order
to collect all the values $p_{0}\in F_{k}$ {[}resp. $p_{0}\in\Gamma${]},
for which there exists some validPS statement {[}resp. $\Gamma$-validPS
statement{]} that factors through the completed resolution $Comp(z,y,w,p,a)$.
For that aim, we use the method of constructing quotient resolutions
presented in \cite{DGIV}.
\begin{thm}
\label{thm:48} Let $Rlim(y,p,a)$ be a minimal rank $p$-graded limit
group, and let $Res(y,p,a)$ be a graded resolution of $Rlim(y,p,a)$.
Let $Comp(z,y,p,a)$ be the graded completion of $Res(y,p,a)$, and
denote its base $p$-rigid or $p$-solid group by $B$. Let $GCl(u,z,y,p,a)$
be a graded closure for $Comp(z,y,p,a)$. Let $\Sigma=\{\Sigma_{i}(c,u,z,y,p,a)\}_{i=1,...,r}$
be a finite collection of systems. Let $\mathscr{A}$ be the collection
of all the specializations $(u_{0},z_{0},y_{0},p_{0},a)$ for which: 
\begin{enumerate}
\item The specialization $(u_{0},z_{0},y_{0},p_{0},a)$ factors through
$GCl(u,z,y,p,a)$. 
\item The specialization $(u_{0},z_{0},y_{0},p_{0},a)$ restricts to an
exceptional specialization of $Term(GCl)$ and to an exceptional specialization
of $B$. 
\item There exists some $c_{0}\in F_{k}$, so that $\Sigma_{i}(c_{0},u_{0},z_{0},y_{0},p_{0},a)$
for some $i=1,...,r$. 
\end{enumerate}
Then, there exists finitely many $p$-graded completed resolutions
$Comp_{1}^{2}(t,c,u,z,y,p,a),...,Comp_{s}^{2}(t,c,u,z,y,p,a)$, so
that for all $j=1,...,s$: 
\begin{enumerate}
\item The graded closure $GCl(u,z,y,p,a)$ is mapped canonically into $Comp_{j}^{2}(t,c,u,z,y,p,a)$. 
\item The words of the system $\Sigma_{i}(c,u,z,y,p,a)$, for some $i=1,...,r$,
represent the trivial element in $Comp_{j}^{2}(t,c,u,z,y,p,a)$. 
\item The image of $Term(GCl)$ is contained in the base group $Term(Comp_{j}^{2})$,
and is mapped compatibly, i.e., the minimal subgraph of $Term(Comp_{j}^{2})$
containing the image of $Term(GCl)$, in the $p$-JSJ of $Term(Comp_{j}^{2})$,
consists only from parts of the $p$-JSJ of $Term(GCl)$. 
\item The complexity of $Comp_{j}^{2}(t,c,u,z,y,p,a)$ is not larger than
the complexity of $Comp(z,y,p,a)$ (definition 1.16 in \cite{DGV}). 
\item If the complexity of $Comp_{j}^{2}(t,c,u,z,y,p,a)$ equals the complexity
of $Comp(z,y,p,a)$, then $Comp_{j}^{2}(t,c,u,z,y,p,a)$ is a graded
formal closure $GFCl(c,u,z,y,p,a)$ of $Comp(z,y,p,a)$. 
\end{enumerate}
Moreover, for every specialization $(u_{0},z_{0},y_{0},p_{0},a)$
in the collection $\mathscr{A}$, there exists a specialization $(u_{0}',z_{0},y_{0},p_{0},a)$
that belongs to $\mathscr{A}$, so that $(u_{0}',z_{0},y_{0},p_{0},a)$
factors through one of the resolutions

\[
Comp_{1}^{2}(t,c,u,z,y,p,a),...,Comp_{s}^{2}(t,c,u,z,y,p,a)\,,
\]
and the specializations $(u_{0},z_{0},y_{0},p_{0},a)$ and $(u_{0}',z_{0},y_{0},p_{0},a)$
restrict to the same exceptional family of $Term(GCl)$.

We call the resolutions $Comp_{1}^{2},...,Comp_{s}^{2}$, the quotient
resolutions of the systems $\Sigma$ w.r.t. the graded closure $GCl(u,z,y,p,a)$. 
\end{thm}

\begin{proof}
The collection $\mathscr{A}$ factors through finitely many maximal
limit groups $QRlim(c,u,z,y,p,a)$. Since $Rlim(y,p,a)$ is assumed
to be a minimal rank graded limit group, if $QRlim(c,u,z,y,p,a)$
is $p$-freely decomposable, then the factor containing the subgroup
$\langle a,p\rangle$ must contain the whole image of the graded completion
$Comp(z,y,p,a)$. Since the specializations in $\mathscr{A}$ restrict
to exceptional specializations for $Term(GCl)$, we conclude that
the whole image of the graded closure $GCl(u,z,y,p,a)$ is contained
in the free factor of $QRlim(c,u,z,y,p,a)$ containing the subgroup
$\langle a,p\rangle$, in the case that $QRlim(c,u,z,y,p,a)$ is $p$-freely
decomposable. Hence, By mapping the factors that does not intersect
a conjugate of the subgroup $\langle a,p\rangle$ to the trivial element,
we may assume that $QRlim(c,u,z,y,p,a)$ is $p$-freely indecomposable.

Now we denote by $zup_{b}^{1}$, a set of generators for the first
level of the graded closure $GCl(u,z,y,p,a)$. We recognize the limit
group $QRlim(c,u,z,y,p,a)$ as a $zup_{b}^{1}$-graded limit group
$QRlim(c,u,z,y,p,a)=QRlim(c,u,z,y,zup_{b}^{1},a)$. The subcollection
of specializations in $\mathscr{A}$ that factor through $QRlim$,
induces a $zup_{b}^{1}$-graded MR diagram of the $zup_{b}^{1}$-graded
limit group $QRlim$. Let $QRes(c,u,z,y,zup_{b}^{1},a)$ be a resolution
in this MR diagram.

Let $Q_{1},...,Q_{t}$ be some surviving surfaces (definitions 1.7
and 1.10 in \cite{DGIV}) that appear in the $p$-JSJ of the first
level of the given resolution $Res(y,p,a)$. That is, $Q_{i}$ is
mapped monomorphically onto a subgroup of finite index of a $QH$
subgroup of the $j$-th level $zup_{b}^{1}$-graded limit group $Glim_{j}(c,u,z,y,zup_{b}^{1},a)$,
for some $j$. If $Q$ is a surviving surface, and $j$ is the minimal
integer for which $Q$ is mapped monomorphically onto a subgroup of
finite index of a $QH$ subgroup of the $j$-th level $zup_{b}^{1}$-graded
limit group $Glim_{j}(c,u,z,y,zup_{b}^{1},a)$, we denote by $Q'$
its isomorphic image in $Glim_{j}$.

Let $j$ be an integer and assume that all the images $Q_{1}',...,Q_{t}'$
appear in the $j$-th level of the $zup_{b}^{1}$-graded limit group
$Glim_{j}(c,u,z,y,zup_{b}^{1},a)$. Assume further that for all $i=1,...,t$,
there is no surviving surface $Q$, other than $Q_{i}$, that is mapped
non-trivially into a conjugate of $Q_{i}'$, and that $j$ is the
minimal integer for which the group $Glim_{j}$ contains the subgroup
$Q'$ for some surviving surface $Q$.

According to the proof of theorem 1.7 in \cite{DGIV}, we can isolate
the images of the surviving surfaces $Q_{1},...,Q_{t}$ along the
subresolution $QRes_{j}$, consisting of the first $j$ levels of
$QRes$. We map the isolated $QH$ subgroups $Q_{1}',...,Q_{t}'$
into the distinguished vertex containing $zup_{b}^{1}$, and continue
iteratively.

We finish up with finitely many $zup_{b}^{1}$-graded resolutions,
denoted again by $QRes(c,u,z,y,zup_{b}^{1},a)$, in which all the
surviving surfaces are isolated. Each of the terminal groups of the
new resolutions is either $zup_{b}^{1}$-rigid or $zup_{b}^{1}$-solid,
and these resolutions cover all the specializations in the collection
$\mathscr{A}$.

We continue by considering the subcollection $\mathscr{A}_{1}$ of
$\mathscr{A}$ that consists of the specializations $(u_{0},z_{0},y_{0},p_{0},a)$
for which there exists some $c_{0}\in F_{k}$, so that $(c_{0},u_{0},z_{0},y_{0},p_{0},a)$
is a $zup_{b}^{1}$-exceptional specialization of $T_{1}(c,u,z,y,zup_{b}^{1},a)$,
the terminal group of $QRes(c,u,z,y,zup_{b}^{1},a)$.

The specializations in the collection $\mathscr{A}_{1}$ factors through
finitely many maximal limit groups $QRlim^{2}$. We denote by $zup_{b}^{2}$,
a set of generators for the second level of the graded closure $GCl(u,z,y,p,a)$.We
recognize each of the maximal limit groups $QRlim^{2}$ as a $zup_{b}^{2}$-graded
limit group. For each resolution in the $zup_{b}^{2}$-graded MR diagram
of $QRlim^{2}$, we use a similar construction to the one described
above, in order to isolate all the surviving surfaces that appear
in the second level of the graded closure $GCl(u,z,y,p,a)$. And we
continue iteratively.

This process terminates in finitely many resolutions denoted again
by $QRes(c,u,z,y,p,a)$. Each of the resolutions terminate in $p$-rigid
or $p$-solid group, and these resolutions cover all the specializations
in the collection $\mathscr{A}$. Moreover, according to the construction,
and using Propositions 1.11 and 1.12 of \cite{DGIV}, all the surviving
surfaces w.r.t. some obtained resolution $QRes(c,u,z,y,p,a)$, in
all the levels of the given resolution $Res(y,p,a)$, are isolated
along the resolution $QRes(c,u,z,y,p,a)$.

We consider the completion $Comp^{2}$ of some obtained resolution
$QRes(c,u,z,y,zup_{b}^{1},a)$. We are going now to modify the completed
resolution $Comp^{2}$. For each of the abelian vertex groups in the
decomposition $\Lambda_{1}$ associated with the top level of $Comp^{2}$,
we fix a basis. We map the $QH$ subgroups that appear in $\Lambda_{1}$
and do not intersect a conjugate of the image of $GCl$ non-trivially,
into the next level of $Comp^{2}$. For each generator of an abelian
vertex in $\Lambda_{1}$, if the cyclic subgroup generated by this
generator, does not intersect a conjugate of the image of $GCl$ non-trivially,
we map it to the next level of $Comp^{2}$.

We continue modifying the various levels of $Comp^{2}$ successively,
until we reach the terminal group of $Comp^{2}$, that we do not change.
According to the analysis of quotient resolutions in the minimal rank
case, that appear in Lemmas 1.3, 1.5, 1.11, and 1.12 in \cite{DGIV},
the quadratic complexity of the (sub)resolution (of) $Comp^{2}$,
is no larger than that of the given resolution $Comp$. In the case
of equality, we can modify the tower structure of $Comp^{2}$ by ``dropping''
some levels that their decompositions do not contain any isolated
surviving surface, and by ``merging'' remaining levels, so that
the quadratic complexity do not change, and the abelian rank of the
obtained resolution is no larger than the abelian rank of the given
resolution $Comp$. In the case of equality in the abelian ranks too,
the obtained resolution $Comp^{2}$ will be a graded formal closure
of $Comp$. 
\end{proof}
\begin{lem}
\label{lem:49} Let $Rlim(y,p,a)$ be a minimal rank $p$-graded limit
group, and let $Res(y,p,a)$ be a graded resolution of $Rlim(y,p,a)$.
Let $Comp(z,y,p,a)$ be the graded completion of $Res(y,p,a)$, and
denote its base $p$-rigid or $p$-solid group by $B$. Let $GCl(u,z,y,p,a)$
be a graded closure for $Comp(z,y,p,a)$. Let $\Sigma=\{\Sigma_{i}(c,u,z,y,p,a)\}_{i=1,...,r}$
be a finite collection of systems. Let $\mathscr{A}$ be the collection
of all the specializations $(u_{0},z_{0},y_{0},p_{0},a)$ for which: 
\begin{enumerate}
\item The specialization $(u_{0},z_{0},y_{0},p_{0},a)$ factors through
$GCl(u,z,y,p,a)$. 
\item The specialization $(u_{0},z_{0},y_{0},p_{0},a)$ restricts to an
exceptional specialization of $Term(GCl)$ and to an exceptional specialization
of $B$. 
\item There exists some $c_{0}\in F_{k}$, so that $\Sigma_{i}(c_{0},u_{0},z_{0},y_{0},p_{0},a)$
for some $i=1,...,r$. 
\end{enumerate}
Let $Comp_{1}^{2}(t,c,u,z,y,p,a),...,Comp_{s}^{2}(t,c,u,z,y,p,a)$
be the $p$-graded completed resolutions from \ref{thm:48}.

Let $\Gamma$ be a group of level $l$ in the model. If $l$ is large
enough, then, the following property is satisfied for $\Gamma$. Let
$\Gamma\mathscr{A}$ be the collection of all the $\Gamma$-specializations
$(u_{0},z_{0},y_{0},p_{0},a)$ for which: 
\begin{enumerate}
\item The $\Gamma$-specialization $(z_{0},y_{0},p_{0},a)$ factors through
$GCl(u,z,y,p,a)$. 
\item The $\Gamma$-specialization $(u_{0},z_{0},y_{0},p_{0},a)$ restricts
to an exceptional specialization of $Term(GCl)$ and to an exceptional
specialization of $B$. 
\item There exists some $c_{0}\in\Gamma$, so that $\Sigma_{i}(c_{0},u_{0},z_{0},y_{0},p_{0},a)$
for some $i=1,...,r$. 
\end{enumerate}
Then, the resolutions $Comp_{1}^{2}(t,c,u,z,y,p,a),...,Comp_{s}^{2}(t,c,u,z,y,p,a)$
cover the collection $\Gamma\mathscr{A}$.

For every $\Gamma$-specialization $(u_{0},z_{0},y_{0},p_{0},a)$
in the collection $\Gamma\mathscr{A}$, there exists a $\Gamma$-specialization
$(u_{0}',z_{0},y_{0},p_{0},a)$ that belongs to $\Gamma\mathscr{A}$,
so that $(u_{0}',z_{0},y_{0},p_{0},a)$ factors through one of the
resolutions

\[
Comp_{1}^{2}(t,c,u,z,y,p,a),...,Comp_{s}^{2}(t,c,u,z,y,p,a)\,,
\]
and the $\Gamma$-specializations $(u_{0},z_{0},y_{0},p_{0},a)$ and
$(u_{0}',z_{0},y_{0},p_{0},a)$ restrict to the same $\Gamma$-exceptional
family of $Term(GCl)$. 
\end{lem}

\begin{proof}
Let $(u_{0},z_{0},y_{0},p_{0},a)$ be a $\Gamma$-specialization in
$\Gamma\mathscr{A}$, and let $c_{0}\in\Gamma$ be so that $\Sigma_{i}(c_{0},u_{0},z_{0},y_{0},p_{0},a)$
for some $i=1,...,r$. We lift $(c_{0},u_{0},z_{0},y_{0},p_{0},a)$
to a specialization $(\tilde{c}_{0},\tilde{u}_{0},\tilde{z}_{0},\tilde{y}_{0},\tilde{p}_{0},a)$
in $F_{k}$, so that $(\tilde{u}_{0},\tilde{z}_{0},\tilde{y}_{0},\tilde{p}_{0},a)$
is a specialization of $GCl$, and $\Sigma_{i}(\tilde{c}_{0},\tilde{u}_{0},\tilde{z}_{0},\tilde{y}_{0},\tilde{p}_{0},a)=1$.

Hence, the specialization $(\tilde{u}_{0},\tilde{z}_{0},\tilde{y}_{0},\tilde{p}_{0},a)$
belongs to the collection $\mathscr{A}$, and according to \thmref{48},
there exists a specialization $(\tilde{u}_{0}',\tilde{z}_{0},\tilde{y}_{0},\tilde{p}_{0},a)$
that belongs to $\mathscr{A}$, so that $(\tilde{u}_{0}',\tilde{z}_{0},\tilde{y}_{0},\tilde{p}_{0},a)$
factors through one of the resolutions $Comp_{1}^{2}(t,c,u,z,y,p,a),...,Comp_{s}^{2}(t,c,u,z,y,p,a)$,
and the specializations $(\tilde{u}_{0},\tilde{z}_{0},\tilde{y}_{0},\tilde{p}_{0},a)$
and $(\tilde{u}_{0}',\tilde{z}_{0},\tilde{y}_{0},\tilde{p}_{0},a)$
restrict to the same exceptional family of $Term(GCl)$. We denote
by $(u_{0}',z_{0},y_{0},p_{0},a)$, the projection of $(\tilde{u}_{0}',\tilde{z}_{0},\tilde{y}_{0},\tilde{p}_{0},a)$
to $\Gamma$. Then, $(u_{0}',z_{0},y_{0},p_{0},a)$ belongs to the
collection $\Gamma\mathscr{A}$ and factors through the resolution
$Comp_{j}^{2}(t,c,u,z,y,p,a)$.

If the induced $\Gamma$-specialization $(c_{0},u_{0}'',z_{0}',y_{0}',p_{0},a)$
of the terminal $p$-rigid or $p$-solid group $Term(Comp_{j}^{2})$
of $Comp_{j}^{2}$ is not $\Gamma$-exceptional, we lift it to a non-exceptional
specialization $(\tilde{c}_{0},\tilde{u}_{0}'',\tilde{z}_{0}',\tilde{y}_{0}',\tilde{p}_{0},a)$
of $Term(Comp_{j}^{2})$ in $F_{k}$. The specialization $(u_{0}'',z_{0}',y_{0}',p_{0},a)$
belongs to $\Gamma\mathscr{A}$ and hence $(\tilde{u}_{0}'',\tilde{z}_{0}',\tilde{y}_{0}',\tilde{p}_{0},a)$
belongs to $\mathscr{A}$. Thus, and according to the construction
of the resolutions $Comp_{1}^{2},...,Comp_{s}^{2}$, there exists
some resolution $Comp_{j'}^{2}$, for $j'\neq j$, whose upper levels
are identical to that of $Comp_{j}^{2}$ but its terminal group is
a $p$-rigid or $p$-solid limit quotient of the $Term(Comp_{j}^{2})$,
so that there exists a specialization $(\tilde{u}_{0}^{(3)},\tilde{z}_{0}'',\tilde{y}_{0}'',\tilde{p}_{0},a)$
of $Term(Comp_{j'}^{2})$ that restrict to the same exceptional family
of $Term(GCl)$ represented by $(\tilde{u}_{0},\tilde{z}_{0},\tilde{y}_{0},\tilde{p}_{0},a)$.
In particular, the specialization $(\tilde{z}_{0}'',\tilde{y}_{0}'',\tilde{p}_{0},a)$
restrict to the same exceptional family of $B$ represented by $(\tilde{z}_{0},\tilde{y}_{0},\tilde{p}_{0},a)$.
Hence, there exists some $\tilde{u}_{0}^{(4)}\in F_{k}$, so that
the specialization $(\tilde{u}_{0}^{(4)},\tilde{z}_{0},\tilde{y}_{0},\tilde{p}_{0},a)$
factors through $Comp_{j'}$, and restricts to the same exceptional
family of $Term(GCl)$ represented by $(\tilde{u}_{0},\tilde{z}_{0},\tilde{y}_{0},\tilde{p}_{0},a)$.
We denote by $u_{0}^{(4)}$ the projection of $\tilde{u}_{0}^{(4)}$
to $\Gamma$. Then, the $\Gamma$-specialization $(u_{0}^{(4)},z_{0},y_{0},p_{0},a)$
factors through $Comp_{j'}$, and restricts to the same $\Gamma$-exceptional
family represented by $(u_{0},z_{0},y_{0},p_{0},a)$. A finite induction
argument, implies that we can assume that the induced $\Gamma$-specialization
of $Term(Comp_{j'}^{2})$ is $\Gamma$-exceptional. 
\end{proof}
For each terminal graded closure $ExtraPSGCl(u,z,y,w,p,a)$ in the
$ExtraPS$ MR diagram $ExtraPSMRD$ of the collection $WPRS$ over
the resolution $Comp(z,y,w,p,a)$, we consider the finitely many systems
$\Sigma=\{\Sigma_{i}(c,u,z,y,w,p,a)\}_{i}$ stating that the specializations
of the variables $u$ in the specialization $(u,z,y,w,p,a)$ are collapsed.
Let $Comp_{1}^{2}(t,c,u,z,y,w,p,a),...,Comp_{s}^{2}(t,c,u,z,y,w,p,a)$
be the quotient resolutions of the systems $\Sigma$ w.r.t. the graded
closure $ExtraPSGCl(u,z,y,w,p,a)$. 
\begin{thm}
\label{thm:50} 
\begin{enumerate}
\item Let $(z_{0},y_{0},w_{0},p_{0},a)$ be a validPS statement that factors
through the completed resolution $Comp(z,y,w,p,a)$. If $p_{0}\notin TSPS(p)$,
then $(z_{0},y_{0},w_{0},p_{0},a)$ factors through some quotient
resolution $Comp^{2}(t,c,u,z,y,w,p,a)$, associated with some terminal
graded closure $ExtraPSGCl(u,z,y,w,p,a)$ in the $ExtraPS$ MR diagram
$ExtraPSMRD$, so that the complexity of $Comp^{2}$ is strictly smaller
than that of $Comp(z,y,w,p,a)$. 
\item Let $\Gamma$ be a group of level $l$ in the model. If $l$ is large
enough, then, the following property is satisfied for $\Gamma$. Let
$(z_{0},y_{0},w_{0},p_{0},a)$ be a $\Gamma$-validPS statement that
factors through the completed resolution $Comp(z,y,w,p,a)$. If $p_{0}\notin TSPS_{\Gamma}(p)$,
then $(z_{0},y_{0},w_{0},p_{0},a)$ factors through some quotient
resolution $Comp^{2}(t,c,u,z,y,w,p,a)$, associated with some terminal
graded closure $ExtraPSGCl(u,z,y,w,p,a)$ in the $ExtraPS$ MR diagram
$ExtraPSMRD$, so that the complexity of $Comp^{2}$ is strictly smaller
than that of $Comp(z,y,w,p,a)$. 
\end{enumerate}
\end{thm}

\begin{proof}
We start with the first part of the claim. Since the specialization
$(z_{0},y_{0},w_{0},p_{0},a)$ is a validPS statement, it does not
factor through the diagram $PosGFMRD$. Let $ExtraPSGCl(u,z,y,w,p,a)$
be a terminal graded closure in the $ExtraPS$ MR diagram $ExtraPSMRD$,
and assume that there exists some $u_{0}\in F_{k}$, so that $(u_{0},z_{0},y_{0},w_{0},p_{0},a)$
factors through $ExtraPSGCl(u,z,y,w,p,a)$.

Since $(z_{0},y_{0},w_{0},p_{0},a)$ is a validPS statement, the specialization
$u_{0}$ must be collapsed in the specialization $(u_{0},z_{0},y_{0},w_{0},p_{0},a)$.
Hence, and according to \thmref{48}, there exists some $u_{0}'\in F_{k}$,
so that $(u_{0}',z_{0},y_{0},w_{0},p_{0},a)$ factors through some
quotient resolution $Comp^{2}(t,c,u,z,y,w,p,a)$ associated with \\
 $ExtraPSGCl(u,z,y,w,p,a)$, and the specializations $(u_{0},z_{0},y_{0},w_{0},p_{0},a)$
and $(u_{0}',z_{0},y_{0},w_{0},p_{0},a)$ restrict to the same exceptional
family of $Term(ExtraPSGCl)$.

If the complexity of $Comp^{2}$ equals that of $Comp(z,y,w,p,a)$,
then $Comp^{2}$ is a graded formal closure of $Comp(z,y,w,p,a)$.
Hence, according to \lemref{41}, in this case, the specialization
$(u_{0}',z_{0},y_{0},w_{0},p_{0},a)$ factors through the $CollapseExtra$
MR diagram $CollapseExtraMRD$ associated with the graded closure
$ExtraPSGCl(u,z,y,w,p,a)$.

Hence, and since $p_{0}\notin TSPS(p)$, there must exist some terminal
graded closure $ExtraPSGCl(u,z,y,w,p,a)$ in the $ExtraPS$ MR diagram
$ExtraPSMRD$, and some $u_{0}\in F_{k}$, so that $(u_{0},z_{0},y_{0},w_{0},p_{0},a)$
factors through $ExtraPSGCl(u,z,y,w,p,a)$ and for all $u_{0}'\in F_{k}$,
for which the specializations $(u_{0},z_{0},y_{0},w_{0},p_{0},a)$
and $(u_{0}',z_{0},y_{0},w_{0},p_{0},a)$ restrict to the same exceptional
family of $Term(ExtraPSGCl)$, the specialization $(u_{0}',z_{0},y_{0},w_{0},p_{0},a)$
does not factor through any of the quotient resolutions $Comp^{2}$
associated with \\
 $ExtraPSGCl(u,z,y,w,p,a)$ and have the same complexity as that of
$Comp(z,y,w,p,a)$. That is, $(z_{0},y_{0},w_{0},p_{0},a)$ must factor
through some quotient resolution $Comp^{2}$ associated with $ExtraPSGCl(u,z,y,w,p,a)$,
with complexity strictly smaller than that of $Comp(z,y,w,p,a)$.

The second statement follows by a similar argument that uses \lemref[s]{49}
and \ref{lem:42} in place of \thmref[s]{48} and \ref{lem:41}. 
\end{proof}
In light of \thmref{50}, for every terminal graded closure $ExtraPSGCl(u,z,y,w,p,a)$
in the $ExtraPS$ MR diagram $ExtraPSMRD$ of the collection $WPRS$
over the resolution $Comp(z,y,w,p,a)$, we collect those quotient
resolutions $Comp^{2}$ associated with $ExtraPSGCl(u,z,y,w,p,a)$,
who admit a strictly smaller complexity than that of $Comp(z,y,w,p,a)$,
and for which there exists some non-collapsed specialization $(z_{0},y_{0},w_{0},p_{0},a)$
of $Comp(z,y,w,p,a)$ that factors through $Comp^{2}$. According
to \thmref{47}, the sets $TSPS^{2}(p)$ and $TSPS_{\Gamma}^{2}(p)$,
that correspond to the new resolution $Comp^{2}$, are defined by
one formula $\varphi^{2}(p)$ in the Boolean algebra of $AE$-formulas,
for every group $\Gamma$ that belongs to sufficiently high level
of the model.

We Continue iteratively. By the definition of the complexity of a
graded resolution, this process must terminate after finitely many
steps. In each step, we have a formula $\varphi^{j}(p)$ in the Boolean
algebra of $AE$-formulas, that defines a subset of $EAE(p)$ and
a subset of $EAE_{\Gamma}(p)$, for every group $\Gamma$ that belongs
to sufficiently high level of the model. Moreover, all the quotient
resolutions in the final step of this procedure, must be of maximal
complexity. Hence, every $p_{0}\in EAE(p)$ for which there exists
some validPS statement that factors through the initial resolution
$Comp(z,y,w,p,a)$, must be defined by some formula $\varphi^{j}(p)$
along the process. Similarly, if $\Gamma$ belongs to sufficiently
high level in the model, then, every $p_{0}\in EAE_{\Gamma}(p)$ for
which there exists some $\Gamma$-validPS statement that factors through
the initial resolution $Comp(z,y,w,p,a)$, must be defined by some
formula $\varphi^{j}(p)$ along the process.

\section{Proof Systems}\label{sec:Proof-Systems}

Assume that we have a $EAE(p)$ set over $F_{k}$ defined by the formula:

\[
\exists w\;\forall y\;\exists x\quad\Sigma(x,y,w,p,a)=1\,\wedge\,\Psi(x,y,w,p,a)\neq1\,.
\]
Given a group $\Gamma$, we denote by $EAE_{\Gamma}(p)$, the set
of $\Gamma$-specializations $p_{0}\in\Gamma$ of the parameters $p$,
that are defined by the above formula when applied on $\Gamma$. We
also denote by $AE(w,p)$ and by $AE_{\Gamma}(w,p)$ the sets defined
by the formula

\[
\forall y\;\exists x\quad\Sigma(x,y,w,p,a)=1\,\wedge\,\Psi(x,y,w,p,a)\neq1\,,
\]
over $F_{k}$ and over $\Gamma$, correspondingly. If $(w_{0},p_{0})\in AE(w,p)$,
we say that $w_{0}$ is a \emph{witness} for $p_{0}\in EAE(p)$. Note
that for every $p_{0}\in EAE(p)$ there must exist some witness $w_{0}\in F_{k}$
so that $(w_{0},p_{0})\in AE(w,p)$. Similarly, if $(w_{0},p_{0})\in AE_{\Gamma}(w,p)$,
we say that $w_{0}$ is a \emph{witness} for $p_{0}\in EAE_{\Gamma}(p)$,
and we note that for every $p_{0}\in EAE_{\Gamma}(p)$ there must
exist some witness $w_{0}\in\Gamma$ so that $(w_{0},p_{0})\in AE_{\Gamma}(w,p)$.

Our aim in this section, is proving quantifier elimination. That is,
we will show that the sets $EAE(p)$ and $EAE_{\Gamma}(p)$ belong
to the Boolean algebra of $AE$-sets. Moreover, we show that there
exists a formula $\varphi(p)$ in the Boolean algebra of $AE$-formulas,
that defines each of these sets uniformly (over the corresponding
groups), for almost all the groups $\Gamma$. For that aim, we bring
the situation to the setting of \secref{An-Approximation-to-Abstract-EAE}.
For that, we construct finitely many (graded) \emph{proof systems}.
A proof system, roughly, is a finite tree whose branches consist of
alternating graded MR diagrams and graded formal MR diagrams. Using
the graded formal MR diagrams along a proof system, we can apply the
formal solutions of our fixed system $\Sigma(x,y,w,p,a)$ on the specializations
that factor through a graded resolution in the graded MR diagram that
was placed in the previous stage in the proof system. Those specializations
for which all of the formal solutions applied on them do not satisfy
our fixed set of inequalities $\Psi(x,y,w,p,a)\neq1$ (the remaining
$y$'s), are passed to the next graded MR diagram in the proof system,
and so on.

\subsection{The Construction of Proof Systems of Depth 2 under Minimal Rank Assumption}

Let $GFMRD_{1}$ be the graded formal MR diagram of the system $\Sigma(x,y,w,p,a)$
over the $wp$-graded resolution $F_{y}\ast F(w,p,a)$. According
to the construction of $GFMRD_{1}$, every graded formal resolution
$GFRes^{1}(x,y,w,p,a)$ in it, is defined over a graded formal closure
of $F_{y}\ast F(w,p,a)$, which is a limit group of the form $WPH(h_{1},w,p,a)\ast F_{y}$,
where $WPH(h_{1},w,p,a)$ is a $wp$-rigid or a $wp$-solid limit
group.

We apply all the modular automorphisms of the decompositions associated
with the various groups along the resolution $GFRes^{1}(x,y,w,p,a)$,
and we specialize the free factors that do not intersect a conjugate
of the subgroup $\langle y,w,p,a\rangle$ in all the possible ways.
That way, each word $\psi_{i}(x,y,w,p,a)$ in the system $\Psi(x,y,w,p,a)$,
induces a set of elements $\lambda'{}_{i}^{1}(h_{1},y,w,p,a)$ in
the group $WPH(h_{1},w,p,a)\ast F_{y}$. We think on the induced set
$\lambda'{}_{i}^{1}(h_{1},y,w,p,a)$ as a system of equations imposed
on the variables $(h_{1},y,w,p,a)$. By Guba's theorem, for all $i$,
the system $\lambda'{}_{i}^{1}(h_{1},y,w,p,a)$ is equivalent to a
finite system $\lambda_{i}^{1}(h_{1},y,w,p,a)$ over the free group
$F_{k}$.

Denote by $\mathfrak{J}^{2}$ the collection of all the specializations
$(h_{1,0},y_{0},w_{0},p_{0},a)$ of $WPH(h_{1},w,p,a)\ast F_{y}$,
so that $(h_{1,0},y_{0},w_{0},p_{0},a)$ restricts to an exceptional
specialization of $WPH(h_{1},w,p,a)$, and solves the system $\lambda_{i}^{1}(h_{1},y,w,p,a)$,
for some $i$. From the collection $\mathfrak{J}^{2}$, we construct
a $wp$-graded MR diagram, denoted $GRemMRD_{2}$. Note that the $wp$-graded
limit group $WPH(h_{1},w,p,a)$ is mapped compatibly into each of
the $wp$-graded limit groups in this MR diagram $GRemMRD_{2}$. Every
resolution $\lambda^{1}WPGRes(h_{1},y,w,p,a)$ in the diagram $GRemMRD_{2}$,
terminates in a group $WPHG(g_{1},h_{1}',w,p,a)$ which is either
$wp$-rigid or $wp$-solid. The variables $g_{1},h_{1}'$ stand for
the canonical images of the variables $y,h_{1}$ correspondingly in
the terminal group $WPHG(g_{1},h_{1}',w,p,a)$ of the resolution $\lambda^{1}WPGRes(h_{1},y,w,p,a)$.
We denote the collection of all the resolutions in the diagram $GRemMRD_{2}$
briefly by $GRem_{2}$.

We denote the first level groups in $GRemMRD_{2}$, i.e., the finitely
many maximal limit groups associated with the collection $\mathfrak{J}^{2}$,
by $\lambda^{1}WPGL_{1}(h_{1},y,w,p,a),...,\lambda^{1}WPGL_{d}(h_{1},y,w,p,a)$. 
\begin{lem}
\label{lem:51} Let $\Gamma$ be a group in the model. Then, $\lambda_{i}^{1}(h_{1},y,w,p,a)$
is equivalent to $\lambda'{}_{i}^{1}(h_{1},y,w,p,a)$ over $\Gamma$.
Moreover, if $(h_{1,0},y_{0},w_{0},p_{0},a)$ is a $\Gamma$-specialization
of $WPH(h_{1},w,p,a)\ast F_{y}$, so that $(h_{1,0},y_{0},w_{0},p_{0},a)$
restricts to a $\Gamma$-exceptional specialization of $WPH(h_{1},w,p,a)$,
and solves the system $\lambda_{i}^{1}(h_{1},y,w,p,a)$, for some
$i$, then $(h_{1,0},y_{0},w_{0},p_{0},a)$ factors through one of
the resolutions in $GRem_{2}$. 
\end{lem}

\begin{proof}
Let $(h_{1,0},y_{0},w_{0},p_{0},a)$ be a $\Gamma$-solution for the
system $\lambda_{i}^{1}(h_{1},y,w,p,a)$. We lift $(h_{1,0},y_{0},w_{0},p_{0},a)$
to a solution $(\tilde{h}_{1,0},\tilde{y}_{0},\tilde{w}_{0},\tilde{p}_{0},a)$
of the system $\lambda_{i}^{1}(h_{1},y,w,p,a)$ in $F_{k}$. Since
$\lambda_{i}^{1}(h_{1},y,w,p,a)$ and $\lambda'{}_{i}^{1}(h_{1},y,w,p,a)$
are equivalent over $F_{k}$, then $(\tilde{h}_{1,0},\tilde{y}_{0},\tilde{w}_{0},\tilde{p}_{0},a)$
is a solution for $\lambda'{}_{i}^{1}(h_{1},y,w,p,a)$ in $F_{k}$.
Hence, the projection of $(\tilde{h}_{1,0},\tilde{y}_{0},\tilde{w}_{0},\tilde{p}_{0},a)$
to $\Gamma$, which is $(h_{1,0},y_{0},w_{0},p_{0},a)$, is a solution
for $\lambda'{}_{i}^{1}(h_{1},y,w,p,a)$ in $\Gamma$.

Assume further that $(h_{1,0},w_{0},p_{0},a)$ is a $\Gamma$-exceptional
specialization of $WPH(h_{1},w,p,a)$. Then $(\tilde{h}_{1,0},\tilde{w}_{0},\tilde{p}_{0},a)$
is an exceptional specialization of $WPH(h_{1},w,p,a)$ in $F_{k}$,
and the specialization $(\tilde{h}_{1,0},\tilde{y}_{0},\tilde{w}_{0},\tilde{p}_{0},a)$
belongs to the collection $\mathfrak{J}^{2}$. Hence, by the definition
of $GRem_{2}$, the specialization $(\tilde{h}_{1,0},\tilde{y}_{0},\tilde{w}_{0},\tilde{p}_{0},a)$
factors through some resolution in $GRem_{2}$, and thus, the $\Gamma$-specialization
$(h_{1,0},y_{0},w_{0},p_{0},a)$ factors through that resolution.
If the $\Gamma$-specialization of the terminal group $WPHG(g_{1},h_{1}',w,p,a)$
of that resolution is not $\Gamma$-exceptional, we lift it to a non-exceptional
specialization of $WPHG(g_{1},h_{1}',w,p,a)$. This lift must belong
to the collection $\mathfrak{J}^{2}$, and we can continue the factorization
through the flexible quotients of $WPHG(g_{1},h_{1}',w,p,a)$. Hence,
a finite induction implies the desired. 
\end{proof}
\begin{lem}
\label{lem:52}$ $ 
\begin{enumerate}
\item Let $(h_{1,0},w_{0},p_{0},a)$ be an exceptional specialization of
$WPH(h_{1},w,p,a)$, and assume that it factors through one of the
groups $\lambda^{1}WPGL_{1}(h_{1},y,w,p,a),...,\lambda^{1}WPGL_{d}(h_{1},y,w,p,a)$.
Then, every specialization of $WPH(h_{1},w,p,a)$ that belongs to
the exceptional family represented by $(h_{1,0},w_{0},p_{0},a)$,
factors through one of the groups $\lambda^{1}WPGL_{1}(h_{1},y,w,p,a),...,\lambda^{1}WPGL_{d}(h_{1},y,w,p,a)$. 
\item Let $\Gamma$ be a group in the model. Let $(h_{1,0},w_{0},p_{0},a)$
be a $\Gamma$-exceptional specialization of $WPH(h_{1},w,p,a)$,
and assume that it factors through one of the groups $\lambda^{1}WPGL_{1}(h_{1},y,w,p,a),...,\lambda^{1}WPGL_{d}(h_{1},y,w,p,a)$.
Then, every $\Gamma$-specialization of $WPH(h_{1},w,p,a)$ that belongs
to the $\Gamma$-exceptional family represented by $(h_{1,0},w_{0},p_{0},a)$,
factors through one of the groups $\lambda^{1}WPGL_{1}(h_{1},y,w,p,a),...,\lambda^{1}WPGL_{d}(h_{1},y,w,p,a)$. 
\end{enumerate}
\end{lem}

\begin{proof}
Assume that the specialization $(h_{1,0},y_{0},w_{0},p_{0},a)$ factors
through one of the groups

\[
\lambda^{1}WPGL_{1}(h_{1},y,w,p,a),...,\lambda^{1}WPGL_{d}(h_{1},y,w,p,a)\,,
\]
for some $y_{0}\in F_{k}$. Then, the specialization $(h_{1,0},y_{0},w_{0},p_{0},a)$
is a solution for the system $\lambda_{i}^{1}(h_{1},y,w,p,a)$, for
some $i$. By the definition of the set of elements $\lambda'{}_{i}^{1}(h_{1},y,w,p,a)$
of the group $WPH(h_{1},w,p,a)\ast F_{y}$, if $\theta$ is a modular
automorphism of the $wp$-graded limit group $WPH(h_{1},w,p,a)$,
then, applying $\theta$ on the set of elements $\lambda'{}_{i}^{1}(h_{1},y,w,p,a)$,
does not change this set. Hence, if $(h_{1,0}\theta,y_{0},w_{0},p_{0},a)$
is the specialization of $WPH(h_{1},w,p,a)\ast F_{y}$ obtained by
precomposing $\theta$ with the specialization $(h_{1,0},y_{0},w_{0},p_{0},a)$,
then $(h_{1,0}\theta,y_{0},w_{0},p_{0},a)$ is a solution for the
system $\lambda_{i}^{1}(h_{1},y,w,p,a)$.

For the second part of the claim, let $(h_{1,0},y_{0},w_{0},p_{0},a)$
be a $\Gamma$-specialization of one of the groups $\lambda^{1}WPGL_{j}(h_{1},y,w,p,a)$,
for some $j=1,...,d$ and some $y_{0}\in\Gamma$. We lift $(h_{1,0},y_{0},w_{0},p_{0},a)$
to a specialization $(\tilde{h}_{1,0},\tilde{y}_{0},\tilde{w}_{0},\tilde{p}_{0},a)$
of $\lambda^{1}WPGL_{j}(h_{1},y,w,p,a)$ in $F_{k}$. Then, by the
above argument, every specialization of $WPH(h_{1},w,p,a)$ that belongs
to the same exceptional family represented by $(\tilde{h}_{1,0},\tilde{w}_{0},\tilde{p}_{0},a)$,
factors through one of the groups $\lambda^{1}WPGL_{1}(h_{1},y,w,p,a),...,\lambda^{1}WPGL_{d}(h_{1},y,w,p,a)$.
Hence, every $\Gamma$-specialization of $WPH(h_{1},w,p,a)$ that
belongs to the same $\Gamma$-exceptional family represented by $(h_{1,0},w_{0},p_{0},a)$,
factors through one of the groups $\lambda^{1}WPGL_{1}(h_{1},y,w,p,a),...,\lambda^{1}WPGL_{d}(h_{1},y,w,p,a)$. 
\end{proof}
Before we continue the construction of the proof systems, we stop
for an approximation of the sets $EAE(p)$ and $EAE_{\Gamma}(p)$
that is uniform for all the groups $\Gamma$ in the model. 
\begin{lem}
\label{lem:53} We denote: 
\begin{enumerate}
\item Let $U_{1}(p)$ be the set of all $p_{0}\in F_{k}$, for which there
exists $w_{0}\in F_{k}$, so that for one of the groups $WPH(h_{1},w,p,a)$,
there exists $h_{1,0}\in F_{k}$, so that $(h_{1,0},w_{0},p_{0},a)$
is an exceptional specialization of $WPH(h_{1},w,p,a)$, and for every
limit group $\lambda^{1}WPGL_{j}(h_{1},y,w,p,a)$ associated with
$WPH(h_{1},w,p,a)$, there exists no $y_{0}\in F_{k}$, so that $(h_{1,0},y_{0},w_{0},p_{0},a)$
is a specialization of $\lambda^{1}WPGL_{j}(h_{1},y,w,p,a)$. 
\item Given a group $\Gamma$ in the model, we let $U_{1}^{\Gamma}(p)$
be the set of all $p_{0}\in\Gamma$, for which there exists $w_{0}\in\Gamma$,
so that for one of the groups $WPH(h_{1},w,p,a)$, there exists $h_{1,0}\in\Gamma$,
so that $(h_{1,0},w_{0},p_{0},a)$ is a $\Gamma$-exceptional specialization
of $WPH(h_{1},w,p,a)$, and for every limit group $\lambda^{1}WPGL_{j}(h_{1},y,w,p,a)$
associated with $WPH(h_{1},w,p,a)$, there exists no $y_{0}\in\Gamma$,
so that $(h_{1,0},y_{0},w_{0},p_{0},a)$ is a $\Gamma$-specialization
of $\lambda^{1}WPGL_{j}(h_{1},y,w,p,a)$. 
\end{enumerate}
Then, 
\begin{enumerate}
\item $U_{1}(p)\subset EAE(p)$, and there exists a formula $\varphi_{1}(p)$
in the Boolean algebra of $AE$-formulas, that defines the set $U_{1}(p)$. 
\item Let $\Gamma$ be a group in the model. Then, $U_{1}^{\Gamma}(p)\subset EAE_{\Gamma}(p)$,
and the same formula $\varphi_{1}(p)$ above, defines the set $U_{1}^{\Gamma}(p)$
in $\Gamma$. 
\end{enumerate}
\end{lem}

\begin{proof}
The existence of a formula $\varphi_{1}(p)$ that satisfies the claim
is obvious. The fact that $U_{1}(p)\subset EAE(p)$ follows from the
definition of the limit groups $\lambda^{1}WPGL_{j}(h_{1},y,w,p,a)$
as the maximal limit groups through which the specializations in the
collection $\mathfrak{J}^{2}$ factor, and from \lemref{28}.

Now let $\Gamma$ be a group in the model. We prove that $U_{1}^{\Gamma}(p)\subset EAE_{\Gamma}(p)$.
Let $(h_{1,0},w_{0},p_{0},a)$ be a $\Gamma$-exceptional specialization
of $WPH(h_{1},w,p,a)$, and assume it does not factor through any
of the limit groups $\lambda^{1}WPGL_{j}(h_{1},y,w,p,a)$ associated
with $WPH(h_{1},w,p,a)$. Hence, according to \lemref{51}, every
$\Gamma$-specialization $(h_{1,0},y_{0},w_{0},p_{0},a)$ of $WPH(h_{1},w,p,a)\ast F_{y}$,
does not solve any of the systems $\lambda_{i}^{1}(h_{1},y,w,p,a)$.
Hence, and using \lemref{29}, we deduce that $p_{0}\in EAE_{\Gamma}(p)$. 
\end{proof}
Now we consider a $wp$-graded resolution $\lambda^{1}WPGRes(h_{1},y,w,p,a)$
in $GRem_{2}$, who terminates in a group $WPHG(g_{1},h_{1}',w,p,a)$
which is either $wp$-rigid or $wp$-solid. Let $Comp^{2}=Comp(\lambda^{1}WPGRes)(z,h_{1},y,w,p,a)$
be the completion of the resolution $\lambda^{1}WPGRes(h_{1},y,w,p,a)$,
and denote its base group by \\
 $B=WPHG(g_{1},h_{1}',w,p,a)$.

We construct the $wp$-graded formal MR diagram $GFMRD_{2}$ of the
system $\Sigma(x,y,w,p,a)$ over the resolution $Comp(\lambda^{1}WPGRes)(z,h_{1},y,w,p,a)$.
Every graded formal resolution $GFRes^{2}(x,z,h_{1},y,w,p,a)$ in
$GFMRD_{2}$, terminates in a graded formal closure $GFCl(x,z,h_{1},y,w,p,a)$
of $Comp(\lambda^{1}WPGRes)(z,h_{1},y,w,p,a)$.

We fix a generating set $(h_{2},g_{1},h_{1}',w,p,a)$ for the terminal
group $Term(GFCl)$ of $GFCl(x,z,h_{1},y,w,p,a)$, and denote $Term(GFCl)=WPHGH(h_{2},g_{1},h_{1}',w,p,a)$.
The elements $(g_{1},h_{1}',w,p,a)$ of \\
 $WPHGH(h_{2},g_{1},h_{1}',y,w,p,a)$ are the canonical images of
the elements $(g_{1},h_{1}',w,p,a)$ of $B$ in $Term(GFCl)$. Note
that, since $GFCl(x,z,h_{1},y,w,p,a)$ is a graded formal closure
for the resolution $Comp(\lambda^{1}WPGRes)$, it is generated, as
a group, by $Term(GFCl)$ and (the canonical image of) $Comp(\lambda^{1}WPGRes)$,
together with a finite collection of elements $s$, that are roots
for some elements in the abelian vertex groups that appear in the
decompositions associated with the various levels of $Comp(\lambda^{1}WPGRes)$.

We apply all the modular automorphisms of the decompositions associated
with the various groups along the resolution $GFRes^{2}(x,z,h_{1},y,w,p,a)$,
and we specialize the free factors that do not intersect a conjugate
of the subgroup $\langle z,h_{1},y,w,p,a\rangle$ in all the possible
ways. That way, each word $\psi_{i}(x,y,w,p,a)$ in the system $\Psi(x,y,w,p,a)$,
induces a set of elements $\lambda'{}_{i}^{2}(s,z,h_{2},g_{1},h_{1}',y,w,p,a)$
in the group $GFCl(x,z,h_{1},y,w,p,a)$. We think on the induced set
$\lambda'{}_{i}^{2}(s,z,h_{2},g_{1},h_{1}',y,w,p,a)$ as a system
of equations imposed on the variables $(s,z,h_{2},g_{1},h_{1}',y,w,p,a)$.
By Guba's theorem, for all $i$, the system $\lambda'{}_{i}^{2}(s,z,h_{2},g_{1},h_{1}',y,w,p,a)$
is equivalent to a finite system $\lambda_{i}^{2}(s,z,h_{2},g_{1},h_{1}',y,w,p,a)$
over the free group $F_{k}$.

Let $\mathfrak{J}^{3}$ be the collection of all the specializations
$(s_{0},z_{0},h_{2,0},g_{1,0},h_{1,0}',y_{0},w_{0},p_{0},a)$ of $GFCl$
that restrict to exceptional specializations of $Term(GFCl)$, and
satisfy one of the systems $\lambda_{i}^{2}(s,z,h_{2},g_{1},h_{1}',y,w,p,a)$,
for some $i$. The specializations in the collection $\mathfrak{J}^{3}$,
factor through finitely many maximal limit groups $\lambda^{2}WPGL_{1}(s,z,h_{2},g_{1},h_{1}',y,w,p,a),...,\lambda^{2}WPGL_{d_{2}}(s,z,h_{2},g_{1},h_{1}',y,w,p,a)$. 
\begin{lem}
\label{lem:54} Let $\Gamma$ be a group in the model. Then, $\lambda_{i}^{2}(s,z,h_{2},g_{1},h_{1}',y,w,p,a)$
is equivalent to \\
 $\lambda'{}_{i}^{2}(s,z,h_{2},g_{1},h_{1}',y,w,p,a)$ over $\Gamma$.
Moreover, if $(s_{0},z_{0},h_{2,0},g_{1,0},h_{1,0}',y_{0},w_{0},p_{0},a)$
is a $\Gamma$-specialization of $GFCl$, so that $(s_{0},z_{0},h_{2,0},g_{1,0},h_{1,0}',y_{0},w_{0},p_{0},a)$
restricts to a $\Gamma$-exceptional specialization of $Term(GFCl)$,
and solves the system $\lambda_{i}^{2}(s,z,h_{2},g_{1},h_{1}',y,w,p,a)$,
for some $i$, then $(s_{0},z_{0},h_{2,0},g_{1,0},h_{1,0}',y_{0},w_{0},p_{0},a)$
factors through one of the groups $\lambda^{2}WPGL_{1}(s,z,h_{2},g_{1},h_{1}',y,w,p,a),...,\lambda^{2}WPGL_{d_{2}}(s,z,h_{2},g_{1},h_{1}',y,w,p,a)$. 
\end{lem}

\begin{proof}
Identical to \lemref{51}. 
\end{proof}
\begin{lem}
\label{lem:55}$ $ 
\begin{enumerate}
\item Let $(h_{2,0},g_{1,0},h_{1,0}',w_{0},p_{0},a)$ be an exceptional
specialization of $WPHGH(h_{2},g_{1},h_{1}',w,p,a)$, and assume that
it factors through one of the groups 
\[
\lambda^{2}WPGL_{1}(s,z,h_{2},g_{1},h_{1}',y,w,p,a),...,\lambda^{2}WPGL_{d_{2}}(s,z,h_{2},g_{1},h_{1}',y,w,p,a)\,.
\]
Then, every specialization of $WPHGH(h_{2},g_{1},h_{1}',w,p,a)$ that
belongs to the exceptional family represented by $(h_{2,0},g_{1,0},h_{1,0}',w_{0},p_{0},a)$,
factors through one of the groups 
\[
\lambda^{2}WPGL_{1}(s,z,h_{2},g_{1},h_{1}',y,w,p,a),...,\lambda^{2}WPGL_{d_{2}}(s,z,h_{2},g_{1},h_{1}',y,w,p,a)\,.
\]
\item Let $\Gamma$ be a group in the model. Let $(h_{2,0},g_{1,0},h_{1,0}',w_{0},p_{0},a)$
be a $\Gamma$-exceptional specialization of \\
 $WPHGH(h_{2},g_{1},h_{1}',w,p,a)$, and assume that it factors through
one of the groups 
\[
\lambda^{2}WPGL_{1}(s,z,h_{2},g_{1},h_{1}',y,w,p,a),...,\lambda^{2}WPGL_{d_{2}}(s,z,h_{2},g_{1},h_{1}',y,w,p,a)\,.
\]
Then, every $\Gamma$-specialization of $WPHGH(h_{2},g_{1},h_{1}',w,p,a)$
that belongs to the $\Gamma$-exceptional family represented by $(h_{2,0},g_{1,0},h_{1,0}',w_{0},p_{0},a)$,
factors through one of the groups 
\[
\lambda^{2}WPGL_{1}(s,z,h_{2},g_{1},h_{1}',y,w,p,a),...,\lambda^{2}WPGL_{d_{2}}(s,z,h_{2},g_{1},h_{1}',y,w,p,a)\,.
\]
\end{enumerate}
\end{lem}

\begin{proof}
Similar to \lemref{52}. 
\end{proof}
\begin{lem}
\label{lem:56} Let $GFCl_{1},...,GFCl_{r}$ be a collection of terminal
graded formal closures in the graded formal MR diagram $GFMRD_{2}$
associated with the resolution $\lambda^{1}WPGRes(h_{1},y,w,p,a)$. 
\begin{enumerate}
\item Let $(g_{1,0},h_{1,0}',w_{0},p_{0},a)$ be an exceptional specialization
of the $wp$-graded base group $B$ of $Comp(\lambda^{1}WPGRes)$,
and for all $l=1,...,r$, let $(h_{2,0}^{l},g_{1,0}^{l},h_{1,0}'{}^{l},w_{0},p_{0},a)$
be an exceptional specialization of the $wp$-graded base group $Term(GFCl_{l})$
of $GFCl_{l}$, so that the restriction $(g_{1,0}^{l},h_{1,0}'{}^{l},w_{0},p_{0},a)$
belongs to the exceptional family of $B$ represented by $(g_{1,0},h_{1,0}',w_{0},p_{0},a)$.

Assume that the collection of the induced ungraded closures 
\[
GFCl_{l,0}=GFCl_{l,(h_{2,0}^{l},g_{1,0}^{l},h_{1,0}'{}^{l},w_{0},p_{0},a)}\,,
\]
$l=1,...,r$, form a covering closure for the induced ungraded resolution
\[
Comp_{0}^{2}=Comp(\lambda^{1}WPGRes)_{(g_{1,0},h_{1,0}',w_{0},p_{0},a)}\,.
\]
Assume further that for all $l=1,...,r$, the specialization $(h_{2,0}^{l},g_{1,0}^{l},h_{1,0}'{}^{l},w_{0},p_{0},a)$
cannot be extended to a specialization of any of the groups $\lambda^{2}WPGL_{j}(s,z,h_{2},g_{1},h_{1}',y,w,p,a)$
associated with $GFCl_{l}$.

Then, for every $y_{0}\in F_{k}$ that factors through the ungraded
resolution $Comp_{0}^{2}$, there exists some $x_{0}\in F_{k}$, so
that 
\[
\Sigma(x_{0},y_{0},w_{0},p_{0},a)=1\wedge\Psi(x_{0},y_{0},w_{0},p_{0},a)\neq1\,.
\]

\item Let $\Gamma$ be a group in the model.

Let $(g_{1,0},h_{1,0}',w_{0},p_{0},a)$ be a $\Gamma$-exceptional
specialization of the $wp$-graded base group $B$ of $Comp(\lambda WPGRes)$,
and for all $l=1,...,r$, let $(h_{2,0}^{l},g_{1,0}^{l},h_{1,0}'{}^{l},w_{0},p_{0},a)$
be a $\Gamma$-exceptional specialization of the $wp$-graded base
group $Term(GFCl_{l})$ of $GFCl_{l}$, so that the restriction $(g_{1,0}^{l},h_{1,0}'{}^{l},w_{0},p_{0},a)$
belongs to the $\Gamma$-exceptional family of $B$ represented by
$(g_{1,0},h_{1,0}',w_{0},p_{0},a)$.

Assume that the collection of the induced ungraded closures 
\[
GFCl_{l,0}=GFCl_{l,(h_{2,0}^{l},g_{1,0}^{l},h_{1,0}'{}^{l},w_{0},p_{0},a)}\,,
\]
$l=1,...,r$, form a covering closure for the induced ungraded resolution
\[
Comp_{0}^{2}=Comp(\lambda^{1}WPGRes)_{(g_{1,0},h_{1,0}',w_{0},p_{0},a)}\,.
\]
Assume further that for all $l=1,...,r$, the $\Gamma$-specialization
$(h_{2,0}^{l},g_{1,0}^{l},h_{1,0}'{}^{l},w_{0},p_{0},a)$ cannot be
extended to a $\Gamma$-specialization of any of the groups $\lambda^{2}WPGL_{j}(s,z,h_{2},g_{1},h_{1}',y,w,p,a)$
associated with $GFCl_{l}$.

Then, for every $y_{0}\in\Gamma$ that factors through the ungraded
resolution $Comp_{0}^{2}$, there exists some $x_{0}\in\Gamma$, so
that 
\[
\Sigma(x_{0},y_{0},w_{0},p_{0},a)=1\wedge\Psi(x_{0},y_{0},w_{0},p_{0},a)\neq1\,.
\]

\end{enumerate}
\end{lem}

\begin{proof}
For the first statement, let $y_{0}\in F_{k}$ so that it factors
through the ungraded resolution $Comp_{0}^{2}$. Since the collection
$GFCl_{1,0},...,GFCl_{r,0}$ is a covering closure for $Comp_{0}^{2}$,
$y_{0}$ must factor through $GFCl_{l,0}$, for some $l=1,...,r$.
Since $(h_{2,0}^{l},g_{1,0}^{l},h_{1,0}'{}^{l},w_{0},p_{0},a)$ does
not factor through any of the groups $\lambda^{2}WPGL_{j}(s,z,h_{2},g_{1},h_{1}',y,w,p,a)$
associated with $GFCl_{l}$, and using \lemref{28}, we deduce the
desired.

For the second statement, let $y_{0}\in\Gamma$ so that it factors
through the ungraded resolution $Comp_{0}^{2}$. Since the collection
$GFCl_{1,0},...,GFCl_{r,0}$ is a covering closure for $Comp_{0}^{2}$,
$y_{0}$ must factor through $GFCl_{l,0}$, for some $l=1,...,r$.
Since $(h_{2,0}^{l},g_{1,0}^{l},h_{1,0}'{}^{l},w_{0},p_{0},a)$ does
not factor through any of the groups $\lambda^{2}WPGL_{j}(s,z,h_{2},g_{1},h_{1}',y,w,p,a)$
associated with $GFCl_{l}$, then according to \lemref{54}, and using
\lemref{29}, we deduce the desired. 
\end{proof}
We assume that the first level limit groups $\lambda^{1}WPGL_{1}(h_{1},y,w,p,a),...,\lambda^{1}WPGL_{d}(h_{1},y,w,p,a)$
of the $wp$-graded MR diagram $GRemMRD_{2}$ are of \emph{minimal
rank}. That is, we assume that for all $j=1,...,d$, if the group
$\lambda^{1}WPGL_{j}(h_{1},y,w,p,a)$ admits a restricted epimorphism
onto $F\ast F_{k}$, for some free group $F$, so that the subgroup
$\langle w,p,a\rangle$ is mapped entirely into $F_{k}$, then $F$
is the trivial group.

From now until the end of this section, we continue under this assumption.

This assumption, allows us to apply \thmref{48} in order to construct
from the collection $\mathfrak{J}^{3}$ of all the specializations
$(s_{0},z_{0},h_{2,0},g_{1,0},h_{1,0}',y_{0},w_{0},p_{0},a)$ of $GFCl$
that restrict to exceptional specializations of $Term(GFCl)=WPHGH(h_{2},g_{1},h_{1}',w,p,a)$,
and satisfy one of the systems $\lambda_{i}^{2}(s,z,h_{2},g_{1},h_{1}',y,w,p,a)$,
for some $i$, to construct the quotient resolutions $Comp_{1}^{3}(t,s,z,h_{2},g_{1},h_{1}',y,w,p,a),...,Comp_{s}^{3}(t,s,z,h_{2},g_{1},h_{1}',y,w,p,a)$
of the systems $\lambda_{i}^{2}(s,z,h_{2},g_{1},h_{1},y,w,p,a)$ w.r.t.
the $wp$-graded closure $GFCl$. According to the construction of
the quotient resolutions in this case, $Comp_{1}^{3},...,Comp_{s}^{3}$
are quotients of the graded formal closure $GFCl$. And in the case
that a quotient resolution $Comp_{j}^{3}$ is of maximal complexity,
i.e., of complexity that equals the complexity of $Comp(\lambda^{1}WPGRes)$,
then, the completed resolution $Comp_{j}^{3}$ can differ from $GFCl$
only in its base group, i.e., $Comp_{j}^{3}$ is obtained from $GFCl$
by adding some relations to $Term(GFCl)$.

The collection of quotient resolutions associated with the graded
formal closure $GFCl$ that are not of maximal complexity, is denoted
by $GRem_{3}$.
\begin{lem}
\label{lem:57}$ $ 
\begin{enumerate}
\item Let $(g_{1,0},h_{1,0}',w_{0},p_{0},a)$ be an exceptional specialization
of the $wp$-rigid or $wp$-solid base group $B$ of $Comp(\lambda^{1}WPGRes)$,
and assume that $(w_{0},p_{0})\in AE(w,p)$. Then, there exists a
collection of terminal graded formal closures $GFCl_{1},...,GFCl_{r}$
in the graded formal MR diagram $GFMRD_{2}$ associated with the resolution
$\lambda^{1}WPGRes(h_{1},y,w,p,a)$, together with exceptional specializations
$(h_{2,0}^{l},g_{1,0}^{l},h_{1,0}'{}^{l},w_{0},p_{0},a)$, $l=1,...,r$,
of their corresponding $wp$-rigid or $wp$-solid base groups $Term(GFCl_{l})$,
so that: 
\begin{enumerate}
\item The restriction $(g_{1,0}^{l},h_{1,0}'{}^{l},w_{0},p_{0},a)$ belongs
to the exceptional family of $B$ represented by \\
 $(g_{1,0},h_{1,0}',w_{0},p_{0},a)$, for all $l=1,...,r$. 
\item The induced ungraded closures $GFCl_{1,0},...,GFCl_{r,0}$ form a
covering closure for the induced ungraded completion $Comp_{0}^{2}$. 
\item For every maximal complexity quotient resolution $Comp_{j}^{3}$ that
is associated with $GFCl_{l}$, there is no specialization of the
$wp$-rigid or $wp$-solid base group $Term(Comp_{j}^{3})$ of $Comp_{j}^{3}$,
that restricts to the exceptional family of $Term(GFCl_{l})$ represented
by $(h_{2,0}^{l},g_{1,0}^{l},h_{1,0}'{}^{l},w_{0},p_{0},a)$. 
\end{enumerate}
\item Let $\Gamma$ be a group in the model. \\
 Let $(g_{1,0},h_{1,0}',w_{0},p_{0},a)$ be a $\Gamma$-exceptional
specialization of the $wp$-rigid or $wp$-solid base group $B$ of
$Comp(\lambda WPGRes)$, and assume that $(w_{0},p_{0})\in AE_{\Gamma}(w,p)$.
Then, there exists a collection of terminal graded formal closures
$GFCl_{1},...,GFCl_{r}$ in the graded formal MR diagram $GFMRD_{2}$
associated with the resolution $\lambda WPGRes(h_{1},y,w,p,a)$, together
with $\Gamma$-exceptional specializations $(h_{2,0}^{l},g_{1,0}^{l},h_{1,0}'{}^{l},w_{0},p_{0},a)$,
$l=1,...,r$, of their corresponding $wp$-rigid or $wp$-solid base
groups $Term(GFCl_{l})$, so that: 
\begin{enumerate}
\item The restriction $(g_{1,0}^{l},h_{1,0}'{}^{l},w_{0},p_{0},a)$ belongs
to the $\Gamma$-exceptional family of $B$ represented by \\
 $(g_{1,0},h_{1,0}',w_{0},p_{0},a)$, for all $l=1,...,r$. 
\item The induced ungraded closures $GFCl_{1,0},...,GFCl_{r,0}$ form a
covering closure for the induced ungraded completion $Comp_{0}^{2}$. 
\item For every maximal complexity quotient resolution $Comp_{j}^{3}$ that
is associated with $GFCl_{l}$, there is no $\Gamma$-specialization
of the $wp$-rigid or $wp$-solid base group $Term(Comp_{j}^{3})$
of $Comp_{j}^{3}$, that restricts to the $\Gamma$-exceptional family
of $Term(GFCl_{l})$ represented by $(h_{2,0}^{l},g_{1,0}^{l},h_{1,0}'{}^{l},w_{0},p_{0},a)$. 
\end{enumerate}
\end{enumerate}
\end{lem}

\begin{proof}
For the first statement, since $(w_{0},p_{0})\in AE(w,p)$, then,
according to \thmref{25}, there exists a collection of terminal graded
formal closures $GFCl_{1},...,GFCl_{r}$ in $GFMRD_{2}$, together
with exceptional specializations $(h_{2,0}^{l},g_{1,0}^{l},h_{1,0}'{}^{l},w_{0},p_{0},a)$,
$l=1,...,r$, of their corresponding $wp$-graded base groups $Term(GFCl_{l})$,
so that the conditions (a) and (b) are satisfied, and the elements
of the system $\lambda_{i}^{2}(s,z,h_{2},g_{1},h_{1}',y,w,p,a)$ are
not all trivial in the induced ungraded formal closure $GFCl_{l,0}$,
for all $i$ and all $l=1,...,r$. Assume by contradiction that for
some maximal complexity quotient resolution $Comp_{j}^{3}$, the base
group $Term(Comp_{j}^{3})$ admits a specialization that belongs to
the modular family of $Term(GFCl_{l})$ represented by $(h_{2,0}^{l},g_{1,0}^{l},h_{1,0}'{}^{l},w_{0},p_{0},a)$,
and let $Comp_{j,0}^{3}$ be the induced ungraded resolution. Hence,
since $Comp_{j}^{3}$ is of maximal complexity, and according to \lemref{55},
we deduce that the elements of the system $\lambda_{i}^{2}(s,z,h_{2},g_{1},h_{1}',y,w,p,a)$
are not all trivial in $Comp_{j,0}^{3}$, for all $i$. This contradicts
the construction of $Comp_{j}^{3}$ being a quotient resolution for
the systems $\lambda_{i}^{2}(s,z,h_{2},g_{1},h_{1}',y,w,p,a)$ w.r.t.
the $wp$-graded closure $GFCl_{l}$.

For the second statement, since $(w_{0},p_{0})\in AE_{\Gamma}(w,p)$,
then, according to \thmref{27}, there exists a collection of terminal
graded formal closures $GFCl_{1},...,GFCl_{r}$ in $GFMRD_{2}$, together
with $\Gamma$-exceptional specializations $(h_{2,0}^{l},g_{1,0}^{l},h_{1,0}'{}^{l},w_{0},p_{0},a)$,
$l=1,...,r$, of their corresponding $wp$-graded base groups $Term(GFCl_{l})$,
so that the conditions (a) and (b) are satisfied, and the elements
of the system $\lambda_{i}^{2}(s,z,h_{2},g_{1},h_{1}',y,w,p,a)$ are
not all trivial in the induced ungraded formal closure $GFCl_{l,0}$,
for all all $i$ and all $l=1,...,r$. Assume by contradiction that
for some maximal complexity quotient resolution $Comp_{j}^{2}$, the
base group $Term(Comp_{j}^{3})$ admits a $\Gamma$-specialization
that belongs to the modular family of $Term(GFCl_{l})$ represented
by $(h_{2,0}^{l},g_{1,0}^{l},h_{1,0}'{}^{l},w_{0},p_{0},a)$, and
let $Comp_{j,0}^{3}$ be the induced ungraded resolution. Hence, since
$Comp_{j}^{3}$ is of maximal complexity, and according to \lemref{55},
we deduce that the elements of the system $\lambda_{i}^{2}(s,z,h_{2},g_{1},h_{1}',y,w,p,a)$
are not all trivial in $Comp_{j,0}^{3}$, for all $i$. This contradicts
the construction of $Comp_{j}^{3}$ being a quotient resolution for
the systems $\lambda_{i}^{2}(s,z,h_{2},g_{1},h_{1}',y,w,p,a)$ w.r.t.
the $wp$-graded closure $GFCl_{l}$. 
\end{proof}
\begin{defn}
\label{def:58} A \emph{proof system of depth 2} is to make the following
ordered choice: 
\begin{enumerate}
\item To choose one of the terminal graded formal closures $WPH(h_{1},w,p,a)\ast F_{y}$
in $GFMRD_{1}$. 
\item For every resolution $\lambda^{1}WPGRes(h_{1},y,w,p,a)$ in $GRemMRD_{2}$,
to choose an integer $0\leq n\leq M(WPHG)$, where $WPHG(g_{1},h_{1}',w,p,a)$
is the $wp$-rigid or $wp$-solid base group of the completed resolution
$Comp^{2}=Comp(\lambda^{1}WPGRes)(z,h_{1},y,w,p,a)$, and $M(WPHG)$
is the maximal possible number of exceptional families of specializations
of $WPHG$ for a fixed value of the parameters $w_{0},p_{0}\in F_{k}$. 
\item For every resolution $\lambda^{1}WPGRes(h_{1},y,w,p,a)$ in $GRemMRD_{2}$
for which the integer $n$ was chosen, and for every $1\leq i\leq n$,
to choose finitely many terminal graded formal closures $GFCl_{1}^{i},...,GFCl_{m_{i}}^{i}$
in the graded formal MR diagram $GFMRD_{2}$ associated with the resolution
$\lambda^{1}WPGRes(h_{1},y,w,p,a)$, where $1\leq m_{i}\leq\Sigma\left\{ M(Term(GFCl)):\,GFCl\text{ is a terminal closure in }GFMRD_{2}\right\} $,
and $M(Term(GFCl))$ is the maximal possible number of exceptional
families of specializations of the $wp$-rigid or $wp$-solid base
group $Term(GFCl)$ for a fixed value of the parameters $w_{0},p_{0}\in F_{k}$.
. 
\end{enumerate}
We say that a proof system $PS$ as above, \emph{verifies $(w_{0},p_{0})\in AE(w,p)$},
if for every graded formal closure $GFCl_{j}^{i}$ listed in the proof
system $PS$, there exists an exceptional specialization $(h_{2,0}^{i,j},g_{1,0}^{i,j},h_{1,0}'{}^{i,j},w_{0},p_{0},a)$
of $Term(GFCl_{j}^{i})$, so that 
\begin{enumerate}
\item For all $i,j$, the specializations $(h_{1,0}'{}^{i,j},w_{0},p_{0},a)$
represent the same exceptional family of the group $WPH(h_{1},w,p,a)$
that was specified by the proof system $PS$. 
\item For all $i$, all the specializations $(g_{1,0}^{i,j},h_{1,0}'{}^{i,j},w_{0},p_{0},a)$
represent the same exceptional family of the group $WPHG(g_{1},h_{1}',w,p,a)$
for which the graded formal closures $GFCl_{j}^{i}$ were associated
by the proof system $PS$. 
\item For all $i$, the induced ungraded closures $GFCl_{1,0}^{i},...,GFCl_{m_{i},0}^{i}$
form a covering closure for the induced ungraded resolution $Comp_{0}^{2}$,
where $Comp^{2}=Comp(\lambda^{1}WPGRes)(z,h_{1},y,w,p,a)$ is the
completion of the resolution of $GRemMRD_{2}$ for which the graded
formal closures $GFCl_{j}^{i}$ were associated by the proof system
$PS$. 
\item For every resolution $\lambda^{1}WPGRes(h_{1},y,w,p,a)$ in $GRemMRD_{2}$
for which the integer $n$ was chosen by the proof system $PS$, and
for every $1\leq i_{1}\neq i_{2}\leq n$, the exceptional specializations
$(g_{1,0}^{i_{1},j},h_{1,0}'{}^{i_{1},j},w_{0},p_{0},a)$ and $(g_{1,0}^{i_{2},j},h_{1,0}'{}^{i_{2},j},w_{0},p_{0},a)$
represent distinct exceptional families of the terminal $wp$-graded
group \\
 $WPHG(g_{1},h_{1}',w,p,a)$ of $\lambda^{1}WPGRes$. 
\item For every resolution $\lambda^{1}WPGRes(h_{1},y,w,p,a)$ in $GRemMRD_{2}$
for which the integer $n$ was chosen by the proof system $PS$, the
terminal group $WPHG(g_{1},h_{1}',w,p,a)$ of $\lambda^{1}WPGRes$,
admits exactly $n$ distinct exceptional families corresponding to
the value $w_{0}p_{0}$ of the parameters. 
\item For all $i,j$, the specialization $(h_{2,0}^{i,j},g_{1,0}^{i,j},h_{1,0}'{}^{i,j},w_{0},p_{0},a)$
does not factor through any of the maximal limit groups $\lambda^{2}WPGL_{1}(s,z,h_{2},g_{1},h_{1}',y,w,p,a),...,\lambda^{2}WPGL_{d_{2}}(s,z,h_{2},g_{1},h_{1}',y,w,p,a)$
associated with the graded formal closure $GFCl_{j}^{i}$. 
\end{enumerate}
Such a tuple 
\[
\left(\left(h_{2,0}^{i,j},g_{1,0}^{i,j},h_{1,0}'{}^{i,j}\right)_{i,j},w_{0},p_{0},a\right)
\]
that satisfies the conditions listed above, is called a \emph{validPS
statement} (of depth 2). Given a group $\Gamma$ in the model, a \emph{$\Gamma$-validPS
statement} (of depth 2) is defined similarly.
\end{defn}

\subsection{The Construction of Proof Systems of General Depth under Minimal
Rank Assumption}

The construction that we explained so far, suffices for encoding ``proofs''
for that $(w_{0},p_{0})\in AE(w,p)$ for a subset of the parameters
in $AE(w,p)$ (those that admit a validPS statement of depth 2 or
of depth 1 - see \defref{58} and \lemref{53} correspondingly).
In order to encode ``proofs'' for all the parameters in $AE(w,p)$,
we continue the construction with each of the quotient resolutions
\[
Comp_{1}^{3}(t,s,z,h_{2},g_{1},h_{1}',y,w,p,a),...,Comp_{s}^{3}(t,s,z,h_{2},g_{1},h_{1}',y,w,p,a)\,.
\]
According to \lemref{57}, we can continue only with those quotient
resolutions that are not of maximal complexity. We summarize the construction
so far, and explain how to continue it iteratively.

Let $Comp^{2}=Comp(\lambda^{1}WPGRes)(z,h_{1},y,w,p,a)$ be the completion
of the $wp$-graded resolution \\
 $\lambda^{1}WPGRes(h_{1},y,w,p,a)$ in the diagram $GRemMRD_{2}$,
who terminates in the $wp$-rigid or $wp$-solid group $WPHG(g_{1},h_{1}',w,p,a)$.
Let $GFMRD_{2}$ be the $wp$-graded formal MR diagram of the system
$\Sigma(x,y,w,p,a)$ over the resolution $Comp^{2}$. Let $GFRes^{2}(x,z,h_{1},y,w,p,a)$
be a graded formal resolution in $GFMRD_{2}$, who terminates in a
graded formal closure $GFCl(x,z,h_{1},y,w,p,a)$ of $Comp^{2}$ whose
$wp$-rigid or $wp$-solid base group is $Term(GFCl)=WPHGH(h_{2},g_{1},h_{1}',w,p,a)$.
Let $\lambda_{i}^{2}(s,z,h_{2},g_{1},h_{1}',y,w,p,a)$ be the finitely
many systems associated with $GFCl$ using the resolution $GFRes^{2}$
and the elements $\psi_{i}$ of the inequalities $\Psi$. Let $\mathfrak{J}^{3}$
be the collection of all the specializations $(s_{0},z_{0},h_{2,0},g_{1,0},h_{1,0}',y_{0},w_{0},p_{0},a)$
of $GFCl$ that restrict to exceptional specializations of $Term(GFCl)$,
and satisfy one of the systems $\lambda_{i}^{2}(s,z,h_{2},g_{1},h_{1}',y,w,p,a)$,
for some $i$. The specializations in the collection $\mathfrak{J}^{3}$,
factors through finitely many maximal limit groups $\lambda^{2}WPGL_{j}(s,z,h_{2},g_{1},h_{1}',y,w,p,a)$.

Let $Comp_{l}^{3}(t,s,z,h_{2},g_{1},h_{1}',y,w,p,a)$ be the finitely
many quotient resolutions of the systems \\
 $\lambda_{i}^{2}(s,z,h_{2},g_{1},h_{1},y,w,p,a)$ w.r.t. the $wp$-graded
closure $GFCl$, that are not of maximal complexity, i.e., the complexity
of each $Comp_{l}^{3}$ is strictly smaller than the complexity of
$Comp^{2}$. The collection of the quotient resolutions $Comp_{l}^{3}$
is denoted by $GRem_{3}$.

And we continue the construction iteratively with each quotient resolution
$Comp^{3}$ in $GRem_{3}$.

In light of the strict reduction in the complexities, the procedure
must terminate after finitely many steps. 
\begin{defn}
\label{def:59} Let $d\geq2$ be an integer. A (general) \emph{proof
system} is to make the following ordered choice: 
\begin{enumerate}
\item To choose one of the terminal graded formal closures $WPH(h_{1},w,p,a)\ast F_{y}$
in $GFMRD_{1}$. 
\item For every resolution $\lambda^{1}WPGRes(h_{1},y,w,p,a)$ in $GRemMRD_{2}$,
to choose an integer $0\leq n\leq M(WPHG)$, where $WPHG(g_{1},h_{1}',w,p,a)$
is the $wp$-rigid or $wp$-solid base group of the completed resolution
$Comp^{2}=Comp(\lambda^{1}WPGRes)(z,h_{1},y,w,p,a)$, and $M(WPHG)$
is the maximal possible number of exceptional families of specializations
of $WPHG$ for a fixed value of the parameters $w_{0},p_{0}\in F_{k}$. 
\item For every resolution $\lambda^{1}WPGRes(h_{1},y,w,p,a)$ in $GRemMRD_{2}$
for which the integer $n$ was chosen, and for every $1\leq i\leq n$,
to choose finitely many terminal graded formal closures $GFCl_{1}^{i},...,GFCl_{m_{i}}^{i}$
in the graded formal MR diagram $GFMRD_{2}$ associated with the resolution
$\lambda^{1}WPGRes(h_{1},y,w,p,a)$, where $1\leq m_{i}\leq\Sigma\left\{ M(Term(GFCl)):\,GFCl\text{ is a terminal closure in }GFMRD_{2}\right\} $,
and $M(Term(GFCl))$ is the maximal possible number of exceptional
families of specializations of the $wp$-rigid or $wp$-solid base
group $Term(GFCl)$ for a fixed value of the parameters $w_{0},p_{0}\in F_{k}$. 
\item And by induction on $l\geq2$, if for the resolution $Comp^{l}$ in
$GRem_{l}$ who terminates in $WP(HG)^{l-1}$ the integer $n$ was
chosen, and if for some $1\leq i\leq n$ the graded formal closure
$GFCl$ in the diagram $GFMRD_{l}$ associated with $Comp^{l}$ was
chosen, then, for every resolution $Comp^{l+1}$ in $GRem_{l+1}$
associated with $GFCl$ who terminates in $WP(HG)^{l}$, to choose
an integer $0\leq m\leq M(WP(HG)^{l})$, and for every $1\leq j\leq m$,
to choose finitely many terminal graded formal closures $GFCl_{1},...,GFCl_{s}$
in the diagram $GFMRD_{l+1}$ associated with $Comp^{l+1}$, where
\[
1\leq s\leq\Sigma\left\{ M(Term(GFCl')):\,GFCl'\text{ is a terminal closure in }GFMRD_{l+1}\right\} \,.
\]
\end{enumerate}
We say that a proof system $PS$ as above, is \emph{of depth $d$},
if $d$ is the maximal integer with the property that some graded
formal closure $GFCl$ listed in the proof system $PS$, is associated
with a resolution $Comp^{d}$ in $GRem_{d}$.

We say that a proof system $PS$ of depth $d$, \emph{verifies $(w_{0},p_{0})\in AE(w,p)$},
if for every graded formal closure $GFCl$ listed in the proof system
$PS$ and is deepest w.r.t. the proof system $PS$, there exists an
exceptional specialization $f_{0}$ of $Term(GFCl)$, so that: 
\begin{enumerate}
\item The restrictions of all the specializations $f_{0}$ to the group
$WPH(h_{1},w,p,a)$ that was specified by the proof system $PS$,
represent the same exceptional family of it. 
\item For every resolution $Comp^{l}$ that was listed in $PS$ and for
all $1\leq i\leq n$, where $n$ is the chosen integer for $Comp^{l}$,
all the induced specializations of the terminal group $WP(HG)^{l-1}$
of $Comp^{l}$ represent the same exceptional family of $WP(HG)^{l-1}$. 
\item For every resolution $Comp^{l}$ that was listed in $PS$ and for
all $1\leq i\leq n$, where $n$ is the chosen integer for $Comp^{l}$,
if the chosen collection of terminal graded formal closures in the
diagram $GFMRD_{l}$ associated with $Comp^{l}$, is $GFCl_{1},...,GFCl_{s}$,
then the induced ungraded closures $GFCl_{1,0},...,GFCl_{s,0}$ form
a covering closure for the induced ungraded resolution $Comp_{0}^{l}$. 
\item For every resolution $Comp^{l}$ that was listed in $PS$ and for
all $1\leq i_{1}\neq i_{2}\leq n$, where $n$ is the chosen integer
for $Comp^{l}$, the induced two exceptional families of the terminal
group $WP(HG)^{l-1}$ of $Comp^{l}$ from the directions of $i_{1}$
and $i_{2}$ in the proof system $PS$, are distinct exceptional families
of $WP(HG)^{l-1}$. 
\item For every resolution $Comp^{l}$ that was listed in $PS$ and for
which the integer $n$ was chosen, the terminal group $WP(HG)^{l-1}$
of $Comp^{l}$ admits exactly $n$ distinct exceptional families corresponding
to the value $w_{0}p_{0}$ of the parameters. 
\item Each of the specializations $f_{0}$ does not factor through any of
the maximal limit groups $\lambda^{t}WPGL_{j}$ associated with the
graded formal closure $GFCl$. 
\end{enumerate}
Such a tuple $\left((f_{0}),w_{0},p_{0},a\right)$, denoted for brevity
$(f_{0},w_{0},p_{0},a)$, that satisfies the conditions listed above,
is called a \emph{validPS statement} (of depth $d$). Given a group
$\Gamma$ in the model, a \emph{$\Gamma$-validPS statement} (of depth
$d$) is defined similarly.
\end{defn}

By the construction of proof systems that we have explained in this
section, there exist only finitely many proof systems. According to
the construction that we have explained in this section, for every
$p_{0}\in EAE(p)$, there exists a witness $w_{0}\in F_{k}$ and a
proof system $PS$, so that the proof system $PS$ verifies that $(w_{0},p_{0})\in AE(w,p)$.
Similarly, for every group $\Gamma$ in the model, and for every $p_{0}\in EAE_{\Gamma}(p)$,
there exist a witness $w_{0}\in\Gamma$ and a proof system $PS$,
so that $PS$ verifies that $(w_{0},p_{0})\in AE(w,p)$.

We fix a proof system $PS$. We extend each validPS statement $\left((f_{0}),w_{0},p_{0},a\right)$
(and similarly each $\Gamma$-validPS statement) in the following
two ways: 
\begin{enumerate}
\item Let $Comp$ be one of the quotient resolutions listed in the proof
system $PS$. Denote by $WPHG$ the $wp$-rigid or $wp$-solid base
group of $Comp$. Assume that $n$ is the integer chosen for the resolution
$Comp$, and let $1\leq i\leq n$. Let $GFCl_{1},...,GFCl_{s}$ be
the collection of terminal graded formal closures in the diagram $GFMRD_{l}$
(associated with $Comp$) that was chosen for $Comp$ and $i$. For
every non-cyclic maximal (pegged) abelian subgroup $A$ that appears
in one of the decompositions along the various levels of $Comp$,
and for each $j=1,...s$, one can associate a free abelian group $\mathbb{Z}^{m_{A}}$,
a subgroup of finite index $C_{A}^{j}\leq\mathbb{Z}^{m_{A}}$ and
a parameterized coset of $C_{A}^{j}$ (with integer parameter). We
denote by $lcm(A)$ the least common multiple of the indices of the
subgroups $C_{A}^{1},...,C_{A}^{s}$ in $\mathbb{Z}^{m_{A}}$.

The value of a representative for the associated coset depends only
on the (modular family of a) specialization of the base group $Term(GFCl_{j})$.
In other words, once we have fixed a specialization $f_{0}^{j}$ for
$Term(GFCl_{j})$, the coset $C_{A,0}^{j}$ is determined. Actually,
if the abelian group $A$ does not appear in the terminal level of
$Comp$ (i.e., $A$ does not appear in the $wp$-graded JSJ of $WPHG$),
then the coset $C_{A}^{j}$ depends only on the graded closure $GFCl_{j}$
itself, and does not depend on the specific specialization $f_{0}^{j}$
(i.e., $C_{A}^{j}=C_{A,0}^{j}$). Moreover, the entries of the elements
in the coset $C_{A}^{j}=C_{A,0}^{j}$ correspond to the indices of
the cyclic subgroups generated by the images of a fixed sub-basis
of the abelian group $A$, in the maximal cyclic subgroup generated
by the peg $peg_{A}$ of $A$ in the next level of the graded closure
$GFCl_{j}$, where $peg_{A}$ is a fixed generator for the cyclic
edge group connecting $A$ to the next level of $Comp$. However,
in the case that the abelian group $A$ appears in the terminal level
of $Comp$, the entries of the elements in the coset $C_{A,0}^{j}$
correspond to the indices of the cyclic subgroups generated by the
images of a fixed sub-basis of the abelian group $A$, in the maximal
cyclic subgroup generated by the maximal root of the image of $peg_{A}$
in the next level of the induced ungraded closure $GFCl_{j,0}$, under
the specializations of $Comp$ that factor through the induced ungraded
closure $GFCl_{j,0}$.

If further $f_{0}^{1},...,f_{0}^{s}$ are specializations of the base
groups $Term(GFCl_{1}),...,Term(GFCl_{s})$, and in addition the restrictions
of $f_{0}^{1},...,f_{0}^{s}$ to the base group $WPHG$ of $Comp$
represent the same exceptional family, then in order for the collection
of the induced ungraded closures $GFCl_{1,0},...,GFCl_{s,0}$ to form
a covering closure for the corresponding ungraded resolution $Comp_{0}$,
the union of the induced cosets $C_{A,0}^{1},...,C_{A,0}^{s}$ must
cover the entire group $\mathbb{Z}^{m_{A}}$ for every maximal abelian
group $A$ that appears along the decompositions in the resolution
$Comp$.

Hence, for each maximal abelian group $A$ that appears along the
decompositions in the resolution $Comp$, we add a new variable $r_{A}$
that will be mapped to the image of $peg_{A}$ when $A$ does not
appear in the terminal level of $Comp$, or to $Rpeg_{A,0}^{q}$ otherwise,
where $Rpeg_{A,0}$ is the maximal root of the image of $peg_{A}$
and $q$ is co-prime to each of the indices of the subgroups $C_{A}^{j}$
in $\mathbb{Z}^{m_{A}}$. In order for the collection of the induced
ungraded closures $GFCl_{1,0},...,GFCl_{s,0}$ to form a covering
closure for the corresponding ungraded resolution $Comp_{0}$, we
must be able to extend each of the specializations $(r_{A,0}^{l_{1}},...,r_{A,0}^{l_{m_{A}}})$
of the fixed sub-basis of the abelian group $A$, for all $l_{1},...,l_{m_{A}}=0,...,lcm(A)$,
to a specialization that factor through one of the ungraded closures
$GFCl_{1,0},...,GFCl_{s,0}$.

By the definition of a validPS statement, for every resolution $Comp^{l}$
that was listed in $PS$ and for all $1\leq i\leq n$, where $n$
is the chosen integer for $Comp^{l}$, if the chosen collection of
terminal graded formal closures in the diagram $GFMRD_{l}$ associated
with $Comp^{l}$ is $GFCl_{1},...,GFCl_{s}$, then the induced ungraded
closures $GFCl_{1,0},...,GFCl_{s,0}$ form a covering closure for
the induced ungraded resolution $Comp_{0}^{l}$. For each such $Comp^{l}$,
and each $1\leq i\leq n$, we add the corresponding specializations
$r_{A,0}$ for the new variables $r_{A}$. Moreover, we add specializations
for new variables $AbDem$. These specializations demonstrate that
each of the (finitely many) corresponding specializations $(r_{A,0}^{l_{1}},...,r_{A,0}^{l_{m_{A}}})$
of the fixed sub-bases of the maximal abelian groups $A$ that appear
in $Comp^{l}$, factors through one of the closures $GFCl_{1,0},...,GFCl_{s,0}$. 
\item We add specializations for new variables $SameFamily$ that demonstrates
that the restrictions of all the specializations $f_{0}$ in a given
validPS statement, to the group $WPH(h_{1},w,p,a)$ that was specified
by the proof system $PS$, represent the same modular family of it,
and demonstrates that for every resolution $Comp^{l}$ that was listed
in $PS$ and for all $1\leq i\leq n$, where $n$ is the chosen integer
for $Comp^{l}$, all the induced specializations of the terminal group
$WP(HG)^{l-1}$ of $Comp^{l}$ represent the same modular family of
$WP(HG)^{l-1}$. 
\end{enumerate}
We denote the extended validPS statements by the extended tuples $\left((SameFamily_{0}),(AbDem_{0}),(r_{0}),(f_{0}),w_{0},p_{0},a\right)$,
or briefly by $\left(SameFamily_{0},AbDem_{0},r_{0},f_{0},w_{0},p_{0},a\right)$.
We want to prove that the given $EAE(p)$ set, can be defined over
$F_{k}$ and over the random groups $\Gamma$ by a uniform formula
which belongs to the Boolean algebra of $AE$-formulas. For that,
we start by collecting all the validPS statements and $\Gamma$-validPS
statements in a finite uniform collection of limit groups.

We consider the collection $\mathfrak{J}$ of all the specializations
$\left(\left((SameFamily_{0}),(AbDem_{0}),(r_{0}),(f_{0}),w_{0},p_{0},a\right)\right)$
(over $F_{k}$) that satisfy all the conditions in the definition
of validPS statement, except maybe that for some resolution $Comp^{l}$
that was listed in $PS$ and for which the integer $n$ was chosen,
the terminal group $WP(HG)^{l-1}$ of $Comp^{l}$ admits more than
$n$ distinct exceptional families corresponding to the value $w_{0}p_{0}$
of the parameters. The collection $\mathfrak{J}$ factors through
finitely many maximal limit groups $PS\left(SameFamily,AbDem,r,f,w,p,a\right)$. 
\begin{lem}
$ $ 
\begin{enumerate}
\item Every validPS statement factors through the groups $PS\left(SameFamily,AbDem,r,f,w,p,a\right)$. 
\item Let $\Gamma$ be a group in the model. Every $\Gamma$-validPS statement
factors through the groups \\
 $PS\left(SameFamily,AbDem,r,f,w,p,a\right)$. 
\end{enumerate}
\end{lem}

\begin{proof}
The first claim follows directly from the definition of the groups
$PS\left(SameFamily,AbDem,r,f,w,p,a\right)$. For the second one,
let $\left((SameFamily_{0}),(AbDem_{0}),(r_{0}),(f_{0}),w_{0},p_{0},a\right)$
be a $\Gamma$-validPS statement. Let $\left((\widetilde{SameFamily_{0}}),(\widetilde{AbDem_{0}}),(\tilde{r}_{0}),(\tilde{f}_{0}),\tilde{w}_{0},\tilde{p}_{0},a\right)$
be a lift of $\left(SameFamily_{0},(AbDem_{0}),r_{0},f_{0},w_{0},p_{0},a\right)$
to a specialization in $F_{k}$ so that: 
\begin{enumerate}
\item each $\tilde{f}_{0}$ is a specialization for the corresponding graded
formal closure $GFCl$ listed in the proof system $PS$, 
\item for every resolution $Comp^{l}$ that was listed in $PS$ and for
all $1\leq i\leq n$, where $n$ is the chosen integer for $Comp^{l}$,
the specializations $\widetilde{AbDem_{0}}$ demonstrates that each
of the corresponding specializations $(\tilde{r}_{A,0}^{l_{1}},...,\tilde{r}_{A,0}^{l_{m_{A}}})$
of the fixed sub-bases of the maximal abelian groups that appear in
$Comp^{l}$, factor through one of the induced ungraded closures $GFCl_{1,0},...,GFCl_{s,0}$
(induced from the lifted specializations $(\tilde{f}_{0})$). 
\item The specializations $(\widetilde{SameFamily_{0}})$ demonstrates that
the restrictions of all the specializations $\tilde{f}_{0}$, to the
group $WPH(h_{1},w,p,a)$ that was specified by the proof system $PS$,
represent the same modular family of it, and demonstrates that for
every resolution $Comp^{l}$ that was listed in $PS$ and for all
$1\leq i\leq n$, where $n$ is the chosen integer for $Comp^{l}$,
all the induced specializations of the terminal group $WP(HG)^{l-1}$
of $Comp^{l}$ represent the same modular family of $WP(HG)^{l-1}$. 
\end{enumerate}
Note that if some specialization $\tilde{r}_{0}$ equals to $\tilde{r}_{0}=x_{0}^{q}$
for some non-trivial $x_{0}\in F_{k}$, then its projection $r_{0}$
to $\Gamma$ is also a $q$-th power ($\Gamma$ is torsion-free and
the corresponding ungraded resolutions are non-degenerate). Hence,
and since all the lifts of a $\Gamma$-exceptional specialization
of a given rigid or solid group to $F_{k}$ is also exceptional, we
deduce that the lifted specialization $\left((\widetilde{SameFamily_{0}}),(\widetilde{AbDem_{0}}),(\tilde{r}_{0}),(\tilde{f}_{0}),\tilde{w}_{0},\tilde{p}_{0},a\right)$
factors through the groups $PS\left(SameFamily,AbDem,r,f,w,p,a\right)$. 
\end{proof}
Finally, for every group $PS\left(SameFamily,AbDem,r,f,w,p,a\right)$,
we associate its standard $p$-graded MR diagram (note that $w$ is
dropped from the parameter set), and consider all the resolutions
\\
$GRes\left(SameFamily,AbDem,r,f,w,p,a\right)$ in these diagrams,
through which some specialization from the collection $\mathfrak{J}$
factors. For each such resolution, we apply the method presented in
\secref{An-Approximation-to-Abstract-EAE}, and this implies the desired. 
\begin{thm}
\label{thm:60} Let $\zeta(p)$ be an $EAE$ formula

\[
\zeta(p):\qquad\exists w\;\forall y\;\exists x\quad\Sigma(x,y,w,p,a)=1\,\wedge\,\Psi(x,y,w,p,a)\neq1\,.
\]

There exists a formula $\varphi(p)$ in the Boolean algebra of $AE$-formulas,
so that: 
\begin{enumerate}
\item Let $EAE(p)$ be the set of values $p=p_{0}\in F_{k}$ defined by
the formula $\zeta(p)$. The set $EAE(p)$ is defined by the formula
$\varphi(p)$. 
\item Let $\Gamma$ be a group of level $l$ in the model, and let $EAE_{\Gamma}(p)$
be the set of values $p=p_{0}\in\Gamma$ defined by the formula $\zeta(p)$.
If $l$ is large enough, then the set $EAE_{\Gamma}(p)$ is defined
by the formula $\varphi(p)$. 
\end{enumerate}
\end{thm}

\begin{thm}
\label{thm:61} Let $\zeta V(p)$ be an $EAE$ formula

\begin{align*}
\zeta V(p) & :\\
 & \quad\exists w\;\forall y\;\exists x\quad\left(\Sigma_{1}(x,y,w,p,a)=1\,\wedge\,\Psi_{1}(x,y,w,p,a)\neq1\right)\,\vee\,...\,\vee\,\left(\Sigma_{r}(x,y,w,p,a)=1\,\wedge\,\Psi_{r}(x,y,w,p,a)\neq1\right)\,.
\end{align*}

There exists a formula $\varphi(p)$ in the Boolean algebra of $AE$-formulas,
so that: 
\begin{enumerate}
\item Let $EAEV(p)$ be the set of values $p=p_{0}\in F_{k}$ defined by
the formula $\zeta V(p)$. The set $EAEV(p)$ is defined by the formula
$\varphi(p)$. 
\item Let $\Gamma$ be a group of level $l$ in the model, and let $EAEV_{\Gamma}(p)$
be the set of values $p=p_{0}\in\Gamma$ defined by the formula $\zeta V(p)$.
If $l$ is large enough, then the set $EAEV_{\Gamma}(p)$ is defined
by the formula $\varphi(p)$. 
\end{enumerate}
\end{thm}

\begin{proof}
For proving \thmref{60}, we have constructed finitely many proof
systems $PS$ for the formula $\zeta(p)$ that encode validPS statements
and $\Gamma$-validPS statements for all the values $p=p_{0}$ for
which the sentence $\zeta(p_{0})$ is a truth sentence over the corresponding
group.

For constructing a sentence $\varphi(p)$ that defines the sets $EAEV(p)$
and $EAEV_{\Gamma}(p)$, we construct finitely many proof systems
for the formula $\zeta V(p)$, in a similar way to the one explained
in this section.

The only change that we should perform in the construction for this
case, is that instead of constructing (during the corresponding steps)
a graded formal MR diagram of the system $\Sigma(x,y,w,p,a)$ for
each of the resolutions of the graded MR diagram constructed in the
previous step, we construct, for each $j=1,...,r$, the graded formal
MR diagram of the system $\Sigma_{j}(x,y,w,p,a)$ for each of the
resolutions of the graded MR diagram constructed in the previous step.
Then, the graded MR diagrams in the next step, corresponds to the
specializations for which there is no formal solution that verifies
the corresponding set of inequalities $\Psi_{j}(x,y,w,p,a)\neq1$. 
\end{proof}

\section{Truth Sentences of Minimal Rank over Random Groups }\label{sec:Truth-MR-Sentences-over-Random-Groups}
\begin{lem}
\label{lem:62} Let $V(p)$ be a formula in the Boolean algebra of
$AE$-formulas.

Then, there exists a formula $V'(p)$ which is a finite union of $EA$-formulas,
$AE$-formulas, and intersections of an $EA$-formula and an $AE$-formula,
so that $V(p)$ and $V'(p)$ are tautologically equivalent, i.e.,
for every group $\Gamma$, the formulas $V(p)$ and $V'(p)$ define
the same set over $\Gamma$. 
\end{lem}

\begin{proof}
A finite union or a finite intersection, of $AE$-formulas is tautologically
an $AE$-formula. Similarly, a finite union or a finite intersection,
of $EA$-formulas is tautologically an $EA$-formula. 
\end{proof}
\begin{lem}
\label{lem:63} The projection of an $EA$-formula is tautologically
(over any group) an $EA$-formula. The projection of the intersection
of an $EA$-formula and an $AE$-formula, is tautologically a projection
of an $AE$-formula. 
\end{lem}

\begin{proof}
For the first part of the statement, let $EA(w,p)$ be a given $EA$-formula:

\begin{align*}
EA(w,p) & :\\
 & \quad\exists y\;\forall x\;\left(\Sigma_{1}(x,y,w,p,a)=1\,\wedge\,\Psi_{1}(x,y,w,p,a)\neq1\right)\,\vee\,...\,\vee\,\left(\Sigma_{r}(x,y,w,p,a)=1\,\wedge\,\Psi_{r}(x,y,w,p,a)\neq1\right)\,.
\end{align*}

The projection of $EA(w,p)$ is by definition the following formula
\begin{align*}
\pi EA(p) & :\\
 & \quad\exists w\;\exists y\;\forall x\;\left(\Sigma_{1}(x,y,w,p,a)=1\,\wedge\,\Psi_{1}(x,y,w,p,a)\neq1\right)\,\vee\,...\,\vee\,\left(\Sigma_{r}(x,y,w,p,a)=1\,\wedge\,\Psi_{r}(x,y,w,p,a)\neq1\right)\,,
\end{align*}
which is an $EA$-formula.

Now let $V(w,p)$ be the intersection of an an $EA$-formula and an
$AE$-formula. Then, $V(w,p)$ has the form:

\begin{align*}
V(w,p) & :\\
 & \quad\left(\forall y\;\exists x\;\left(\Sigma_{1}(x,y,w,p,a)=1\,\wedge\,\Psi_{1}(x,y,w,p,a)\neq1\right)\,\vee\,...\,\vee\,\left(\Sigma_{r}(x,y,w,p,a)=1\,\wedge\,\Psi_{r}(x,y,w,p,a)\neq1\right)\right)\wedge\\
 & \wedge\left(\exists t\;\forall u\;\left(\Sigma_{1}'(u,t,w,p,a)=1\,\wedge\,\Psi_{1}'(u,t,w,p,a)\neq1\right)\,\vee\,...\,\vee\,\left(\Sigma_{r'}'(u,t,w,p,a)=1\,\wedge\,\Psi_{r'}'(u,t,w,p,a)\neq1\right)\right)\,.
\end{align*}
Hence, the projection of $V(w,p)$ is given by:

\begin{align*}
\pi V(p) & :\\
 & \quad\exists w\;\left(\forall y\;\exists x\;\left(\Sigma_{1}(x,y,w,p,a)=1\,\wedge\,\Psi_{1}(x,y,w,p,a)\neq1\right)\,\vee\,...\,\vee\,\left(\Sigma_{r}(x,y,w,p,a)=1\,\wedge\,\Psi_{r}(x,y,w,p,a)\neq1\right)\right)\wedge\\
 & \quad\quad\wedge\left(\exists t\;\forall u\;\left(\Sigma_{1}'(u,t,w,p,a)=1\,\wedge\,\Psi_{1}'(u,t,w,p,a)\neq1\right)\,\vee\,...\,\vee\,\left(\Sigma_{r'}'(u,t,w,p,a)=1\,\wedge\,\Psi_{r'}'(u,t,w,p,a)\neq1\right)\right)\,,
\end{align*}
which is tautologically equivalent to:

\begin{align*}
 & \exists w\;\exists t\;\forall y\;\forall u\;\exists x\;\left(\left(\Sigma_{1}(x,y,w,p,a)=1\,\wedge\,\Psi_{1}(x,y,w,p,a)\neq1\right)\,\vee\,...\,\vee\,\left(\Sigma_{r}(x,y,w,p,a)=1\,\wedge\,\Psi_{r}(x,y,w,p,a)\neq1\right)\right)\wedge\\
 & \qquad\qquad\qquad\wedge\left(\left(\Sigma_{1}'(u,t,w,p,a)=1\,\wedge\,\Psi_{1}'(u,t,w,p,a)\neq1\right)\,\vee\,...\,\vee\,\left(\Sigma_{r'}'(u,t,w,p,a)=1\,\wedge\,\Psi_{r'}'(u,t,w,p,a)\neq1\right)\right)\,.
\end{align*}
Hence, iterative replacements of intersections of the form $A\wedge(B\vee C)$,
by unions of the form $(A\wedge B)\vee(A\wedge C)$, in the formula

\begin{align*}
 & \ \left(\left(\Sigma_{1}(x,y,w,p,a)=1\,\wedge\,\Psi_{1}(x,y,w,p,a)\neq1\right)\,\vee\,...\,\vee\,\left(\Sigma_{r}(x,y,w,p,a)=1\,\wedge\,\Psi_{r}(x,y,w,p,a)\neq1\right)\right)\wedge\\
 & \wedge\left(\left(\Sigma_{1}'(u,t,w,p,a)=1\,\wedge\,\Psi_{1}'(u,t,w,p,a)\neq1\right)\,\vee\,...\,\vee\,\left(\Sigma_{r'}'(u,t,w,p,a)=1\,\wedge\,\Psi_{r'}'(u,t,w,p,a)\neq1\right)\right)\,,
\end{align*}
imply the required. 
\end{proof}
\begin{lem}
\label{lem:64} Let $V(w,p)$ be a formula in the Boolean algebra
of $AE$-formulas.

Then, there exists finitely many $AE$-formulas $V_{1}(w,p),...,V_{s}(w,p)$,
and an $EA$-formula $H(p)$, so that the projection of $V(w,p)$
is tautologically equivalent to the union of $H(p)$ with the projections
of the $AE$-formulas $V_{1}(w,p),...,V_{s}(w,p)$: 
\[
\pi V(p):\qquad H(p)\vee\pi V_{1}(p)\vee...\vee\pi V_{s}(p)\,.
\]
\end{lem}

\begin{proof}
According to \lemref{62}, the formula $V(w,p)$ is a finite union
of $EA$-formulas, $AE$-formulas, and intersections of an $EA$-formula
and an $AE$-formula.

Hence, the projection of the formula $V(w,p)$ is tautologically the
union of the projections of these $EA$-formulas, $AE$-formulas,
and the intersections of an $EA$-formula and an $AE$-formula.

According to \lemref{63}, and since the union of $EA$-formulas is
tautologically an $EA$-formula, we obtain the required. 
\end{proof}
\begin{defn}
\label{def:65} Let $V(w,p)$ be a formula in the Boolean algebra
of $AE$-formulas. Let $V_{1}(w,p),...,V_{s}(w,p)$, and $H(p)$ be
as in \ref{lem:64}. We call the formula $V(w,p)$ a \emph{minimal
rank formula}, if all the graded limit groups along the construction
of the proof systems for each of the $EAE$-formulas $\pi V_{1}(p),...,\pi V_{s}(p)$,
are of minimal rank (see \secref{Proof-Systems}). Note that the construction
of proof systems depends only on the fixed free group $F_{k}$ and
the given $AE$-formula. 
\end{defn}

\begin{thm}
\label{thm:66} Let $V(w,p)$ be a minimal rank formula in the Boolean
algebra of $AE$-formulas. Let $\pi V(p)$ be the projection of $V(w,p)$.
Then, there exists a formula $R(p)$ in the Boolean algebra of $AE$-formulas,
so that: 
\begin{enumerate}
\item The formulas $\pi V(p)$ and $R(p)$ define the same set over $F_{k}$. 
\item Let $\Gamma$ be a random group. With overwhelming probability, the
formulas $\pi V(p)$ and $R(p)$ define the same set over $\Gamma$. 
\end{enumerate}
\end{thm}

\begin{proof}
Let $V_{1}(w,p),...,V_{s}(w,p)$, and $H(p)$ be as in \lemref{64},
and write 
\[
\pi V(p):\qquad H(p)\vee\pi V_{1}(p)\vee...\vee\pi V_{s}(p)\,.
\]
According to \thmref{61}, for every $i=1,...,s$, there exists a
formula $\varphi_{i}(p)$ in the Boolean algebra of $AE$-formulas,
so that: 
\begin{enumerate}
\item The formulas $\pi V_{i}(p)$ and $\varphi_{i}(p)$ define the same
set over $F_{k}$. 
\item The formulas $\pi V_{i}(p)$ and $\varphi_{i}(p)$ define the same
set over $\Gamma$ (with overwhelming probability). 
\end{enumerate}
Hence, the formula 
\[
R(p):\qquad H(p)\vee\varphi_{1}(p)\vee...\vee\varphi_{s}(p)\,,
\]
is as required. 
\end{proof}
\begin{defn}
\label{def:67} Let $\psi$ be a first order sentence. 
\begin{enumerate}
\item If $\psi$ belongs to the Boolean algebra of $AE$-sentences, then,
we say that $\psi$ is a \emph{minimal rank sentence}, if $\psi$,
as a formula with empty set of parameters, is a minimal rank formula
(\defref{65}). 
\item If $\psi$ is an $EAE$-sentence, then $\psi$ is called a minimal
rank sentence if it is the projection of a minimal rank formula. 
\item If $\psi$ is an $AEA$-sentence, then $\psi$ is called a minimal
rank sentence if $\neg\psi$ is a minimal rank sentence. 
\item Assume that the sentence $\psi$ has the (prenex) normal form 
\begin{align*}
 & \forall y_{1}\;\exists x_{1}\;...\;\exists x_{n-1}\;\forall y_{n}\;\exists x_{n}\;\\
 & \left(\Sigma_{1}(x_{i},y_{i},w,a)=1\,\wedge\,\Psi_{1}(x_{i},y_{i},w,a)\neq1\right)\,\vee\,...\,\vee\,\left(\Sigma_{r}(x_{i},y_{i},w,a)=1\,\wedge\,\Psi_{r}(x_{i},y_{i},w,a)\neq1\right)\,.
\end{align*}
We denote $y=y_{n}$, $x=x_{n}$, $w=x_{n-1}$, and we denote by $p$
the union of the variables sets $y_{1},x_{1},...,y_{n-1}$. We consider
the $AE$-formula 
\begin{align*}
V(w,p) & :\\
 & \quad\forall y\;\exists x\;\left(\Sigma_{1}(x,y,w,p,a)=1\,\wedge\,\Psi_{1}(x,y,w,p,a)\neq1\right)\,\vee\,...\,\vee\,\left(\Sigma_{r}(x,y,w,p,a)=1\,\wedge\,\Psi_{r}(x,y,w,p,a)\neq1\right)\,,
\end{align*}
and we assume that $V(w,p)$ is a minimal rank formula (\defref{65}).
Note that 
\[
\psi=\forall y_{1}\;\exists x_{1}\;...\;\forall y_{n-1}\;\pi V(p)\,.
\]
Let $R(p)$ be the formula obtained from \thmref{66} corresponding
to $V(w,p)$. In light of \thmref{66}, the sentence $\psi$ is equivalent
over $F_{k}$ and over random groups (with overwhelming probability)
to the sentence 
\[
\psi':\qquad\forall y_{1}\;\exists x_{1}\;...\;\forall y_{n-1}\;R(p)\,.
\]
Since $R(p)$ belongs to the Boolean algebra of $AE$-formulas, $\psi'$
is tautologically equivalent to a sentence of the form 
\[
\psi'=\forall y_{1}\;\exists x_{1}\;...\;\forall y_{n-1}\;\forall u\;\exists t\;\forall z\;R'(p,u,t,z)\,,
\]
(note the consecutive universal quantifiers $\forall y_{n-1}\;\forall u$).
The sentence $\psi$ is called a minimal rank sentence if $\psi'$
is a minimal rank sentence (defined by induction). 
\item Assume that the sentence $\psi$ has the normal form 
\begin{align*}
 & \exists y_{1}\;\forall x_{1}\;...\;\forall x_{n-1}\;\exists y_{n}\;\forall x_{n}\;\\
 & \left(\Sigma_{1}(x_{i},y_{i},w,a)=1\,\wedge\,\Psi_{1}(x_{i},y_{i},w,a)\neq1\right)\,\vee\,...\,\vee\,\left(\Sigma_{r}(x_{i},y_{i},w,a)=1\,\wedge\,\Psi_{r}(x_{i},y_{i},w,a)\neq1\right)\,.
\end{align*}
We denote $y=y_{n}$, $x=x_{n}$, $w=x_{n-1}$, and we denote by $p$
the union of the variables sets $y_{1},x_{1},...,y_{n-1}$. We consider
the $AE$-formula 
\begin{align*}
V(w,p) & :\\
 & \ \neg\left(\exists y\;\forall x\;\left(\Sigma_{1}(x,y,w,p,a)=1\,\wedge\,\Psi_{1}(x,y,w,p,a)\neq1\right)\,\vee\,...\,\vee\,\left(\Sigma_{r}(x,y,w,p,a)=1\,\wedge\,\Psi_{r}(x,y,w,p,a)\neq1\right)\right)\,,
\end{align*}
and we assume that $V(w,p)$ is a minimal rank formula (\defref{65}).
Note that 
\[
\psi=\exists y_{1}\;\forall x_{1}\;...\;\exists y_{n-1}\;\neg\pi V(p)\,.
\]
In a similar way to the previous point of the definition, using $V(w,p)$
and \thmref{66}, we construct a sentence $\psi'$ that is equivalent
to $\psi$ over $F_{k}$ and over random groups, so that $\psi'$
is of the form 
\[
\psi'=\exists y_{1}\;\forall x_{1}\;...\;\exists y_{n-1}\;\neg\left(\forall u\;\exists t\;\forall z\;R'(p,u,t,z)\right)\,.
\]
And we call the sentence $\psi$ a minimal rank sentence if $\psi'$
is a minimal rank sentence (defined by induction). 
\end{enumerate}
\end{defn}

\begin{thm}
\label{thm:68} Let $\psi$ be a first order sentence of minimal rank.
Then, there exists a minimal rank sentence $R(\psi)$ that belongs
to the Boolean algebra of $AE$-sentences, so that: 
\begin{enumerate}
\item The sentences $\psi$ and $R(\psi)$ are equivalent over $F_{k}$. 
\item Let $\Gamma$ be a random group. With overwhelming probability, the
sentences $\psi$ and $R(\psi)$ are equivalent over $\Gamma$. 
\end{enumerate}
\end{thm}

\begin{thm}
\label{thm:69} Let $\psi$ be a minimal rank sentence in the Boolean
algebra of $AE$-sentences. Let $\Gamma$ be a random group. Then,
$\psi$ is a truth sentence over $F_{k}$ if and only if $\psi$ is
a truth sentence over $\Gamma$. 
\end{thm}

\begin{proof}
If $\psi$ is an $EA$-sentence, and $\psi$ is true over $F_{k}$,
then, according to the first paper, $\psi$ is a truth sentence over
$\Gamma$. Hence, it suffices to prove that if $\psi$ is an $AE$-sentence
that is true over $F_{k}$, then $\psi$ is true over $\Gamma$. Indeed,
we construct the proof system for the sentence $\psi$ (as a $EAE$-formula
with empty set of parameters) over $F_{k}$. According to the construction
of proof systems in \secref{Proof-Systems}, the proof system of $\psi$
constructed over $F_{k}$, validates the truthiness of $\psi$ over
$\Gamma$ (with overwhelming probability). 
\end{proof}
\begin{thm}
\label{thm:70} Let $\psi$ be a first order sentence of minimal rank.
Let $\Gamma$ be a random group. Then, $\psi$ is a truth sentence
over $F_{k}$ if and only if $\psi$ is a truth sentence over $\Gamma$. 
\end{thm}

\begin{proof}
According to \thmref{68}, there exists a sentence $\psi'$ in the
Boolean algebra of $AE$-sentences that is equivalent to $\psi$ over
$F_{k}$ and over $\Gamma$. According to \thmref{69}, $\psi'$ is
true over $F_{k}$ if and only if it is true over $\Gamma$. Hence,
$\psi$ is true over $F_{k}$ if and only if $\psi$ is true over
$\Gamma$. 
\end{proof}

\end{document}